\documentclass[10pt, a4paper, leqno]{amsart}

\pdfoutput=1

\usepackage[T1]{fontenc}

\usepackage{amsfonts}
\usepackage{amssymb}
\usepackage{enumerate}
\usepackage{paralist}

\usepackage{graphicx}
\usepackage{subfigure}

\usepackage[abbrev, backrefs]{amsrefs}

\newtheorem{theorem}{Theorem}
\newtheorem*{ThmA}{Theorem}
\newtheorem{proposition}{Proposition}[section]
\newtheorem{lemma}[proposition]{Lemma}

\theoremstyle{definition}
\newtheorem{remark}{Remark}[section]
\newtheorem{definition}{Definition}[section]
\numberwithin{equation}{section}

\DeclareMathOperator{\supp}{supp}

\newcommand{\N}{{\mathbb N}}
\newcommand{\R}{{\mathbb R}}

\newcommand{\abs}[1]{\lvert #1 \rvert}

\newcommand{\Bigabs}[1]{\Bigl\lvert #1 \Bigr\rvert}

\title[Nonexistence of supersolutions of Choquard equations]
{Nonexistence and optimal decay of supersolutions to Choquard equations in exterior domains}
\author{Vitaly Moroz}
\address{Swansea University\\ Mathematics Department\\ Singleton Park\\
Swansea\\ SA2~8PP\\ Wales, United Kingdom}	
\email{V.Moroz@swansea.ac.uk}

\author{Jean Van Schaftingen}
\address{Universit\'e Catholique de Louvain\\ Institut de Recherche en Math\'ematique et Phy\-si\-que\\ Chemin du Cyclotron 2 bte L7.01.01\\ 1348 Louvain-la-Neuve \\ Belgium}
\email{Jean.VanSchaftingen@uclouvain.be}


\allowdisplaybreaks


\date{\today}

\begin{document}

\begin{abstract}
We consider a semilinear elliptic problem with a nonlinear term which is the product of a power and the Riesz potential of a power.
This family of equations includes the Choquard or nonlinear Schr\"odinger--Newton equation.
We show that for some values of the parameters the equation does not have
nontrivial nonnegative supersolutions in exterior domains.
The same techniques yield optimal decay rates when supersolutions exists.
\end{abstract}

\maketitle

\begin{small}
\tableofcontents
\end{small}

\section{Introduction.}

We study the nonlocal nonlinear equation
\begin{equation}\label{eqChoquardV}
\tag{$\mathcal{C}$}
  -\Delta u + V u = (I_\alpha \ast u^p)u^q\quad\text{in \(\Omega\)},
\end{equation}
where \(\Omega \subseteq \R^N\) is a domain, for given exponents \(p > 0\), \(q \in \R\) and potential \(V : \Omega \to \R\).
Here,
\(I_\alpha : \R^N\setminus\{0\}\to \R\) denotes the Riesz potential, which is defined for \(0<\alpha<N\) and \(x \in \R^N \setminus \{0\}\) by
\[
  I_\alpha(x)=\frac{A_\alpha}{\abs{x}^{N - \alpha}},
\]
where
\[
 A_\alpha = \frac{\Gamma(\tfrac{N - \alpha}{2})}{\Gamma(\tfrac{\alpha}{2})\pi^{N/2}2^{\alpha}}
\]
and \(\Gamma\) is the Gamma function \cite{Riesz}*{p.19}.

The nonlocal equation \eqref{eqChoquardV} has several physical origins.
When \(\Omega = \R^3\), \(\alpha = 2\), \(p = 2\), \(q = 1\) and \(V\) is constant,
the equation writes as
\begin{equation}\label{Choquard}
-\Delta u + V u=(I_2\ast u^2)u\quad\text{in \(\R^3\)},
\end{equation}
or, equivalently,
\begin{equation}
\label{ChoquardS}
\left\{
\begin{aligned}
-\Delta u +V u &=Wu & & \text{in \(\R^3\)},\\
 - \Delta W&=u^2 & & \text{in \(\R^3\)}.
\end{aligned}
\right.
\end{equation}
If \(u\) solves \eqref{Choquard} then the function \(\psi\) defined by \(\psi(t,x)=e^{it\lambda}u(x)\) is a solitary wave of the focusing Hartree equation
\[
 i \psi_t = - \Delta \psi - (I_2 \ast \psi^2)\psi.
\]
Equation \eqref{Choquard} first appeared at least as early as in 1954, in a work by S.\thinspace I.\thinspace Pekar \cite{Pekar} describing the quantum mechanics of a polaron at rest (see discussion in \cite{Lieb-polaron}).
In 1976 P.\thinspace Choquard used \eqref{Choquard} to describe an electron trapped in its own hole,
in a certain approximation to Hartree--Fock theory of one component plasma, see \cite{Lieb-77}.
In 1996 R.\thinspace Penrose proposed \eqref{ChoquardS} as a model of self-gravitating matter,
in a programme in which quantum state reduction is understood as a gravitational phenomenon \cite{Penrose}.

When the potential \(V\) is constant,
E.\thinspace H.\thinspace Lieb \cite{Lieb-77} established existence and uniqueness of a positive radial ground state solution of \eqref{Choquard} using variational methods.
P.-L.\thinspace Lions extended Lieb's results by replacing \(I_2\) with a wider class of convolution kernels and  established existence of infinitely many radial (changing-sign) solutions with increasing energy \citelist{\cite{Lions-80}\cite{Lions-1-1}*{Chapter III}}.
G.\thinspace P.\thinspace Menzala established further existence and nonexistence results for equations
of type \eqref{Choquard} with a variable potential \(V(x)\) and general convolution kernels  \citelist{\cite{Menzala-80}\cite{Menzala-83}}.

I.\thinspace Moroz, R.\thinspace Penrose and P.\thinspace Tod have studied independently the existence, uniqueness and decay properties of the positive ground state and changing-sign radial solutions
of \eqref{Choquard} numerically and via ODE methods
\citelist{\cite{Penrose-1}\cite{Penrose-2}} (see also \cite{Lenzmann}*{Appendix A}).
An ODE based proof of the existence and uniqueness
of the radial ground state of \eqref{eqChoquardV} with \(V = 1\), \(\alpha=2\), \(p=2\) and \(q=1\) in dimensions \(N\ge 1\) was recently obtained by P.\thinspace Choquard, J.\thinspace Stubbe and
M.\thinspace Vuffray \cite{Choquard-08}.
L.\thinspace Ma and L.\thinspace Zhao  have studied the symmetry of positive radial ground state of \eqref{eqChoquardV} with constant \(V\) in higher dimensions using the moving--plane method \cite{Ma-Zhao}.

J.\thinspace Wei and M.\thinspace Winter \cite{Wei-Winter} have considered the singular perturbation problem
\begin{equation}\label{SN-eps}
-\varepsilon^2\Delta u + V u = (I_2 \ast u^2)u\quad\text{in \(\R^3\)}.
\end{equation}
Assuming that \(\inf V>0\) they have proved that in the semi-classical limit \(\varepsilon\to 0\)  there exist multibump positive solutions of \eqref{SN-eps} which concentrate as \(\varepsilon\to 0\)
to critical points of the potential \(V\).
S.\thinspace Secchi \cite{Secchi} studied the existence of positive solutions of \eqref{SN-eps} which concentrate to critical points of \(V\) under the assumption that \(V\) does not decay too fast at infinity.

In the context of local semilinear equations of the type
\begin{equation}\label{C-local}
-\Delta u+Vu=W u^q
\end{equation}
it is well known that the existence of positive solutions and supersolutions in \(\R^N\) or in exterior domains of \(\R^N\)
requires a careful apriori balance between the value of the nonlinear exponent \(q\) and the decay rate at infinity of the potentials \(V\) and \(W\) \cites{Gidas,Bidaut-Veron}.
Such results are often called nonlinear Liouville theorems, see \cite{Quittner-Souplet}*{Section 1.8} and further references therein.
\smallskip

The main purpose of this work is to establish sharp Liouville type nonexistence results for positive supersolutions of
nonlocal Choquard equation \eqref{eqChoquardV} in exterior domains.
For instance, for the classical Choquard equation \eqref{Choquard}  we obtain the following result as a particular case of Theorem~\ref{Thm-exp+}.

\begin{ThmA}
Let \(\lambda > 0\). The problem
\begin{equation}\label{Choquard-gamma}
\tag{\protect{$\mathcal{C}_{\gamma,\lambda}$}}-\Delta u (x) + \frac{\lambda^2}{\abs{x}^\gamma} u (x) \ge (I_2 \ast u^2) (x) \, u(x)\quad
\text{in \(\R^3\setminus\Bar{B}_1\)}
\end{equation}
admits a positive supersolution if and only if \(\lambda>0\) and \(-\infty < \gamma \le 1\).\\
Moreover, if \(\gamma < 1\)  and if \(u\ge 0\) is a nontrivial nonnegative supersolution of \eqref{Choquard-gamma},
then there exists \(m > 0\) such that ,
\[
\liminf_{\abs{x}\to \infty} u(x)\abs{x}^{1 - \frac{\gamma}{4}}
\exp \int_{m^{\frac{1}{1 - \gamma}}}^{\abs{x}} \sqrt{\frac{\lambda^2}{s^\gamma}-\frac{m}{s}\;} \,ds >0;
\]
if \(\gamma = 1\)  and if \(u\ge 0\) is a nontrivial nonnegative supersolution of \eqref{Choquard-gamma},
then there exists \(m \in (0, \lambda)\) such that,
\[
\liminf_{\abs{x}\to \infty} u(x)\abs{x}^\frac{3}{4}
\exp \bigl(2 (\lambda - m) \sqrt{\abs{x}}\bigr) > 0.
\]
The above lower bounds are optimal.
\end{ThmA}

In particular, this gives a negative answer to a question posed by S.\thinspace Secchi \cite{Secchi}*{p.~3855}.

When  \(\gamma < 1\), the integral in the asymptotics is an incomplete Beta function
\[
 \int_{m^{\frac{1}{1 - \gamma}}}^{\abs{x}}\frac{\lambda}{s^\frac{\gamma}{2}} \sqrt{1-\frac{\rho^{1 - \gamma}}{s^{1 - \gamma}}} \,ds
= \tfrac{1}{1 - \gamma} B_{1-\frac{m}{\abs{x}^{1 - \gamma}}}
    \bigl(\tfrac{3}{2}, -\tfrac{2 + \gamma}{2 (1 - \gamma)}\bigr),
\]
which can also be represented in terms of the hypergeometric function.
The asymptotic can be made explicit by taking the Taylor expansion of the square root:
\begin{compactitem}[---]
\item if \(\gamma < 0\), then
\[
 \liminf_{\abs{x}\to \infty} u(x)\abs{x}^{1 - \frac{\gamma}{4}}
\exp \bigl(\tfrac{2 \lambda}{2 - \gamma} \abs{x}^\frac{2 - \gamma}{2}\bigr) > 0,
\]
\item if \(\gamma = 0\), then there exists \(\rho > 0\) such that
\[
  \liminf_{\abs{x}\to \infty}\textstyle{u (x)\abs{x}^{1-\frac{\rho}{2 \lambda}}
\exp\big( \lambda\abs{x}\big)>0},
\]
\item if \(0 < \gamma < \frac{2}{3}\), then there exists \(\rho >0\) such that
\[
\liminf_{\abs{x}\to \infty}\textstyle{u (x)\abs{x}^{1 - \frac{\gamma}{4}}
\exp\big(\frac{2\lambda}{2-\gamma}\abs{x}^\frac{2-\gamma}{2}-
\frac{m}{\lambda \gamma} \abs{x}^{\frac{\gamma}{2}} \big)>0};
\]
\end{compactitem}
for \(\frac{2}{3} \le \gamma < 1\), see Remark~\ref{r-exp-lambda}.

In particular if \(u\) is the unique radial positive ground state solution of
\begin{equation*}
-\Delta u + u = (I_2 \ast u^2)u \quad \text{in } \R^3
\end{equation*}
(see for example \citelist{\cite{Lieb-77}\cite{Choquard-08}\cite{Lenzmann}\cite{Ma-Zhao}} for proofs of existence and uniqueness),
then there exists \(\rho >0\) such that
\begin{equation*}
\liminf_{\abs{x}\to \infty} u(x)\abs{x}^{1-\frac{\rho}{2}}\exp(\abs{x})>0.
\end{equation*}
Thus the ground state \(u\) decays slower than the fundamental solution of \(-\Delta + 1\) in \(\R^3\).
(Note that in \cite{Wei-Winter}*{Theorem I.1 (1.7)}, the correction to the exponent $\rho$ seems missing.)
One has in fact
\[
 \liminf_{\abs{x}\to \infty} u(x)\abs{x}^{1-\frac{\rho}{2}}\exp(\abs{x}) \in (0, \infty),
\]
where \(\rho > 0\) is characterized by the groundstate energy \cite{MVS-ground}.

\medbreak

In the study of the general Choquard equation \eqref{eqChoquardV}
with its multiple parameters, we classify the cases with respect to the decay rate of the potential \(V\) and with respect to the type of the nonlinearity.

We distinguishing between four different types of potentials:
\begin{compactenum}[(i)]
\item \emph{unperturbed} Laplacian
\begin{align*}
V (x) & = 0,
\intertext{\item \emph{fast decay potentials}}
V(x)& =\frac{\lambda}{\abs{x}^\gamma}, & & \text{with \(\lambda\in\R\) and \(\gamma>2\)},
\intertext{\item \emph{Hardy potentials}}
V(x)&=\frac{\nu^2-\big(\tfrac{N - 2}{2}\big)^2}{\abs{x}^2} & & \text{with \(\nu>0\)},
\intertext{\item \emph{slow decay potentials}}
V(x)&=\frac{\lambda^2}{\abs{x}^\gamma} & & \text{with \(\lambda >0\) and \(\gamma < 2\)}.
\end{align*}
\end{compactenum}
The classification of potentials is motivated by the decay rate of the fundamental solution of the linear operator \(-\Delta+V\):
for fast decay and Hardy potentials
it decays polynomially, while for slow decay potentials
it has exponential decay.
This difference is essential for our considerations.

The above radial potentials could be replaced by wider classes of nonradial potentials
with equivalent decay rate of the fundamental solution of \(-\Delta+V\).
We restrict ourself to the explicit power-like potentials in order to simplify the exposition.

\medbreak

The other distinction is made with respect to the types of the nonlinearity.
In the context of the local equations \eqref{C-local} one usually distinguishes between the superlinear case \(q>1\) and sublinear case \(q<1\). The corresponding classification for Choquard's equation \eqref{eqChoquardV} is more complex. According to the order of homogeneity of its right-hand side, we
distinguish
\begin{compactenum}[(i)]
 \item the \emph{superlinear case} \(q > 1\),
 \item the \emph{locally sublinear case} \(p + q > 1\) and \(q < 1\),
 \item the \emph{globally sublinear case} \(p + q < 1\).
\end{compactenum}
The superlinear and the globally sublinear cases correspond to the superlinear and the sublinear cases for the local equation \eqref{C-local}; the locally sublinear case is a transitional region which has no analogue in the local equation.
The transitional \emph{locally linear case} (\(q = 1\)) and \emph{globally linear case} (\(p + q = 1\)) require particularly careful consideration.

\medbreak

The above classifications of potentials and nonlinearities produces a large variety of different cases
in our analysis of \eqref{eqChoquardV} requiring specific consideration.
Nevertheless, for all classes of potentials and types nonlinearities we use the same unified approach
which is based on two main tools:
\begin{compactitem}[---]
\item
lower and upper Phragmen--Lindel\"of type estimates on the decay at infinity of positive supersolutions
of the linear operator \(-\Delta+V\);
see Proposition~\ref{propositionLinearLowerBound} and Proposition~\ref{prop-exp-lambda};
\item
a nonlocal nonlinear extension of the Agmon--Allegretto--Piepenbrink positivity principle \cite{Agmon}*{Theorem 3.1}, which relates the existence of a positive supersolution to \eqref{eqChoquardV} to an integral inequality; see Proposition~\ref{propositionGroundState}.
\end{compactitem}
\noindent
Combining linear estimates for \(-\Delta+V\) with the positivity principle
either leads to a contradiction which implies nonexistence of positive supersolutions of \eqref{eqChoquardV},
or provides a bound on the admissible rate of the decay of a solution.
Explicit construction of appropriate supersolutions shows that these bounds are optimal.

We point out that our nonexistence and decay results are optimal for \emph{supersolutions}.
We have observed that these bounds give nonetheless good insight in the decay of minimal energy \emph{solutions} of \eqref{eqChoquardV} on \(\R^N\) in the variational case \(q = p - 1\) \cite{MVS-ground}.

\section{Statement of the results.}

\subsection{Notion of a supersolution.}
Let \(N\ge 1\) and \(\Omega \subseteq \R^N\) be an open set and \(V\in L^1_\mathrm{loc}(\R^N)\) be a generic potential.
In order to formulate our results we shall clarify the notion of supersolution to Choquard equation \eqref{eqChoquardV} which we adopt in this work.

\begin{definition}\label{def-distro}
A nonnegative function \(u\in L^1_\mathrm{loc}(\Omega)\) is a \emph{(distributional) supersolution} of \eqref{eqChoquardV}
if
\begin{equation}\label{equLp}
  \int_{\Omega} \frac{u (x)^p}{1 + \abs{x}^{N - \alpha}} \,dx < \infty,
\end{equation}
\(Vu \in L^1_{\mathrm{loc}} (\Omega)\), \((I_\alpha \ast u^p) u^q \in L^1_\mathrm{loc} (\Omega)\),
and for every nonnegative test function \(\varphi \in C^\infty_c(\Omega)\) one has
\begin{equation*}
  - \int_{\Omega} u \Delta \varphi + \int_\Omega V u \varphi \ge \int_{\Omega} (I_\alpha \ast u^p) u^q \varphi.
\end{equation*}
\end{definition}

Here we extend the usual definition of the convolution product by setting, for \(x \in \Omega\),
\begin{equation*}
 (I_\alpha \ast u^p) (x) = \int_{\Omega} I_\alpha (x - y) u (y)^p \,dy.
\end{equation*}
This coincides with the standard convolution product of \(I_\alpha\) with the extension of \(u\)  by \(0\) to \(\R^N\).

If \(\omega \subset \Omega\) is open, then in general \(I_\alpha \ast (u{|_\omega})^p < I_\alpha \ast u^p\). Hence if \(u\) is a supersolution of \eqref{eqChoquardV} in \(\Omega\), then \(u\) is a supersolution of \eqref{eqChoquardV} in \(\omega\) but our notion of supersolution is not local, as \(u\) can be a distributional supersolution of \eqref{eqChoquardV} in the open sets \(\Omega_1 \subset \R^N\) and \(\Omega_2 \subset \R^N\) but not in \(\Omega_1 \cap \Omega_2\).

\subsection{Equation with the unperturbed Laplacian.}
Consider the Choquard equation \eqref{eqChoquardV} with the potential \(V\equiv 0\), that is
\begin{equation}
\label{eqChoquard0}
\tag{$\mathcal{C}_0$}
  -\Delta u = (I_\alpha \ast u^p)u^q.
\end{equation}
It is an easy consequence of \eqref{equLp} and standard lower bounds on superharmonic functions that \eqref{eqChoquard0}
has no positive supersolutions in exterior domains of \(\R^N\) in dimensions \(N=1,2\) (Proposition~\ref{dim12}).
In higher dimensions the existence of nontrivial nonnegative supersolutions of \eqref{eqChoquard0} is more complex.

\begin{theorem}\label{Thm-free}
Let \(N\ge 3\), \(0<\alpha<N\), \(p > 0\), \(q \in\R\) and \(\rho>0\).
Equation \eqref{eqChoquard0} has a nonnegative nontrivial supersolution in \(\R^N \setminus\Bar{B}_\rho\) if and only if the following assumptions hold simultaneously:
\begin{subequations}\label{eqA}
\begin{align}
p & > \frac{\alpha}{N - 2},\label{A-1}\\
p + q & > \frac{N + \alpha}{N - 2},\label{A-2}\\
q & > \frac{\alpha}{N - 2} &  \text{if } \alpha>N-2 \label{A-3},\\
q & \ge 1 &  \text{if } \alpha=N-2 \label{A-3plus},\\
q & > 1-\frac{N - \alpha - 2}{N}p&  \text{if } \alpha<N-2.\label{A-4}
\end{align}
\end{subequations}
Moreover, if \(u\ge 0\) is a nontrivial supersolution of \eqref{eqChoquard0} in \(\R^N \setminus\Bar{B}_{\rho}\) then
\begin{subequations}
\label{eqB}
\begin{align}
&&\liminf_{\abs{x} \to \infty} u (x) \abs{x}^{N - 2}&>0& & \text{if }\textstyle{q > \frac{\alpha}{N - 2},}\label{B-1}\\
& &\exists m>0:\;\qquad\liminf_{\abs{x} \to \infty} u (x) \abs{x}^{N - 2-m}& >0
& & \text{if \(q = \tfrac{\alpha}{N - 2} = 1\),}\label{B-1plus}\\
&&\liminf_{\abs{x} \to \infty} u (x) \abs{x}^{N - 2}  \bigl(\log \abs{x}\bigr)^{-\frac{N - 2}{N - \alpha - 2}}
&>0 & & \text{if \(q =\tfrac{\alpha}{N - 2} < 1\),}\label{B-1plusplus}\\
&&\liminf_{\abs{x} \to \infty} u (x) \abs{x}^{\frac{N - \alpha - 2}{1-q}}&>0
& &\text{if \(q < \tfrac{\alpha}{N - 2} < 1\).}\label{B-2}
\end{align}
\end{subequations}
The above lower bounds are optimal.
\end{theorem}

\begin{figure}
\subfigure[\(\alpha \ge N - 2\)]{\includegraphics{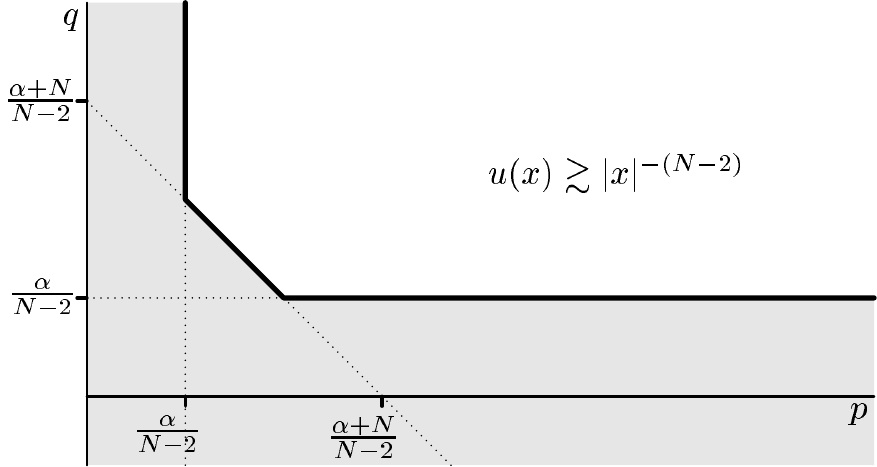}}
\subfigure[\(\alpha < N - 2\)]{\includegraphics{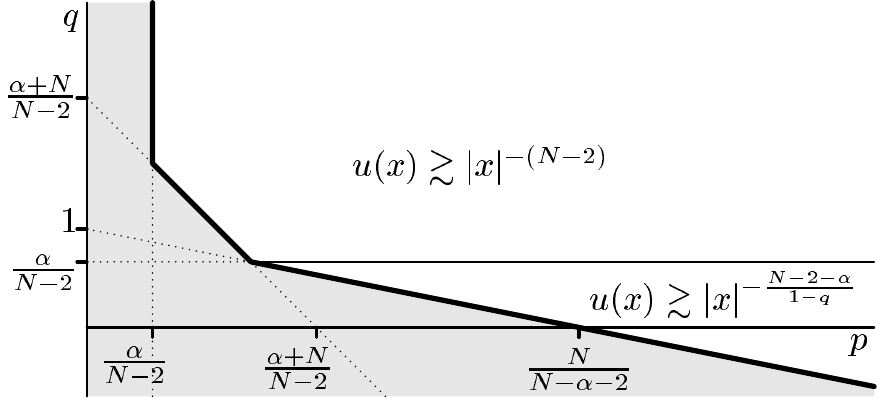}}

\caption{Existence, decay and nonexistence regions for \eqref{eqChoquard0} in the \((p,q)\)--plane}
\end{figure}

The optimality of lower bounds \eqref{eqB} is understood in the sense that
\begin{compactitem}[---]
\item
if \(p > \frac{\alpha}{N - 2}\), \(p + q > \frac{N + \alpha}{N - 2}\) and \(q > \frac{\alpha}{N - 2}\), then there exists a nontrivial supersolution \(u \ge 0\)  such that
\[
  \limsup_{\abs{x} \to \infty} u (x) \abs{x}^{N - 2} < \infty,
\]
\item
if \(p>\frac{N}{N-2}\) and \(q = \frac{\alpha}{N - 2} = 1\),  then for every \(m > 0\) there exists a nontrivial supersolution \(u \ge 0\)  such that
\[
  \limsup_{\abs{x} \to \infty} u (x) \abs{x}^{N - 2 - m} < \infty,
\]
\item
if \(p>\frac{N}{N-2}\) and \(q = \frac{\alpha}{N - 2} = 1\),  then there exists a nontrivial supersolution \(u \ge 0\)  such that
\[
  \limsup_{\abs{x} \to \infty} u (x) \abs{x}^{N - 2} \bigl(\log \abs{x}\bigr)^{-\frac{N - 2}{N - \alpha - 2}} < \infty,
\]
\item
if \(1-\frac{N - \alpha - 2}{N}p < q<\frac{\alpha}{N - 2} < 1\), then there exists a nontrivial supersolution \(u \ge 0\) such that
\[
  \limsup_{\abs{x} \to \infty} u (x) \abs{x}^{\frac{N - \alpha - 2}{1-q}}< \infty.
\]
\end{compactitem}
In view of the bounds \eqref{eqB} in what follows we refer to the region \(\big\{q > \frac{\alpha}{N - 2}\}\)
of the \((p,q)\)--plane as the \emph{Green decay} region,
while we call \(\big\{q < \frac{\alpha}{N - 2}\big\}\) the \emph{sublinear decay} region.

When \(q=p-1\), equation \eqref{eqChoquard0} has a variational structure, with the energy formally defined by
\[
E(u)=\frac{1}{2}\int_{\R^N \setminus B_\rho} \abs{\nabla u}^2-\frac{1}{2p}\int_{\R^N \setminus B_\rho} \bigl(I_\alpha\ast u^p\bigr)u^p.
\]
The existence conditions \eqref{eqA} then transform into
\begin{align*}
\eqref{A-2}'&& p & > \frac{1}{2}\frac{2N - 2+\alpha}{N - 2} & & \text{if } 0<\alpha\le 2,\\
\eqref{A-4}'&& p & >  \frac{2N}{2N - \alpha - 2} & &\text{if } 2<\alpha< N - 2,\\
\eqref{A-3plus}'&& p & \ge 2 & &\text{if } 2<\alpha= N - 2,\\
\eqref{A-3}'&& p & > 1+\frac{\alpha}{N - 2} & &\text{if } \max\{2,N - 2\}<\alpha<N.
\end{align*}

When \(q=0\), equation \eqref{eqChoquard0} is written as \(-\Delta u = I_\alpha\ast u^p\). If \(\alpha<N - 2\), this is equivalent to
\(
  u=I_{\alpha+2}\ast u^p.
\)
Existence and nonexistence of nontrivial nonnegative supersolutions in exterior domains for such equations were recently studied in \cites{Mitidieri,Caristi}.

As \(\alpha \to 0\) one has \(\lim_{\alpha\to 0} I_\alpha\ast\varphi=\varphi\) for every \(\varphi\in C^\infty_c(\R^N)\) and \(-\Delta u = u^{p + q}\) becomes a limiting equation for \eqref{eqChoquard0} when \(\alpha\to 0\).
Such local equation admits nontrivial nonnegative supersolutions in exterior domains if and only if \(p + q>\frac{N}{N - 2}\) (see \citelist{\cite{Gidas}\cite{Quittner-Souplet}*{Section 1.8}}). Similarly, as \(\alpha \to N\), \(-\Delta u = (4 \pi)^{-\frac{N}{2}} (\int_{\R^N} \abs{u}^p) u^q\) is a limiting equation which has a supersolution in an exterior domains if and only if \(q>\frac{N}{N - 2}\) and \(p > \frac{N}{N - 2}\).
Our results are formally consistent with these limiting cases.

%%%%%%%%%%%%%%%%%%%%%%%%%%%%%%%%%%%%%%%%%FFFFFFFFFFFFFFFFFFF

%%%%%%%%%%%%%%%%%%%%%%%%%%%%%%%%%%%%%%%%%%%ffffffffffffffffffffff

\subsection{Equations with fast decay and Hardy potentials.}
Consider Choquard equation \eqref{eqChoquardV} with the \emph{fast decay} potential, that is
\begin{equation}\label{eqChoquardFast}
\tag{$\mathcal{C}^F_{\lambda, \gamma}$}
  -\Delta u (x) + \frac{\lambda}{\abs{x}^\gamma} u (x) = (I_\alpha \ast u^p) (x) u(x)^q,
\end{equation}
where \(\lambda \in \R\) and \(\gamma > 2\).
In Theorem~\ref{Thm-fast} we show that all the nonexistence, existence and optimal decay results of Theorem~\ref{Thm-free} remain stable with respect to the perturbations of \(-\Delta\) by the fast decay potentials
and do not depend on particular values of \(\lambda\) and \(\gamma\).
This is a consequence of the well known fact that the fundamental solution of the Schr\"odinger operator \(-\Delta+V\)
with a fast decay potential \(V\) decays  at infinity as \(\abs{x}^{-(N-2)}\),
that is as the Green function of \(-\Delta\) on \(\R^N\).
It turns out that the values of all critical exponents and decay rates of Theorem~\ref{Thm-free}
are controlled by the decay rate of the fundamental solution of \(-\Delta+V\).
See Section~\ref{sect-Fast} for details.
\smallskip

In Section~\ref{Sect-Hardy} we study Choquard equation \eqref{eqChoquardV} with the \emph{Hardy} potential,
that is
\begin{equation}\label{eqChoquardHardy}
\tag{$\mathcal{C}^H_{\nu}$}
  -\Delta u (x) + \bigl( \nu^2-\big(\tfrac{N - 2}{2}\big)^2 \bigr)\frac{1}{\abs{x}^2} u (x) = (I_\alpha \ast u^p) (x) u(x)^q,
\end{equation}
where \(\nu > 0\). Hardy potential provides an important example of a perturbation
where the decay rate of the fundamental solution of \(-\Delta+V\) remains polynomial
but depends explicitly on the value of the constant \(\nu\).
In Theorem~\ref{Thm-Hardy} we show that as a consequence,
some of the critical exponents and decay rates of Theorem~\ref{Thm-free} become sensitive
to the constant \(\nu\), although the qualitative picture remains essentially
similar to the case of the unperturbed equation \eqref{eqChoquardV}.
Full statements and sketches of the proofs of relevant results are given in Section~\ref{Sect-Hardy}.

\subsection{Equation with slow decay potentials.}
Consider the Choquard equation \eqref{eqChoquardV} with the \emph{slow decay} potential, that is
\begin{equation}\label{eqChoquardSlow}
\tag{$\mathcal{C}^S_{\lambda, \gamma}$}
  -\Delta u (x) + \frac{\lambda^2}{\abs{x}^\gamma} u (x) = (I_\alpha \ast u^p) (x) u(x)^q,
\end{equation}
where \(\lambda>0\) and \(-\infty<\gamma < 2\).
It is well known that if \(V\) is a slow decay potential then the fundamental solution of \(-\Delta+V\)
decays exponentially at infinity. As a consequence, the qualitative picture of the existence and nonexistence
of positive supersolutions of \eqref{eqChoquardSlow} changes compared to equations with fast decay or Hardy potentials.

For local equations of type \eqref{C-local} with superlinear \(q>1\) one usually expects to find
a \emph{fast decay} positive solution, which decays at infinity at the same rate
as the fundamental solution of \(-\Delta+V\).
We will see that this is indeed the case for the \eqref{eqChoquardSlow} when \(q> 1\),
while for \(q<1\) positive supersolutions of \eqref{eqChoquardSlow} decay at most polynomially.
The decay of solutions in the borderline region \(q=1\) remains exponential
but the detailed picture becomes particularly complex.
Note however that for Choquard equation \eqref{eqChoquardSlow} the natural threshold between sub and superlinear homogeneity is \(p + q=1\) rather then \(q=1\), so the polynomial behavior of supersolutions to \eqref{eqChoquardSlow}
in the superlinear region seems to be a new phenomenon.

For equation \eqref{eqChoquardSlow} we shall consider separately the \emph{exponential decay region \(q\ge 1\)} and
the \emph{polynomial decay region \(q<1\)}, because different mechanisms are responsible
for the decay and nonexistence properties of positive solutions in these two regions.

\subsubsection{Exponential decay region $q > 1$.}
Our first result regarding equations with slow decaying potentials
states that in the globally superlinear case \(q>1\) then \eqref{eqChoquardSlow} always admits a positive solution
which decay at infinity at the same rate as the fundamental solution of \(-\Delta+V\).

\begin{theorem}\label{Thm-exp}
Let \(N\ge 1\), \(\lambda > 0\), \(\gamma<2\), \(0<\alpha<N\), \(p>0\), \(q>1\) and \(\rho > 0\).
Then equation \eqref{eqChoquardSlow} has a nontrivial nonnegative supersolution in \(\R^N \setminus \Bar{B}_{\rho}\).
Moreover, if \(u\ge 0\) is a nontrivial supersolution of \eqref{eqChoquardSlow} in \(\R^N \setminus \Bar{B}_{\rho}\) then
\begin{equation}\label{exp-1+}
\liminf_{\abs{x}\to \infty}u (x)\abs{x}^{\frac{N - 1}{2} - \frac{\gamma}{4}}\exp\big(\textstyle{\frac{2\lambda}{2-\gamma}}\abs{x}^\frac{2-\gamma}{2}\big)>0.
\end{equation}
The above lower bound is optimal.
\end{theorem}

\subsubsection{Locally linear region $q = 1$.}
In the borderline case \(q=1\) the behavior of positive solutions to \eqref{eqChoquardSlow} is more complex.
It turns out that the existence and decay properties of nontrivial nonnegative supersolutions of \eqref{eqChoquardSlow}
with \(q=1\) and arbitrary \(p>0\) are controlled by the relevant properties
of positive solutions of the linear equation
\begin{align}\label{fundamental}
-\Delta u (x) + \frac{\lambda^2}{\abs{x}^\gamma} u (x) & =\frac{m}{\abs{x}^{N - \alpha}} u (x),
\end{align}
where \(\gamma<2\), \(\lambda>0\) and \(m> 0\).
The detailed analysis of the decay rate of positive supersolutions of equation \eqref{fundamental} is given
in Section~\ref{fundamental-sect}.
The corresponding result for \eqref{eqChoquardSlow} reads as follows.

\begin{theorem}\label{Thm-exp+}
Let \(N\ge 1\), \(\lambda > 0\), \(\gamma<2\), \(0<\alpha<N\), \(p>0\), \(q=1\) and \(\rho > 0\).
Then equation \eqref{eqChoquardSlow} has a nontrivial nonnegative supersolution in \(\R^N \setminus \Bar{B}_{\rho}\) if and only if
\begin{equation}\label{nonexist-exp+}
  \gamma\le  N - \alpha.
\end{equation}
Moreover, if \(\gamma < N - \alpha\) and \(u\ge 0\) is a nontrivial nonnegative supersolution of \eqref{eqChoquardSlow} in \(\R^N \setminus \Bar{B}_{\rho}\) then
then there exists \(m > 0\) such that
\begin{equation}\label{exp-fund-Thm}
\liminf_{\abs{x}\to \infty} u(x)\abs{x}^{\frac{N - 1}{2} - \frac{\gamma}{4}}
\exp \int_{m^\frac{1}{N - \alpha - \gamma}}^{\abs{x}} \sqrt{\frac{\lambda^2}{s^{\gamma}}-\frac{m}{s^{N - \alpha}}\;} \,ds >0.
\end{equation}
and if \(\gamma = N - \alpha\) and \(u\ge 0\) is a nontrivial nonnegative supersolution of \eqref{eqChoquardSlow} in \(\R^N \setminus \Bar{B}_{\rho}\) then there exists \(m \in (0, \lambda) \) such that
\begin{equation}\label{exp-fund-Thm2}
  \liminf_{\abs{x}\to \infty}\textstyle{u (x)\abs{x}^{\frac{N - 1}{2} - \frac{\gamma}{4}}
\exp\big(\frac{2 (\lambda - m)}{2 - \gamma}) \abs{x}^\frac{2-\gamma}{2}\big)>0}.
\end{equation}
The above lower bounds are optimal.
\end{theorem}

The value of the constant \(m >0\) in  \eqref{exp-fund-Thm} and  in \eqref{exp-fund-Thm2} depends on the supersolution \(u\).
Hence, optimality of  \eqref{exp-fund-Thm} should be understood in the sense that
given any \(m >0\) there is a nontrivial supersolution \(u\ge 0\) of \eqref{eqChoquardSlow}
in an exterior domain such that
\[
 \limsup_{\abs{x}\to \infty} u(x)\abs{x}^{\frac{N - 1}{2} - \frac{\gamma}{4}}
\exp \int_{m^\frac{1}{N - \alpha - \gamma}}^{\abs{x}} \sqrt{\frac{\lambda^2}{s^{\gamma}}-\frac{m}{s^{N - \alpha}}\;} \,ds < \infty.
\]

\subsubsection{Sublinear region \(q<1\).}
When \(q<1\) solutions of \eqref{eqChoquardSlow} start to decay polynomially.
The values \(\gamma=N - \alpha\) and \(\gamma=-\alpha\) should be distinguished as two critical thresholds
separating different qualitative properties of positive supersolutions to \eqref{eqChoquardSlow}.
We set apart our results depending on the value of \(\gamma\).
First we consider the case when \(\alpha > N - 2\) and \(V\) decays at a moderately slow
rate \(N - \alpha \le \gamma < 2\).

\begin{theorem}\label{Thm-slow}
Let \(N\ge 1\), \(\lambda > 0\), \(\gamma<2\), \(N - 2<\alpha<N\), \(p>0\), \(q < 1\) and \(\rho > 0\).
If
\[
  N - \alpha \le \gamma<2,
\]
then equation \eqref{eqChoquardSlow} has no nontrivial nonnegative supersolution in \(\R^N \setminus \Bar{B}_\rho\).
\end{theorem}

\begin{figure}
\subfigure[\(N - \alpha\le\gamma<2\)]{\includegraphics{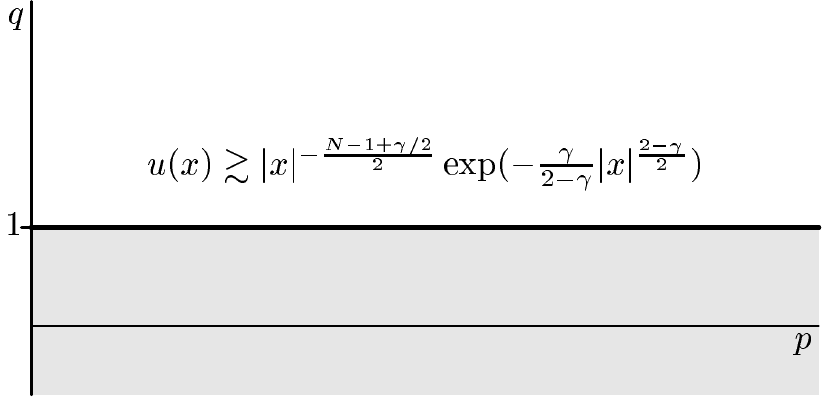}}
\subfigure[\(-\alpha\le \gamma < \min\{N - \alpha,2\}\)]{\includegraphics{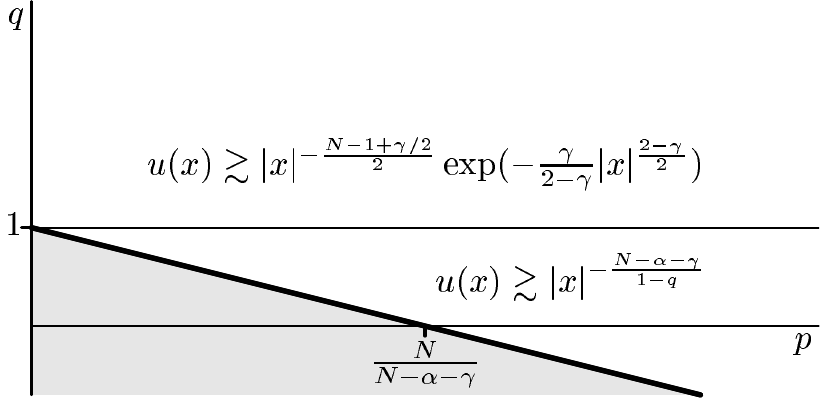}}
\subfigure[\(\gamma < -\alpha\)]{\includegraphics{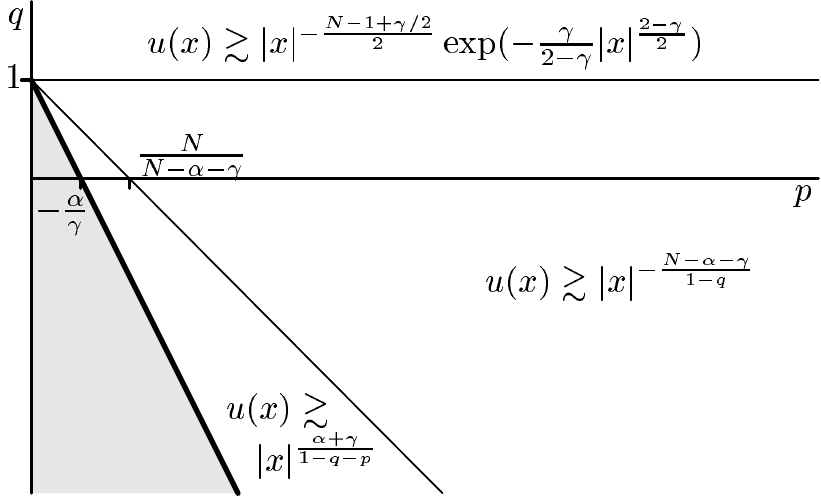}}
\caption{Existence, decay and nonexistence regions for \eqref{eqChoquardSlow} in the \((p,q)\)--plane}
\end{figure}

Next we look at the intermediate slow decay or slow growth r\'egime  \(-\alpha<\gamma < \min\{N - \alpha,2\}\), which includes in particular the autonomous case \(\gamma=0\).

\begin{theorem}\label{Thm-slow-moderate}
Let \(N\ge 1\), \(\lambda > 0\), \(\gamma<2\), \(0<\alpha<N\), \(p>0\), \(q < 1\), and \(\rho > 0\).
If \(p+q\neq 1\) and
\[-\alpha<\gamma < \min\{N - \alpha,2\}.\]
then equation \eqref{eqChoquardSlow} has a nontrivial nonnegative supersolution in \(\R^N \setminus \Bar{B}_{\rho}\) if and only if
\begin{equation}\label{non-slow-moderate}
  q > 1 - \frac{N - \alpha - \gamma}{N}p.
\end{equation}
Moreover, if \(u\ge 0\) is a nontrivial supersolution of \eqref{eqChoquardSlow} in \(\R^N \setminus \Bar{B}_{\rho}\) then
\begin{equation}\label{decay-slow-moderate}
 \liminf_{\abs{x} \to \infty} u (x)\abs{x}^{\frac{N - \alpha - \gamma}{1-q}} > 0.
\end{equation}
The above lower bound is optimal.
\end{theorem}

The most complex picture occurs in the fast growth regime  \(\gamma<-\alpha\).

\begin{theorem}\label{Thm-slow-fast}
Let \(N\ge 1\), \(\lambda > 0\), \(\gamma<2\), \(0<\alpha<N\), \(p>0\), \(q < 1\) and \(\rho > 0\).
If \(p+q\neq 1\) and
\[\gamma \le -\alpha.\]
then equation \eqref{eqChoquardSlow} has a nontrivial nonnegative supersolution in \(\R^N \setminus \Bar{B}_{\rho}\) if and only if
\begin{equation}\label{non-slow-fast}
q > 1+\frac{\gamma}{\alpha}p.
\end{equation}
Moreover, if \(u\ge 0\) is a nontrivial supersolution of \eqref{eqChoquardSlow} in \(\R^N \setminus \Bar{B}_{\rho}\) then
\begin{subequations}
\label{eqC}
\begin{align}
\liminf_{\abs{x} \to \infty} u (x)\abs{x}^{\frac{N - \alpha - \gamma}{1-q}}&>0
&\text{if }&
\textstyle{1+\frac{N - \alpha - \gamma}{N}p<q<1},\label{C-1}\\
\liminf_{\abs{x} \to \infty} u (x)\abs{x}^{\frac{N}{p}}(\log\abs{x})^{-\frac{1}{1 - q - p}}&>0
&\text{if }&
\textstyle{q=1+\frac{N - \alpha - \gamma}{N}p},\label{C-2}\\
\liminf_{\abs{x} \to \infty} u (x) \abs{x}^{-\frac{\alpha + \gamma}{1 - q - p}}&>0
&\text{if }&
\textstyle{1+\frac{\gamma}{\alpha}p<q<\frac{\alpha}{N - 2}}.\label{C-3}
\end{align}
\end{subequations}
The above lower bounds are optimal.
\end{theorem}

In the homogeneous case \(p + q=1\) we obtain results similar
to those of Theorem~\ref{Thm-slow-moderate}, with exception
that, unlike in all previous results, the existence becomes sensitive to the choice of radius $\rho>0$.
In addition, in the fully homogeneous case \(\gamma = -\alpha\) and \(p + q=1\),
the existence and nonexistence becomes sensitive to the value of \(\lambda\).
To ensure the existence of a positive solution,
\(\lambda\) has to be sufficiently large so that the potential \(V\) can compensate
for the loss of positivity due to the nonlocal right hand side. To formulate the result, denote
\[
  \lambda^\ast := 2^{-\frac{\alpha}{2}}\frac{\Gamma(\frac{N - \alpha}{4})}{\Gamma(\frac{N + \alpha}{4})}.
\]
This quantity is related to the optimal constant in a weighted Hardy-Littlewood-Sobolev inequality
of Stein and Weiss \cite{Stein-Weiss}, see Section~\ref{sect-pq1}.

\begin{theorem}\label{t-pq1}
Let \(N\ge 2\), \(0<\alpha<N\),  \(\lambda>0\) and \(p>0\).
If
\[p+q=1,\]
then 
\begin{compactenum}[(i)]
\item if $\alpha>-\gamma$ and $\lambda>0$, or if $\alpha=-\gamma$ and \(\lambda<\lambda^\ast\), then for every \(\rho > 0\), \eqref{eqChoquardSlow} has no nontrivial nonnegative supersolutions in \(\R^N \setminus \Bar{B}_{\rho}\);
\item if $\alpha<-\gamma$ and $\lambda>0$, or if $\alpha=-\gamma$ and \(\lambda>\lambda^\ast\), then there exists \(\rho_0 > 0\) such that if \(\rho > \rho_0\), \eqref{eqChoquardSlow} has a positive supersolutions in \(\R^N \setminus \Bar{B}_{\rho}\).
\end{compactenum}
Moreover, if \(u\ge 0\) is a nontrivial supersolution of \eqref{eqChoquardSlow} in \(\R^N \setminus \Bar{B}_{\rho}\) then
\begin{equation*}
 \liminf_{\abs{x} \to \infty} u (x)\abs{x}^{\frac{N}{1-q}} > 0.
\end{equation*}
The above lower bound is optimal if $\alpha<-\gamma$.
\end{theorem}

Theorem~\ref{t-pq1} gives only partial results in the borderline case \(p + q=1\).
The accurate description of the existence,
nonexistence and optimal decay properties of positive supersolutions of the equation
\[
  -\Delta u (x)+ \frac{\lambda^2}{\abs{x}^{\gamma}} u=(I_\alpha\ast u^{1-q}) (x) u^{q}(x)
\]
is an interesting open problem which is however lies beyond the scope of the present work.

\subsection{Outline.}
The rest of the paper is organized as follows.
In Section~\ref{sectPositivity} we prove general local and nonlocal versions of positivity principles, in the spirit
of the Agmon--Allegretto--Piepenbrink positivity principle (see \cites{Agmon,Agmon-2}).
These positivity principles are fundamental in our considerations both for nonexistence as well as for optimal decay estimates.
In Section~\ref{sectFree} we prove Theorem~\ref{Thm-free} and discuss briefly equation with fast decay potentials.
In Section~\ref{Sect-Hardy} we consider equation with  Hardy potentials.
Finally, in Sections~\ref{sectSlowLargeq} and~\ref{sectSlowSmallq} we study equation with slow decay potentials.
Appendix A contains various estimates of the Riesz potentials which are extensively used in the paper.
In Appendix B we prove suitable versions of a comparison principle and a weak Harnack inequality
for distributional supersolutions.

\section{Local and nonlocal positivity principles.}
\label{sectPositivity}

According to the classical Agmon--Allegretto--Piepenbrink positivity principle  (see~\cite{Agmon}*{Theorem 3.1}),
the linear equation \(-\Delta u+Vu=0\) admits a nontrivial nonnegative weak supersolution \(u\in H^1_{\mathrm{loc}}(\Omega)\) if and only if
the corresponding quadratic form \(\int_{\Omega} \abs{\nabla \varphi}^2 +V \varphi^2\)  is nonnegative
for every \(\varphi \in C^\infty_c(\Omega)\).
An extension of such positivity principle to distributional supersolutions can be found in \cite{CyconFroeseKirschSimon}*{Theorem 2.12} (see also \cite{Fall-I}*{Lemma B.1}).

We formulate a version of the positivity principle adapted to distributional supersolutions of the nonlinear equation
\[
-\Delta u+Vu=W u^q.
\]

\begin{proposition}
\label{propositionLocalGroundstate}
Let \(N\ge 1\), \(\Omega \subseteq \R^N\) be an open connected set and \(q\in\R\).
Let  \(V \in L^1_{\mathrm{loc}} (\Omega)\), \(W : \Omega \to [0, \infty)\) be measurable and let \(u \in L^1_\mathrm{loc}(\Omega)\).
Assume that \(Vu,\, W u^q \in L^1_\mathrm{loc}(\Omega)\).
If \(u \ge 0\) and
\[
  -\Delta u + V u \ge W u^q \quad\text{in \(\Omega\)},
\]
in the sense of distributions, then either \(u = 0\) almost everywhere
or \(u > 0\) almost everywhere in \(\Omega\), \(W u^{q - 1} \in L^1_\mathrm{loc} (\Omega)\) and for every \(\varphi \in C^\infty_c(\Omega)\) one has
\begin{equation*}
  \int_{\Omega} \abs{\nabla \varphi}^2 +V \varphi^2\ge\int_{\Omega}W u^{q - 1}\varphi^2.
\end{equation*}
\end{proposition}

The conclusion \(u > 0\) is crucial to interpret the \(u^{q - 1}\) when \(q \le 1\).
If \(q < 0\) and \(W > 0\), then this conclusion is already contained implicitly in the assumption
\(W u^q \in L^1_\mathrm{loc} (\Omega)\).

\begin{proof}[Proof of Proposition~\ref{propositionLocalGroundstate}]
Let \(\eta\in C^\infty_c(\R^N)\) be such that \(\supp \eta \subset B_1\), \(\int_{\R^N} \eta = 1\) and \(\eta \ge 0\).
For \(\delta>0\) and \(x \in \R^N\), let \(\eta_\delta(x)=\delta^{-N}\eta (x/\delta)\) and let \(\check{\eta}_{\delta} (x) = \eta_\delta (-x)\).
Let \(\varphi \in C^\infty_c(\Omega)\).
Since the support of \(\varphi\) is compact, there exists \(\delta_0 > 0\) such that for every \(x \in \Omega\), \(\varphi (x) > 0\) implies \(B_{\delta_0} (x) \subset \Omega\).
Given \(\varepsilon>0\) and \(\delta\in(0,\delta_0]\), we can thus take \(\check{\eta}_\delta \ast \frac{\varphi^2}{\eta_\delta \ast u + \varepsilon}\in C^\infty_c(\Omega)\) as a test function.
We compute
\[
\begin{split}
 - \int_{\Omega} u \,\Delta \Bigl(\check{\eta}_\delta \ast \frac{\varphi^2}{\eta_\delta \ast u + \varepsilon} \Bigr)
 &= - \int_{\Omega} (\eta_\delta \ast u)\, \Delta \Bigl(\frac{\varphi^2}{\eta_\delta \ast u + \varepsilon} \Bigr)\\
 &=  \int_{\Omega} \nabla (\eta_\delta \ast u)
                   \cdot \nabla \Bigl(\frac{\varphi^2}{\eta_\delta \ast u + \varepsilon} \Bigr)
\end{split}
\]
and
\begin{equation*}
  \abs{\nabla \varphi}^2 = \nabla (\eta_\delta \ast u) \cdot \nabla \Bigl( \frac{\varphi^2}{\eta_\delta \ast u +\varepsilon}\Bigr)+\Bigabs{\nabla \Bigl( \frac{\varphi}{\eta_\delta \ast u + \varepsilon} \Bigr)}^2(\eta_\delta \ast u + \varepsilon)^2.
\end{equation*}
Therefore,
\begin{equation}
\label{eqGroundStateMaster}
 \int_{\Omega} \abs{\nabla \varphi}^2 + \frac{\eta_\delta \ast (V u)}{\eta_\delta \ast u + \varepsilon} \varphi^2 \ge \int_{\Omega} \frac{\eta_\delta \ast (W u^q)}{\eta_\delta \ast u + \varepsilon} \varphi^2
+ \int_{\Omega} \Bigabs{\nabla \Bigl( \frac{\varphi}{\eta_\delta \ast u + \varepsilon} \Bigr)}^2(\eta_\delta \ast u + \varepsilon)^2 .
\end{equation}
Since \(Vu \in L^1_{\mathrm{loc}} (\Omega)\) and \(W u^q \in L^1_\mathrm{loc}(\Omega)\), one has \(\eta_\delta \ast Vu \to Vu\) and \(\eta_\delta \ast W u^q \to W u^q\) in \(L^1_\mathrm{loc} (\Omega)\) as \(\delta \to 0\). Since \(u \in L^1 (\Omega)\), \(\eta_\delta \ast u \to u\) almost everywhere in \(\Omega\) as \(\delta \to 0\).
Hence, letting \(\delta \to 0\) we obtain by Lebesgue's dominated convergence theorem
\[
 \int_{\Omega} \abs{\nabla \varphi}^2 + \frac{V u}{u + \varepsilon} \varphi^2 \ge \int_{\Omega} \frac{W u^q}{u + \varepsilon} \varphi^2.
\]
Letting now \(\varepsilon \to 0\), by Lebesgue's dominated convergence theorem again and by Lebesgue's monotone convergence theorem we conclude that
\[
  \int_{\Omega} \abs{\nabla \varphi}^2 + V \varphi^2 \ge \int_{\Omega \setminus u^{-1} (\{0\})} W u^{q - 1} \varphi^2.
\]

Let us now prove that \(u > 0\) almost everywhere, following H.\thinspace Brezis and A.\thinspace C.\thinspace{} Ponce \cite{BrezisPonce2003}.
Let \(a \in \Omega\) and \(\rho > 0\) such that \(B_{2 \rho} (a) \subset \Omega\) and take \(\varphi \in C^\infty_c(B_{2 \rho}(a))\) such that \(\varphi = 1\) on \(B_{\rho} (a)\).
By the Poincar\'e--Wirtinger inequality and the triangle inequality we have
\[
\int_{B_{\rho}(a)}\int_{B_{\rho}(a)} \Bigabs{\log \frac{\eta_\delta \ast u (x) + \varepsilon}{\eta_\delta \ast u (y) + \varepsilon}} \,dx\,dy
\le C \int_{B_{\rho} (a)} \Bigl\lvert \nabla \frac{1}{\eta_\delta \ast u + \delta}\Bigr\rvert^2 \abs{\eta_\delta \ast u + \delta}^2,
\]
whence by \eqref{eqGroundStateMaster}
\[
\int_{B_{\rho}(a)}\int_{B_{\rho}(a)} \Bigabs{\log \frac{\eta_\delta \ast u (x) + \varepsilon}{\eta_\delta \ast u (y) + \varepsilon} } \,dx\,dy
\le \int_{\Omega} \abs{\nabla \varphi}^2 + \frac{(\eta_\delta \ast Vu)}{\eta_\delta \ast u + \varepsilon} \varphi^2 - \frac{\eta_\delta \ast W u^q}{\eta_\delta \ast u + \varepsilon} \varphi^2.
\]
Letting now \(\delta \to 0\), we have as before
\[
 \int_{B_{\rho}(a)}\int_{B_{\rho}(a)} \Bigl\lvert \log \frac{u (x) + \varepsilon}{u (y) + \varepsilon} \Bigr\rvert\,dx\,dy \le
\int_{\Omega} \abs{\nabla \varphi}^2 + \frac{V u}{u + \varepsilon} \varphi^2.
\]
Now note that for every \(x, y \in B_{\rho}(a) \times B_{\rho}(a)\) such that \(u (x) = 0\) and \(u (y) > 0\),
\[
 \lim_{\varepsilon \to 0} \Bigabs{\log \frac{u (x) + \varepsilon}{u (y) + \varepsilon}} = \infty,
\]
this allows to conclude that either \(u = 0\) or \(u > 0\) almost everywhere in \(B_{\rho} (a)\).
Since \(a \in \Omega\) is arbitrary and \(\Omega\) is connected, we conclude that either \(u > 0\) almost everywhere or \(u > 0\) almost everywhere in \(\Omega\).
\end{proof}

If one assumed \(u \in H^1_{\mathrm{loc}} (\Omega)\), one could have taken directly \(\delta = 0\) in the above proof.
If \(u^{-1} \in L^\infty_\mathrm{loc}(\Omega)\), one could take directly  \(\varepsilon = 0\).
The latter would follow
from the weak Harnack inequality if for example \(V_+ \in L^{s}_\mathrm{loc}(\Omega)\) for some \(s > \frac{N}{2}\) (see for example \cite{GilbargTrudinger1983}*{Theorem 8.18}),
but uniform positivity of the solution fails for more singular potentials \(V\in L^1_\mathrm{loc}(\Omega)\).

In the context of Choquard's equation \eqref{eqChoquardV}
we prove the following nonlocal version of the Agmon--Allegretto--Piepenbrink positivity principle for distributional solutions in the sense of Definition~\ref{def-distro}.

\begin{proposition}
\label{propositionGroundState}
Let \(N\ge 1\), \(\Omega\subseteq\R^N\) be an open and connected set, \(p>0\), \(q\in\R\), \(0 < \alpha < N\), \(V \in L^1_{\mathrm{loc}}(\Omega)\) and \(u \in L^1_{\mathrm{loc}} (\R^N)\).
If \(u \ge 0\) and
\begin{equation*}
  -\Delta u + V u \ge (I_\alpha \ast u^p)u^q\quad\text{in \(\Omega\)},
\end{equation*}
in the sense of distribution, then either \(u = 0\) almost everywhere or \(u > 0\) almost everywhere in \(\Omega\), \(u^{q - 1} \in L^1_\mathrm{loc} (\Omega)\) and
for every \(R>0\) and \(\varphi \in C^\infty_c(\Omega\cap B_R)\) one has
\begin{equation*}
  \int_{\Omega} \abs{\nabla \varphi}^2 + \int_{\Omega} V \varphi^2 \ge \frac{A_\alpha}{2^{N - \alpha} R^{N - \alpha}} \Bigl(\int_{\Omega \cap B_R} u^{p}\Bigr) \Bigl(\int_{\Omega} u^{q - 1} \varphi^2 \Bigr).
\end{equation*}
\end{proposition}

\begin{proof}
By Proposition~\ref{propositionLocalGroundstate} with \(W = I_\alpha \ast u^p\in L^1_\mathrm{loc} (\Omega)\),
either \(u = 0\) in \(\R^N\setminus \Bar{B}_{\rho}\) or \(u > 0\) almost everywhere in \(\Omega\), \( (I_\alpha \ast u^p) u^{q - 1} \in L^1_\mathrm{loc} (\Omega)\)
and for every \(\varphi \in C^\infty_c(\Omega)\)
\begin{equation*}
  \int_{\Omega} \abs{\nabla \varphi}^2 +V \varphi^2
  \ge \int_{\Omega}(I_\alpha \ast u^p) u^{q - 1}\varphi^2.
\end{equation*}
Since for every \(x, y \in B_R\),
\[
 I_\alpha(x-y) \ge \frac{A_\alpha}{2^{N - \alpha} R^{N - \alpha}},
\]
we have \(u^{q - 1} \in L^1_\mathrm{loc} (\Omega)\),
\[
  \frac{A_\alpha}{2^{N - \alpha} R^{N - \alpha}} \Bigl(\int_{\Omega \cap B_R} u^p \Bigr) \Bigl(\int_{\Omega} u^{q - 1}\varphi^2
 \Bigr) \le \int_{\Omega}(I_\alpha \ast u^p) u^{q - 1}\varphi^2,
\]
and the conclusion follows.
\end{proof}

\section{Equation with the unperturbed Laplacian: proof of Theorem~\ref{Thm-free}.}
\label{sectFree}

\subsection{Nonexistence.}
Our essential tools in the analysis of nonexistence of nontrivial nonnegative supersolutions to equation \eqref{eqChoquard0}
are the nonlocal positivity principle of Proposition~\ref{propositionGroundState}
and the following quantitative integral estimate, which can be viewed as an integral version
of the comparison principle for the Laplacian in exterior domains.
The result is a particular case of its generalization to Hardy potentials in Proposition~\ref{propositionLinearLowerBound}.

\begin{proposition}
\label{propositionLinearLowerBoundFree}
Let \(R, \rho > 0\) be such that \(R > 2 \rho\), \(u \in L^1_\mathrm{loc}(B_{4 R} \setminus B_{\rho /4})\) and \(f \in L^1_\mathrm{loc}(B_{4 R} \setminus B_{\rho/4})\). If
\(u \ge 0\), \(f \ge 0\) in \(B_{4 R} \setminus \Bar{B}_{\rho / 4}\) and
\[
 -\Delta u \ge f \quad\text{in}\quad B_{4 R} \setminus \Bar{B}_{\rho/4}
\]
in the sense of distributions, then
\[
 \frac{1}{\rho^{2}} \int_{B_{\rho} \setminus B_{\rho/2}} u
 + \int_{B_{R} \setminus B_{\rho}} f
  \le \frac{C}{R^{2}} \int_{B_{2R} \setminus B_{R}} u.
\]
\end{proposition}

Taking \(f=0\) and applying the weak Harnack inequality for superharmonic functions
(see \cite{Lieb-Loss}*{Theorem 9.10} or Lemma~\ref{P-XXX} below),
we immediately derive from Proposition~\ref{propositionLinearLowerBoundFree} the standard Green decay pointwise lower bounds on nontrivial nonnegative supersolutions to \eqref{eqChoquard0} in exterior domains.

\begin{lemma}
\label{lemmaGreenDecay}
Let \(\rho > 0\) and \(u \in L^1_\mathrm{loc}(\R^N\setminus B_{\rho})\).
If \(u \ge 0\) and
\[
 -\Delta u \ge 0\quad\text{in}\quad \R^N\setminus \Bar{B}_{\rho}
\]
in the sense of distributions, then either \(u = 0\) in \(\R^N\setminus \Bar{B}_{\rho}\), or
\begin{enumerate}[(i)]
 \item \label{eqLinearDecay12} if \(N=1,\,2\), then \(\liminf_{\abs{x} \to \infty} u (x) > 0\),
 \item \label{eqLinearDecay} if \(N \ge 3\), then \(\liminf_{\abs{x} \to \infty} u (x)\abs{x}^{N - 2} > 0\).
\end{enumerate}
\end{lemma}

Comparing Riesz potential blowup upper bound of \eqref{equLp} with
the Green decay bounds of Lemma~\ref{lemmaGreenDecay} we immediately
establish our first nonexistence results.

\begin{proposition}\label{dim12}
Let \(N=1,2\) and $\rho>0$.
If \(u \ge 0\) is a supersolution of \eqref{eqChoquard0} in \(\R^N \setminus \Bar{B}_{\rho}\) then \(u = 0\) in \(\R^N\setminus \Bar{B}_{\rho}\).
\end{proposition}
\begin{proof}
Simply note that Lemma~\ref{lemmaGreenDecay} \eqref{eqLinearDecay12} and \eqref{equLp} are incompatible for all \(p>0\).
\end{proof}

\begin{proposition}
Let \(N\ge 3\) and $\rho>0$.
Assume that \(p \le \frac{\alpha}{N - 2}\).
If \(u \ge 0\) is a supersolution of \eqref{eqChoquard0} in \(\R^N \setminus \Bar{B}_{\rho}\) then \(u = 0\) in \(\R^N\setminus \Bar{B}_{\rho}\).
\end{proposition}
\begin{proof}
Simply note that Lemma~\ref{lemmaGreenDecay} \eqref{eqLinearDecay} and \eqref{equLp} are incompatible for \(p \le \frac{\alpha}{N - 2}\).
\end{proof}

Our next step is to explore nonlocal positivity principle of Proposition~\ref{propositionGroundState}
in order to obtain an upper bound on nontrivial nonnegative supersolutions of \eqref{eqChoquard0}
in the superlinear region \(p + q\ge 1\) more accurate than \eqref{equLp}.

\begin{lemma}
\label{lemmaLiminfExistence}
Let \(N \ge 1\), \(0 < \alpha < N\), \(p > 0\), \(q \in \R\) and \(\rho > 0\).
There exists \(C > 0\) such that if \(u \ge 0\) is a supersolution of \eqref{eqChoquard0} in \(\R^N \setminus \Bar{B}_{\rho}\), then either \(u = 0\) almost everywhere or \(u > 0\) almost everywhere, \(u^{q - 1} \in L^1_\mathrm{loc} (\R^N \setminus \Bar{B}_{\rho})\) and for every \(R \ge 2 \rho\),
\[
   \Bigl(\int_{B_{2R}  \setminus B_{\rho}} u^{p} \Bigr) \Bigl(\int_{B_{2R} \setminus B_R} u^{q - 1}\Bigr) \le C R^{2N - \alpha - 2}.
\]
\end{lemma}
\begin{proof}
Choose \(\varphi \in C^\infty_c(\R^N)\) such that \(\supp \varphi \subset B_{4} \setminus \Bar{B}_{1/2}\), \(\varphi = 1\) on \(B_2 \setminus \Bar{B}_1\) and \(\varphi \le 1\).
For \(R > 0\) and \(x \in \R^N\) set
\begin{equation}\label{phiR}
\varphi_R (x) = \varphi\bigl(\tfrac{x}{R}\bigr).
\end{equation}
If \(R \ge 2\rho\), we have
\(\supp \varphi_R \subset B_{4 R} \setminus \Bar{B}_{R/2}\subset{\R^N \setminus \Bar{B}_{\rho}}\) and
\begin{equation*}
\label{scaling-grad}
 \int_{{\R^N \setminus B_{\rho}}} \abs{\nabla \varphi_R}^2= R^{N - 2} \int_{\R^N} \abs{\nabla \varphi}^2.
\end{equation*}
Therefore, by Proposition~\ref{propositionGroundState},
\[
 \Bigl(\int_{B_{2R} \setminus B_{\rho}} u^{p} \Bigr) \Bigl(\int_{B_{2R} \setminus B_R} u^{q - 1}\Bigr) \le R^{2N - \alpha - 2} \Bigl(\frac{4^{N - \alpha}}{A_\alpha} \int_{\R^N} \abs{\nabla \varphi}^2 \Bigr).\qedhere
\]
\end{proof}

A consequence of the upper bound of Lemma~\ref{lemmaLiminfExistence} is the following nonexistence result.

\begin{proposition}
Let \(N \ge 3\) and \(\rho > 0\).
Assume that
\[
  1\le p + q \le \frac{N + \alpha}{N - 2}.
\]
If \(u \ge 0\) is a supersolution of \eqref{eqChoquard0} in \(\R^N \setminus \Bar{B}_{\rho}\) then \(u = 0\) in \(\R^N\setminus \Bar{B}_{\rho}\).
\end{proposition}

\begin{proof}
In the \emph{subcritical case} \(1\le p + q < \frac{N + \alpha}{N - 2}\), we observe
that by the Cauchy--Schwarz inequality and Lemma~\ref{lemmaGreenDecay} \eqref{eqLinearDecay} we obtain
\[
  \Bigl(\int_{B_{2R}  \setminus B_{\rho}} u^{p} \Bigr) \Bigl(\int_{B_{2R} \setminus B_R} u^{q - 1}\Bigr) \ge \Bigl(\int_{B_{2R} \setminus B_R} u^\frac{p + q - 1}{2}\Bigr)^2 \ge R^{2 N-(N - 2)(p + q - 1)}.
\]
This is not compatible with the upper bound of Lemma~\ref{lemmaLiminfExistence}.

In the \emph{critical case} \(p + q=\frac{N + \alpha}{N - 2}\), we need to improve the lower bound
of Lemma~\ref{lemmaGreenDecay} \eqref{eqLinearDecay}.
To do this, assume by contradiction that \(u > 0\) on a set of positive measure of \(\R^N \setminus \Bar{B}_{\rho}\).
Using the Cauchy--Schwarz inequality and the lower bound of Lemma~\ref{lemmaGreenDecay} \eqref{eqLinearDecay},
since \(\frac{p + q}{2} = \frac{N + \alpha}{2(N - 2)} > 0\) we obtain
\[
\begin{split}
  \int_{\R^N \setminus B_{\rho}} (I_\alpha \ast u^p) u^q & \ge
  \int_{\R^N \setminus B_{\rho}} \int_{\R^N \setminus B_{\rho}} u (x)^{\frac{p + q}{2}} I_\alpha (x - y) u (y)^{\frac{p + q}{2}}\,dx\,dy\\
  &\ge \int_{\R^N \setminus B_{\rho}} \int_{\R^N \setminus B_{\rho}} \frac{1}{\abs{x}^\frac{N + \alpha}{2}} I_\alpha (x - y) \frac{1}{\abs{y}^\frac{N + \alpha}{2}}\,dx\,dy = \infty.
\end{split}
\]
Then by Proposition~\ref{propositionLinearLowerBoundFree} we conclude that
\[
 \liminf_{R \to \infty} \frac{1}{R^2} \int_{B_{2R} \setminus B_R} u = \infty.
\]
Applying H\"older's inequality if \(p + q \ge 3\) and the weak Harnack inequality (Lem\-ma~\ref{P-XXX})
if \(1 \le p + q < 3\), we obtain
\[
 \Bigl(\frac{1}{R^2} \int_{B_{2R} \setminus B_R} u\Bigr)^\frac{p + q - 1}{2} \le
 \frac{C}{R^{N - \frac{\alpha + 2}{2}}} \int_{B_{2R} \setminus B_R} u^\frac{p + q - 1}{2},
\]
which brings a contradiction with the upper bound of Lemma~\ref{lemmaLiminfExistence} combined with the Cauchy-Schwarz inequality.
\end{proof}

If \(\alpha \ge N - 2\) we give precise lower bounds on \(\int_{B_{2 R} \setminus B_R} u^{q - 1}\) to obtain a further nonexistence result.

\begin{proposition}
\label{propositionLocalNonexistence}
Let \(N\ge 3\) and \(\rho > 0\).
Assume that \(\alpha \ge N - 2\) and \(1< q \le \frac{\alpha}{N - 2}\).
If \(u \ge 0\) is a supersolution of \eqref{eqChoquard0} in \(\R^N \setminus \Bar{B}_{\rho}\) then \(u = 0\) in \(\R^N\setminus \Bar{B}_{\rho}\).
\end{proposition}

\begin{proof}
Assume that \(u > 0\) on a set of positive measure.
From Lemma~\ref{lemmaGreenDecay} \eqref{eqLinearDecay} we obtain
\begin{equation}
\label{ineqLoweruq_1}
 \int_{B_{2R} \setminus B_R} u^{q - 1} \ge c R^{N - (N - 2)(q - 1)}.
\end{equation}
On the other hand, by Lemma~\ref{lemmaLiminfExistence}
\[
 \int_{B_{2R} \setminus B_R} u^{q - 1} \le \frac{C R^{2 N- \alpha - 2}}{\displaystyle \int_{B_{2R}  \setminus B_{\rho}} u^p} \le C' R^{2 N- \alpha - 2}.
\]
This brings a contradiction if \(q < \frac{\alpha}{N - 2}\).

In the case \(q = \frac{\alpha}{N - 2}\) using Lemma~\ref{lemmaGreenDecay} \eqref{eqLinearDecay} we obtain
\[
\begin{split}
 \int_{\R^N \setminus B_{2 \rho}} (I_\alpha \ast u^p) u^q & \ge \frac{A_\alpha}{2^{N - \alpha}} \int_{B_{2 \rho} \setminus B_\rho} u^p \int_{\R^N \setminus B_{\rho}} \frac{u (x)^q}{\abs{x}^{N - \alpha}}\\
&\ge c \int_{\R^N \setminus B_{\rho}} \frac{1}{\abs{x}^{N - \alpha + (N - 2) q}} \,dx = \infty,
\end{split}
\]
for some \(c > 0\),
since \((N - 2) q = \alpha\).
By Proposition~\ref{propositionLinearLowerBoundFree} and the weak Harnack inequality,
\[
 \liminf_{R \to \infty} \frac{1}{R^{N - (N - 2) (q - 1)} } \int_{B_{2R} \setminus B_R} u^{q - 1} = \infty,
\]
which leads to a contradiction with \eqref{ineqLoweruq_1}.
\end{proof}

An alternative proof of Proposition~\ref{propositionLocalNonexistence} is obtained by
noting that if \(\abs{x} \ge 2 \rho\),
\[
 I_\alpha \ast u (x) \ge \frac{A_\alpha}{2^{N - \alpha}\abs{x}^{N - \alpha}} \int_{B_{2 \rho} \setminus B_\rho} u^p,
\]
and exploring the fact that
\begin{equation}
\label{eqnLocal}
  - \Delta u (x) \ge \frac{c}{\abs{x}^{N - \alpha}} u (x)^{q}
\end{equation}
does not have positive solutions in exterior domains if \(1 < q \le \frac{\alpha}{N - 2}\)
\citelist{\cite{KLS}*{Theorem 1.2}\cite{LLM}*{Theorem 2.2 and Lemma 6.3}\cite{Brezis-Tesei}}.
\smallskip

The transitional locally linear case \(q=1\) requires a special consideration.

\begin{proposition}
Let \(N\ge 3\) and \(\rho > 0\).
Assume that \(\alpha > N - 2\) and \(q=1\).
If \(u \ge 0\) is a supersolution of \eqref{eqChoquard0} in \(\R^N \setminus \Bar{B}_{\rho}\) then \(u = 0\) in \(\R^N\setminus \Bar{B}_{\rho}\).
\end{proposition}

\begin{proof}
Since \(q=1\), by Lemma~\ref{lemmaLiminfExistence} for any \(R>\rho\) we have
\[
R^{2N - \alpha-2}\ge c\Bigl(\int_{B_R \setminus B_{\rho}} u^{p}\Bigr)\Bigl(\int_{B_{2R} \setminus B_R} 1\Bigr).
\]
We conclude that \(u = 0\) in \(\R^N\setminus \Bar{B}_{\rho}\) when \(N - 2<\alpha\).
\end{proof}

The next result shows that in the superlinear case \(p + q\ge 1\)
different mechanisms are responsible for the nonexistence and decay of nontrivial nonnegative supersolutions of \eqref{eqChoquard0}
in the Green decay region \(q\ge 1\) and \emph{sublinear decay region} \(q<1\).
\begin{proposition}
Let \(N\ge 3\) and $\rho>0$.
Assume that \(p + q\ge 1\), \(q < 1\) and
\[
  q \le 1 - \frac{N - \alpha - 2}{N}p.
\]
If \(u \ge 0\) is a supersolution of \eqref{eqChoquard0} in \(\R^N \setminus \Bar{B}_{\rho}\) then \(u = 0\) in \(\R^N\setminus \Bar{B}_{\rho}\).
\end{proposition}
\begin{proof}
We observe that by H\"older's inequality since \(q < 1\),
\[
\begin{split}
  \int_{B_{2R} \setminus B_R} 1 &\le \Bigl( \int_{B_{2R} \setminus B_R} u^p \Bigr)^\frac{1 - q}{p + 1 - q}\Bigl( \int_{B_{2R} \setminus B_R} u^{q - 1} \Bigr)^\frac{p}{p + 1 - q}\\
  &\le \Bigl( \int_{B_{2R} \setminus B_R} u^p \Bigr)^\frac{1 - q}{p + 1 - q}\biggl(\Bigl( \int_{B_{2R} \setminus B_R} u^{q - 1} \Bigr)\Bigl( \int_{B_{2R} \setminus B_R} u^{q - 1} \Bigr)\biggr)^\frac{p + q - 1}{p + 1 - q}.
\end{split}
\]
By Lemma~\ref{lemmaLiminfExistence},
on the one hand
\[
 \Bigl(\int_{B_{2R} \setminus B_R} u^p \Bigr)\Bigl( \int_{B_{2R} \setminus B_R} u^{q - 1} \Bigr) \le C R^{2N - \alpha - 2},
\]
and on the other hand
\[
 \int_{B_{2R} \setminus B_R} u^{q - 1} \le C \frac{R^{2N - \alpha - 2}}{\displaystyle \int_{B_{2R}  \setminus B_{\rho}} u^{p}} \le  \frac{C}{\displaystyle \int_{B_{2 \rho}  \setminus B_{\rho}} u^{p}}R^{2N - \alpha - 2}.
\]
In the \emph{subcritical case} \(q < 1 - \frac{N - \alpha - 2}{N}p\), this brings a contradiction since \(p + q \ge 1\).

Otherwise, in the \emph{critical case} \(q = 1 - \frac{N - \alpha - 2}{N}p\), we have by the previous inequalities
\[
 \int_{B_{2R} \setminus B_R} u^p \ge \frac{\Bigl(\displaystyle{\int_{B_{2R} \setminus B_R} 1}\Bigr)^{1 + \frac{p}{1 - q}}} {\displaystyle{\Bigl(\int_{B_{2R} \setminus B_R} u^{1 - q} }\Bigr)^\frac{p}{1 - q} } \ge c R^{N-\frac{p}{1 - q} (N - \alpha - 2)}= c
\]
for some \(c > 0\).
Therefore,
\[
 \lim_{R \to \infty} \int_{B_{2R}  \setminus B_{\rho}} u^{p} = \infty
\]
and
\[
 \lim_{R \to \infty} R^{\alpha + 2 - 2 N} \int_{B_{2R} \setminus B_R} u^{q - 1} = 0,
\]
so we can conclude as previously.
\end{proof}

Comparing Riesz potential blowup upper bound of \eqref{equLp} with
Lemma~\ref{lemmaLiminfExistence} we obtain the next nonexistence statement.

\begin{proposition}
Let \(N\ge 3\) and \(\rho > 0\).
Assume that \(p + q<1\).
If \(u \ge 0\) is a supersolution of \eqref{eqChoquard0} in \(\R^N \setminus \Bar{B}_{\rho}\) then \(u = 0\) in \(\R^N\setminus \Bar{B}_{\rho}\).
\end{proposition}

\begin{proof}
Since \(p + q  < 1\), by the H\"older inequality we obtain
\[
\begin{split}
  \int_{B_{2R} \setminus B_R} 1 &\le \Bigl( \int_{B_{2R} \setminus B_R} u^p \Bigr)^\frac{1 - q}{p + 1 - q}\Bigl( \int_{B_{2R} \setminus B_R} u^{q - 1} \Bigr)^\frac{p}{p + 1 - q}\\
  &\le \Bigl( \int_{B_{2R} \setminus B_R} u^p \Bigr)^\frac{1 - q - p}{p + 1 - q}\biggl(\Bigl( \int_{B_{2R} \setminus B_R} u^{p} \Bigr) \Bigl( \int_{B_{2R} \setminus B_R} u^{q - 1} \Bigr)\biggr)^\frac{p}{p + 1 - q}.
\end{split}
\]
By \eqref{equLp} and by Lebesgue's dominated convergence theorem
\[
  \lim_{R \to \infty} R^{N - \alpha} \int_{B_{2R} \setminus B_R} u^p = 0.
\]
This brings a contradiction with Lemma~\ref{lemmaLiminfExistence}.
\end{proof}

\subsection{Pointwise decay bounds.}
In the sublinear decay region \(q<1\) the integral estimate
of Lemma~\ref{lemmaLiminfExistence} could be used to prove that
the Green decay bounds of Lemma~\ref{lemmaGreenDecay} are no longer accurate if applied to \eqref{eqChoquard0}.
In fact, if \(q<1\) then nontrivial nonnegative supersolution of \eqref{eqChoquard0} in \(\R^N \setminus \Bar{B}_{\rho}\)
decay at the same polynomial rate as positive supersolutions
to the sublinear local equation \eqref{eqnLocal}.

\begin{proposition}
\label{propLimsupExistence-local-free}
Let \(N\ge 3\) and \(\rho > 0\).
Assume that \(q<1\).
If \(u\ge 0\) is a supersolution of \eqref{eqChoquard0} in \(\R^N \setminus \Bar{B}_{\rho}\), then
either \(u = 0\) in \(\R^N\setminus \Bar{B}_{\rho}\) or
\begin{equation*}
 \liminf_{\abs{x}\to \infty} u (x)\abs{x}^{\frac{N - \alpha - 2}{1-q}} > 0.
\end{equation*}
\end{proposition}

In view of Proposition~\ref{propositionLocalNonexistence}, Proposition~\ref{propLimsupExistence-local-free} is trivially true when \(\alpha > N - 2\).
Proposition~\ref{propLimsupExistence-local-free} only gives an improvement over Lemma~\ref{lemmaGreenDecay} if \(q < \frac{\alpha}{N - 2}\).

The lower bound of Proposition~\ref{propLimsupExistence-local-free} for the local equation \eqref{eqnLocal} is well--known \cite{LLM}*{Lemma 6.1}.
We present the proof here 6for completeness and to illustrate the use of Lemma~\ref{lemmaLiminfExistence}.

\begin{proof}[Proof of Proposition~\ref{propLimsupExistence-local-free}]
By Lemma~\ref{lemmaLiminfExistence} we have
\[
 \int_{B_{2R} \setminus B_R} u^{q - 1} \le C \frac{R^{2N - \alpha - 2}}{\displaystyle \int_{B_{2R}  \setminus B_{\rho}} u^{p}} \le C' R^{2N - \alpha - 2}.
\]
By the weak Harnack inequality of Lemma~\ref{P-XXX},
there exists \(c>0\) such that
\begin{equation*}
\inf_{B_{5R/3} \setminus B_{4R/3}} u\ge c \Bigl( R^{-N}\int_{B_{2R} \setminus B_R}u^{q - 1} \Bigr)^\frac{1}{q - 1},
\end{equation*}
so the assertion follows.
\end{proof}

In the limiting case \(q= \frac{\alpha}{N - 2}\) the bounds of Lemma~\ref{lemmaGreenDecay} and Proposition~\ref{propLimsupExistence-local-free} coincide but can be improved.

\begin{proposition}
\label{sharp-B1plusplus-bound}
Let \(N\ge 3\) and \(\rho > 0\). Assume that
$$q =\frac{\alpha}{N - 2} < 1.$$
If \(u\ge 0\) is a supersolution of \eqref{eqChoquard0} in \(\R^N \setminus \Bar{B}_{\rho}\),
then either \(u = 0\) in \(\R^N\setminus \Bar{B}_{\rho}\) or
\begin{equation}\label{log-new}
\liminf_{\abs{x} \to \infty} u (x) \abs{x}^{N - 2} \bigl(\log \abs{x}\bigr)^{-\frac{N - 2}{N - \alpha - 2}} >0.
\end{equation}
\end{proposition}
\begin{proof}
By Proposition~\ref{propositionLinearLowerBoundFree}, there exists \(c > 0\) such that for every \(r > 8 \rho\),
\[
  \frac{1}{r^2} \int_{B_{2 r} \setminus B_r} u \ge c \int_{B_{r} \setminus B_{4 \rho}} (I_\alpha \ast u^p) u^q.
\]
Since for every \(x \in \R^N \setminus B_{4\rho}\),
\begin{equation}
\label{eqRieszLowerBound}
  \bigl(I_\alpha \ast u^p\bigr) (x)\ge \frac{A_\alpha}{2^{N - \alpha} \abs{x}^{N - \alpha}}\int_{B_{4\rho} \setminus B_\rho} u^p,
\end{equation}
we have
\begin{equation}
\label{eqLogMultiple}
\frac{1}{r^2} \int_{B_{2 r} \setminus B_r} u
\ge \frac{A_\alpha}{4^{N - \alpha}}
\int_{8 \rho}^{r} \Bigl(\int_{B_{t} \setminus B_{t/2}} u^q\Bigr) \frac{1}{t^{1 + N - \alpha}}\,dt.
\end{equation}
On the other hand, by Harnack's inequality, there exists \(C> 0\) such that
\begin{multline*}
 \frac{1}{r^2} \int_{B_{2 r} \setminus B_r} u
\le \frac{1}{r^2} \int_{B_{2 r} \setminus B_{r/4}} u
\le C r^{N - 2} \inf_{B_{r} \setminus B_{r/2}} u \\
\le C'  \Bigl( \frac{1}{r^{N - q (N - 2)}} \int_{B_r \setminus B_{r/2}} u^q \Bigr)^\frac{1}{q}.
\end{multline*}
We have thus for some \(c' > 0\),
\begin{equation}
\label{equationIntegralLog}
 \frac{\frac{1}{t^{1 + N - \alpha}} \int_{B_r \setminus B_{r/2}} u^q}
{\Bigl( \int_{8 \rho}^{r} \Bigl(\int_{B_{t} \setminus B_{t/2}} u^q\Bigr) \frac{1}{t^{1 + N - \alpha}}\,dt \Bigr)^{1 - q}}
\ge \frac{c'}{r^{1 + q (N - 2) - \alpha}}.
\end{equation}
If \(q = \frac{\alpha}{N - 2} < 1\), the integration of this inequality with respect to \(r\) from \(8 \rho\) to \(R > 8 \rho\) yields
\[
  \Bigl( \int_{8 \rho}^{R}
             \Bigl(\int_{B_{t} \setminus B_{t/2}} u^q\Bigr) \frac{1}{t^{1 + N - \alpha}}\,dt
  \Bigr)^{\frac{N - \alpha - 2}{N - 2}}
\ge c'\frac{N - 2}{N - \alpha - 2} \log \frac{R}{8 \rho}.
\]
Recalling \eqref{eqLogMultiple}, we have for some \(c''\),
\[
 \frac{1}{r^2} \int_{B_{2 r} \setminus B_r} u \ge c''\Bigl(\log \frac{R}{8 \rho}\Bigr)^{\frac{N - 2}{N - \alpha - 2}}.
\]
We conclude by Harnack's inequality.
\end{proof}

The proof of Proposition~\ref{sharp-B1plusplus-bound} shows in fact that any supersolution of \eqref{eqnLocal} with \(q = \frac{\alpha}{N - 2}\) has the same lower bound at infinity. The result seems to be new also for the local
inequality
\[
-\Delta u (x) \ge\frac{c}{\abs{x}^{N-\alpha}}u (x)^\frac{\alpha}{N - 2}\quad\text{for \(x \in \R^N\setminus\Bar{B}_{\rho}\)},
\]
where the lower bound \eqref{log-new} on positive supersolutions is established by the same arguments as above.

In the transitional locally linear case \(\alpha=N-2\) and \(q=1\),
the Green decay bounds of Lemma~\ref{lemmaGreenDecay} can be improved.

\begin{proposition}
Let \(N\ge 3\) and \(\rho > 0\).
Assume that
$$q =\frac{\alpha}{N - 2} = 1.$$
If \(u\ge 0\) is a supersolution of \eqref{eqChoquard0} in \(\R^N \setminus \Bar{B}_{\rho}\), then
either \(u = 0\) in \(\R^N\setminus \Bar{B}_{\rho}\) or there is \(m>0\) such that
\[\liminf_{\abs{x} \to \infty} u (x) \abs{x}^{N - 2-m} >0.
\]
\end{proposition}

\begin{proof}
One follows the line of the proof of Proposition~\ref{sharp-B1plusplus-bound} until \eqref{equationIntegralLog},
whose integration now yields
\[
 \log \Bigl(\frac{1}{r^2} \int_{B_{2 r} \setminus B_r} u \Bigr) \ge m \log \frac{R}{8 \rho},
\]
for some \(m > 0\),
allowing to conclude by the weak Harnack inequality.
\end{proof}

An alternative proof consists in inferring from \eqref{eqRieszLowerBound} that
\[
  - \Delta u (x) \ge \frac{A_\alpha\int_{B_{2 \rho} \setminus B_\rho} u^p}{2^{N - \alpha} \abs{x}^{2}} u(x)  \quad\text{for \(x \in \R^N\setminus\Bar{B}_{2\rho}\)}.
\]
and deducing the assertion from the lower bound of Lemma~\ref{lemmaGreenDecay-Hardy}.

\subsection{Optimal decay.}

We are going to show that the above nonexistence results are sharp
by constructing explicit nontrivial nonnegative supersolutions.
First we prove that in  the Green decay region \eqref{eqChoquard0} admits nontrivial nonnegative supersolutions
which decay at infinity at the same rate as the Green function of \(-\Delta\).

\begin{proposition}\label{sharp-B1}
Let \(N\ge 3\) and \(\rho > 0\).
If 
\begin{align*}
  p &> \frac{\alpha}{N - 2},&
 p + q & > \frac{N + \alpha}{N - 2} &
 & \text{and} &
 q &> \frac{\alpha}{N - 2},
\end{align*}
then \eqref{eqChoquard0} admits a radial nontrivial nonnegative supersolution \(u \in C^\infty (\R^N \setminus \Bar{B}_\rho)\),
which satisfies
\[
  \limsup_{\abs{x} \to \infty} u (x) \abs{x}^{N - 2} < \infty.
\]
\end{proposition}
\begin{proof}
Fix \(\beta>0\).
For \(x \in \R^N \setminus \Bar{B}_\rho\) and \(\mu > 0\), set
\[
 u_\mu (x)=\frac{\mu}{\abs{x}^{N - 2}}\biggl(1-\frac{1}{\bigl(1 + \log \frac{\abs{x}}{\rho}\bigr)^\beta}\biggr).
\]
Then for \(x \in \R^N \setminus \Bar{B}_\rho\) we compute
\[
\begin{split}
  -\Delta u_\mu (x)& = \mu\frac{\beta}{\abs{x}^N \bigl(1 + \log \frac{\abs{x}}{\rho}\bigr)^{\beta + 1}}\Bigr((N - 2)+\frac{\beta + 1}{1 + \log \frac{\abs{x}}{\rho}}\Bigr)\\
& \ge \mu\frac{\beta (N - 2)}{\abs{x}^N \bigl(1 + \log \frac{\abs{x}}{\rho}\bigr)^{\beta + 1}} .
\end{split}
\]
On the other hand, since \(\frac{\alpha}{N - 2} < p < \frac{N}{N - 2}\), by Lemma~\ref{lemmaUpperasympt} there exists \(C > 0\) such that  for every \(x \in\R^N \setminus \Bar{B}_\rho\)
\begin{equation}
\label{ineqLowerIalphaW}
  (I_\alpha \ast u_\mu^p) (x) \le \frac{C \mu^p}{\abs{x}^{p(N - 2)-\alpha}}.
\end{equation}
Therefore for every \(x \in\R^N \setminus \Bar{B}_\rho\)
\[
 (I_\alpha \ast u_\mu^p) (x) u_\mu (x)^q \le \frac{C \mu^{p + q}}{\abs{x}^{(p + q)(N - 2)-\alpha}}.
\]
Since \(p + q > 1\) and \(p + q>\frac{N + \alpha}{N - 2}\), we conclude that \(u_\mu\) is the required supersolution for all
sufficiently small \(\mu>0\).
\smallskip

If \(p=\frac{N}{N - 2}\) then instead of \eqref{ineqLowerIalphaW} by Lemma~\ref{lemmaUpperasympt} there exists \(C > 0\) such that for every \(x \in\R^N \setminus \Bar{B}_\rho\)
\[
  (I_\alpha \ast u_\mu^p) (x) \le \frac{C \mu^p \bigl(1 + \log \frac{\abs{x}}{\rho}\bigr)}{\abs{x}^{N - \alpha}},
\]
from which we conclude as before.
\smallskip

Finally, if \(p > \frac{N}{N - 2}\) then by Lemma~\ref{lemmaUpperasympt} there exists \(C > 0\) such that for every \(x \in\R^N \setminus \Bar{B}_\rho\)
\[
  (I_\alpha \ast u_\mu^p)(x) \le \frac{C}{\abs{x}^{N - \alpha}}.
\]
from which we conclude since \(q > \frac{\alpha}{N - 2}\).
\end{proof}

Next we construct a supersolution matching decay estimate \eqref{B-1plusplus} in the transitional region \(q=\frac{\alpha}{N-2} < 1\).

\begin{proposition}\label{sharp-B1plusplus}
Let \(N\ge 3\) and \(\rho > 0\).
If
\[
 p>\frac{N}{N-2}\quad\text{and}\quad q = \frac{\alpha}{N - 2}<1,
\]
then \eqref{eqChoquard0} admits a radial nontrivial nonnegative supersolution
\(u \in C^\infty (\R^N \setminus \Bar{B}_\rho)\),
which satisfies
\[
  \limsup_{\abs{x} \to \infty} u (x) \abs{x}^{N - 2} \bigl(\log \abs{x}\bigr)^{-\frac{N - 2}{N - \alpha - 2}} < \infty.
\]
\end{proposition}

\begin{proof}
Given \(\mu > 0\), for \(x \in \R^N \setminus \Bar{B}_\rho\) we set
\[
 u_\mu (x)= \mu \frac{\bigl(\frac{1}{N - \alpha - 2} + \log \frac{\abs{x}}{\rho}\bigr)^{\frac{N - 2}{N - \alpha - 2}}}{\abs{x}^{N - 2}} .
\]
One has for every \(x \in \R^N \setminus \Bar{B}_\rho\),
\[
\begin{split}
 - \Delta u_\mu (x) & = \mu  \frac{\bigl(\frac{1}{N - \alpha - 2} + \log \frac{\abs{x}}{\rho}\bigr)^{\frac{\alpha}{N - \alpha - 2}}}{\abs{x}^N}\\
& \qquad \qquad \times \Bigl( \frac{(N-2)^2}{N - \alpha - 2} - \frac{\alpha (N - 2)}{(N - \alpha - 2)^2\bigl(\frac{1}{N - \alpha - 2} + \log \frac{\abs{x}}{\rho}\bigr)}\Bigr)\\
& \ge \mu (N - 2) \frac{\bigl(\frac{1}{N - \alpha - 2} + \log \frac{\abs{x}}{\rho}\bigr)^{\frac{\alpha}{N - \alpha - 2}}}{\abs{x}^N}.
\end{split}
\]
Since \(\frac{\alpha}{N - 2} < p < \frac{N}{N - 2}\), by Lemma~\ref{lemmaUpperasympt}
we obtain  \eqref{ineqLowerIalphaW} for every \(x \in\R^N \setminus \Bar{B}_\rho\).
Since \(p + q > 1\), \(u_\mu\) is a supersolution when \(\mu\) is small enough.
\end{proof}

Next we construct a supersolution matching decay estimate \eqref{B-1plus} in the transitional locally linear r\'egime \(\alpha=N-2\) and \(q=1\),
when the critical line \(q=\frac{\alpha}{N-2}\) belongs to the existence region.

\begin{proposition}\label{sharp-B1plus}
Let \(N\ge 3\) and \(\rho > 0\).
If
\[\alpha=N - 2,\quad p>\frac{N}{N-2}\quad\text{and}\quad q=1.\]
then for every \(m > 0\) equation \eqref{eqChoquard0} admits a radial nontrivial nonnegative supersolution
\(u \in C^\infty (\R^N \setminus \Bar{B}_\rho)\),
which satisfies
\[
  \limsup_{\abs{x} \to \infty} u (x) \abs{x}^{N - 2-m} < \infty.
\]
\end{proposition}

\begin{proof}
Without loss of generality, we can assume that \( m < N - 2 - \frac{N}{p}\).
Given \(\mu > 0\),  we set for every \(x \in \R^N \setminus \Bar{B}_\rho\)
\[
 u_\mu (x)=\frac{\mu}{\abs{x}^{N - 2 - m}}.
\]
Then we compute for every \(x \ge 0\), if \(m \in (0, N - 2)\),
\[
  -\Delta u_\mu (x)= \mu \frac{m (N - 2 - m)}{\abs{x}^2} u_\mu(x)> 0.
\]
On the other hand, if \(\mu \le N - 2 - \frac{N}{p}\) we obtain by Lemma~\ref{lemmaUpperasympt}  since \(N - \alpha=2\) for every \(x \in\R^N \setminus \Bar{B}_\rho\),
\begin{equation*}
(I_\alpha \ast u_\mu^p) (x) \le \frac{C \mu^p}{\abs{x}^{2}}.
\end{equation*}
Since \(p> 0\), we conclude that \(u_\mu\) is the required supersolution for all
sufficiently small \(\mu>0\).
\end{proof}

Finally we construct a supersolution in the sublinear decay region which matches the decay estimate \eqref{B-2}.

\begin{proposition}\label{sharp-B2}
Let \(N\ge 3\) and \(\rho > 0\).
If
\begin{equation*}
1-\frac{N - \alpha - 2}{N} p <  q < \frac{\alpha}{N - 2} < 1,
\end{equation*}
then \eqref{eqChoquard0} admits a radial nontrivial nonnegative supersolution \(u \in C^\infty (\R^N \setminus \Bar{B}_{\rho})\)
which satisfies
\[
 \limsup_{\abs{x} \to \infty} u (x) \abs{x}^{\frac{N - \alpha - 2}{1-q}} < \infty.
\]
\end{proposition}
\begin{proof}
Set for \(\mu > 0\) and \(x \in \R^N \setminus \Bar{B}_\rho\),
\[
 u_\mu (x)=\frac{\mu}{\abs{x}^\frac{N - \alpha - 2}{1-q}}.
\]
We compute
\[
  -\Delta u_\mu (x)= \mu \frac{(N - \alpha - 2)(\alpha-q(N - 2))}{(1-q)^2} \frac{1}{\abs{x}^\frac{N - 2q-\alpha}{1-q}}.
\]
Since \(q < \frac{\alpha}{N - 2} < 1\), we have \(\frac{(N - \alpha - 2)(\alpha-q(N - 2))}{(1-q)^2} > 0\).
On the other hand, \(p \frac{N - \alpha - 2}{1-q} > N\).
Hence by Lemma~\ref{lemmaUpperasympt} we obtain for every \(x \in \R^N \setminus \Bar{B}_\rho\)
\[
 (I_\alpha \ast u_\mu^p)(x) \le \frac{C\mu^p}{\abs{x}^{N - \alpha}}
\]
and thus
\[
 (I_\alpha \ast u_\mu^p)(x) u_\mu (x)^q \le \frac{C\mu^{p + q}}{\abs{x}^\frac{N - 2q-\alpha}{1-q}}.
\]
Note that \(p + q > 1+\frac{\alpha+2}{N}p > 1\), so
we conclude that \(u_\mu\) is the required supersolution for all sufficiently small \(\mu>0\).
\end{proof}

Propositions~\ref{sharp-B1},~\ref{sharp-B1plus} and~\ref{sharp-B2} confirm sharpness
of the nonexistence statements \eqref{eqA} and optimality of the decay estimates \eqref{eqB}.
This completes the proof of Theorem~\ref{Thm-free}.

\subsection{Equation with fast decay potentials.}
\label{sect-Fast}

Theorem~\ref{Thm-free} could be easily extended to the perturbed Choquard equation \eqref{eqChoquardV}
with \emph{fast decay} potentials
\[V(x)=\frac{\lambda}{\abs{x}^\gamma}\]
where \(\lambda \in \R\) and \(\gamma > 2\).
It is well known that  nontrivial nonnegative supersolutions
to the linear Schr\"o\-dinger operator \(-\Delta + V\) with fast decay potential \(V\)
have the same minimal decay rate at infinity
as the fundamental solution of the unperturbed operator \(-\Delta\),
(cf. \cite{KLS}, \cite{Pinchover}*{Section 3} or \cite{MVS}*{Lemma 3.4}).
As a consequence, we can establish a complete analogue of Theorem~\ref{Thm-free}.

\begin{theorem}\label{Thm-fast}
Let \(N\ge 3\), \(\gamma > 2\), \(\lambda\in\R\), \(0<\alpha<N\), \(p > 0\), \(q \in\R\) and \(\rho > 0\).
Then \eqref{eqChoquardFast} has a nontrivial nonnegative supersolution in \(\R^N \setminus \Bar{B}_{\rho}\)
if and only if the assumptions \eqref{eqA} hold simultaneously.\\
Moreover, if \(u\ge 0\) is a nontrivial supersolution of \eqref{eqChoquardFast} in \(\R^N \setminus \Bar{B}_{\rho}\)
then the lower bounds \eqref{eqB} hold and these bounds are optimal.
\end{theorem}

The proof of Theorem~\ref{Thm-fast} follows closely the proof of Theorem~\ref{Thm-free}.
Note only that if \(V\) is a fast decay potential then complete analogues
of Proposition~\ref{propositionLinearLowerBoundFree} and Lemma~\ref{lemmaGreenDecay}
could be established following the arguments in the proof of
Proposition~\ref{propositionLinearLowerBound}.
In addition, if \(\varphi_R\) is defined by \eqref{phiR},
then
\begin{equation*}
 \lim_{R \to \infty} \frac{1}{R^{N - 2}} \int_{{\R^N \setminus B_{\rho}}} V\abs{\varphi_R}^2=0.
\end{equation*}
The estimate of Lemma~\ref{lemmaLiminfExistence}
remains thus stable after a perturbation of \eqref{eqChoquardV} by a fast decay potential.
We omit further details.

\section{Equation with Hardy potentials.}\label{Sect-Hardy}

In this section we consider perturbed equation \eqref{eqChoquardV} with \emph{Hardy potential}
\begin{equation*}
V(x)=\frac{\nu^2-\big(\tfrac{N - 2}{2}\big)^2}{\abs{x}^2}\qquad(\nu >0).
\end{equation*}
It is well known that if \(V\) is a Hardy potential then nontrivial nonnegative supersolutions
to the linear Schr\"odinger operator \(-\Delta+V\) decay polynomially at infinity,
however the exact rate of decay depends explicitly on the value of the constant \(\nu\).
We will show that all the results of Theorem~\ref{Thm-free} could be extended with minimal suitable
modifications to Choquard's equations \eqref{eqChoquardSlow} with Hardy potentials.

\subsection{Equation with Hardy potentials.}

Using decay estimate for supersolutions to linear equations with Hardy's potential
we deduce the following extension of Theorem~\ref{Thm-free}.

\begin{theorem}\label{Thm-Hardy}
Let \(N\ge 2\), \(0<\alpha<N\), \(p > 0\), \(q \in\R\), \(\nu>0\) and \(\rho > 0\).
Then \eqref{eqChoquardHardy} has a nontrivial nonnegative supersolution in \(\R^N \setminus \Bar{B}_{\rho}\) if and only if
the following assumptions hold simultaneously:
\begin{subequations}\label{eqA-Hardy}
\begin{align}
p & > \frac{\alpha}{\frac{N - 2}{2}+\nu},\label{A-1-Hardy}\\
p + q & > 1 + \frac{\alpha+2} {\frac{N - 2}{2}+\nu},\label{A-2-Hardy}\\
q & > 1 + \frac{\alpha-(N - 2)} {\frac{N - 2}{2}+\nu} & & \text{if } \alpha>N-2 \label{A-3-Hardy},\\
q & \ge 1, & & \text{if } \alpha=N-2 \label{A-3plus-Hardy},\\
q & > 1-\frac{N - \alpha - 2}{N}p& & \text{if } \alpha<N-2\label{A-3+Hardy},\\
q & > 1 - \frac{N - \alpha - 2} {\frac{N - 2}{2}-\nu}& & \text{if \(\alpha<N-2\) and \(0<\nu<\tfrac{N - 2}{2}\).}\label{A-4-Hardy}
\end{align}
\end{subequations}
Moreover, if \(u\ge 0\) is a nontrivial supersolution of \eqref{eqChoquardHardy} in \(\R^N \setminus \Bar{B}_{\rho}\) then
\begin{subequations}
\label{eqB-Hardy}
\begin{align}
&\liminf_{\abs{x} \to \infty} u (x) \abs{x}^{\frac{N - 2}{2}+\nu}>0& & \text{if }\textstyle{q > 1 + \frac{\alpha-(N - 2)} {\frac{N - 2}{2}+\nu} > 1,}\label{B-1-Hardy}\\
\exists m>0\;:&\liminf_{\abs{x} \to \infty} u (x) \abs{x}^{\frac{N - 2}{2}+\nu-m}>0& & \text{if \(q = 1\) and \(\alpha=N - 2\),}\label{B-1plus-Hardy}\\
&\liminf_{\abs{x} \to \infty} u (x) \abs{x}^{\frac{N - 2}{2} + \nu}  \bigl(\log \abs{x}\bigr)^{-\frac{\frac{N - 2}{2} + \nu}{N - \alpha - 2}}
>0 & & \text{if \(q = 1 - \frac{N - \alpha - 2} {\frac{N - 2}{2}+\nu} < 1\),}\label{B-1plusplus-Hardy}\\
&\liminf_{\abs{x} \to \infty} u (x) \abs{x}^{\frac{N - \alpha - 2}{1-q}}>0& &\text{if \(q < 1 - \frac{N - \alpha - 2} {\frac{N - 2}{2}+\nu} < 1\).}\label{B-2-Hardy}
\end{align}
\end{subequations}
The above lower bounds are optimal.
\end{theorem}

The optimality of lower bounds \eqref{eqB-Hardy} is understood in the sense similar to that of Theorem~\ref{Thm-free}.
In particular, if \(q =1\),  \(\alpha=N-2\) and \(p>\frac{N}{\frac{N-2}{2}+\nu}\) then for every \(m > 0\) there exists a positive radial supersolution \(u\in C^\infty(\R^N\setminus\Bar{B}_\rho)\)  such that
\[
 \limsup_{\abs{x} \to \infty} u (x) \abs{x}^{\frac{N - 2}{2}+\nu-m}< \infty.
\]

\begin{figure}
\subfigure[\(\alpha \ge N - 2\)]{\includegraphics{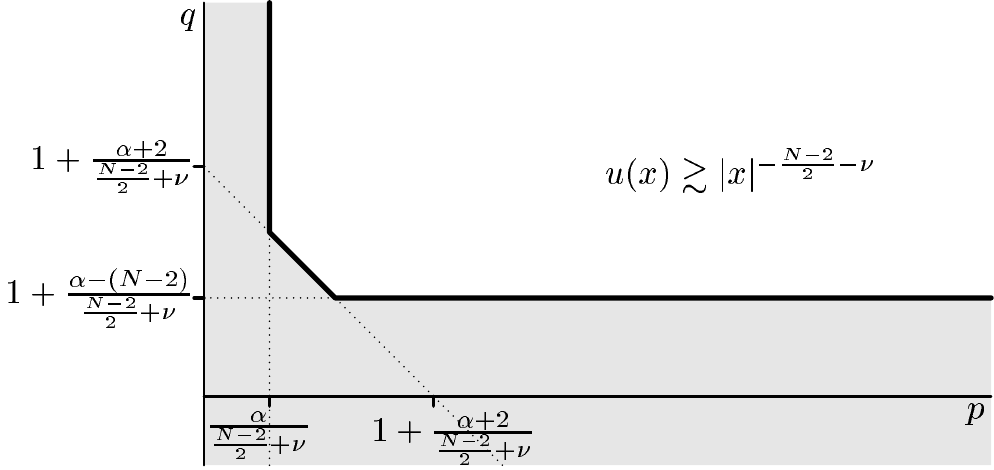}}
\subfigure[\(\alpha < N - 2\) and \(\nu \ge \frac{N - 2}{2} \)]{\includegraphics{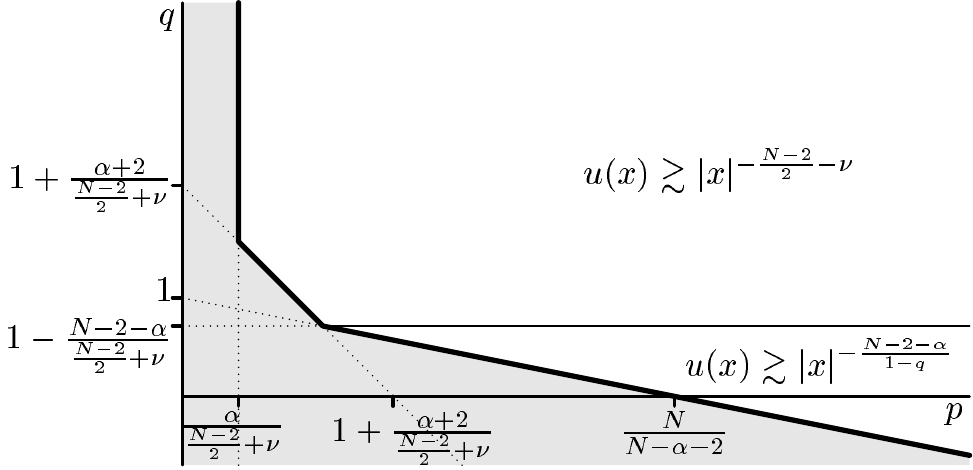}}
\subfigure[\(\alpha < N - 2\) and \(\nu < \frac{N - 2}{2} \)]{\includegraphics{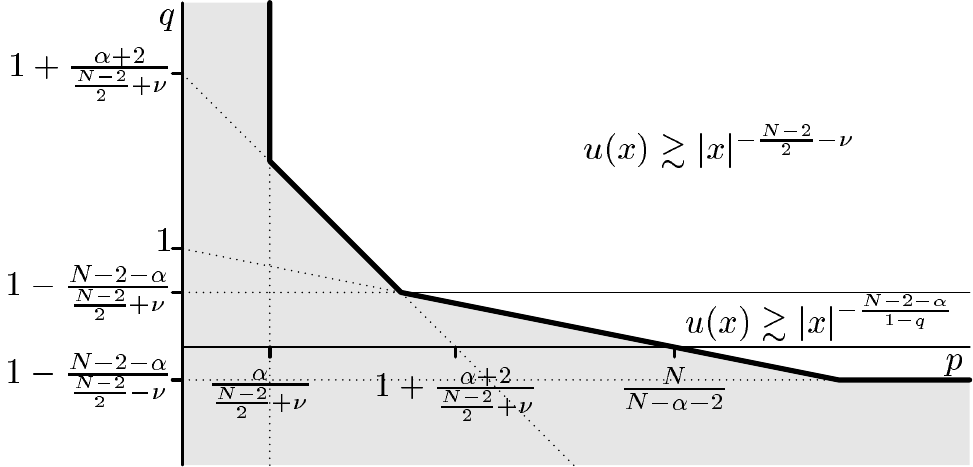}}
\caption{Existence, decay and nonexistence regions for \eqref{eqChoquardHardy} in the \((p,q)\)--plane}
\end{figure}

The nonexistence region \eqref{A-3+Hardy} as well as lower bound \eqref {B-1plus-Hardy} are
stable with respect to the variation of \(\nu\).
The nonexistence region \eqref{A-4-Hardy}, which is nonempty only for \(0<\nu<\frac{N - 2}{2}\),
is a new phenomenon compared to the free Laplacian.

\subsection{Estimates for linear equations with Hardy potentials.}
We derive several decay estimate for the auxiliary linear equations with Hardy potential.
Our first result is an integral version of the Phragmen-Lindel\"of type estimates.

\begin{proposition}
\label{propositionLinearLowerBound}
Let \(N\ge 2\), \(R,r > 0\) be such that \(R > 2 r\), \(u \in L^1_\mathrm{loc}(B_{4 R} \setminus B_{r/4})\) and \(f \in L^1_\mathrm{loc}(B_{2R} \setminus B_{r/4})\). If
\(u \ge 0\), \(f \ge 0\) and
\[
 -\Delta u (x) + \frac{\nu^2-(\frac{N - 2}{2})^2}{\abs{x}^2} u (x) \ge f(x)\quad\text{in}\quad \R^N\setminus \Bar{B}_{r/4}
\]
in the sense of distributions, then
\begin{align*}
 \frac{1}{R^{\frac{N+2}{2}+\nu}} \int_{B_{2R} \setminus B_{R}} u
 + \int_{B_{R} \setminus B_{r}} \frac{f(x)}{\abs{x}^{\frac{N - 2}{2}+\nu}} \,dx
& \le \frac{C}{r^{\frac{N+2}{2}+\nu}} \int_{B_{r} \setminus B_{r/2}} u,\\
 \frac{1}{r^{\frac{N+2}{2}-\nu}} \int_{B_{r} \setminus B_{r/2}} u
 + \int_{B_{R} \setminus B_{r}} \frac{f(x)}{\abs{x}^{\frac{N - 2}{2}-\nu}} \,dx
& \le \frac{C}{R^{\frac{N+2}{2}-\nu}} \int_{B_{2R} \setminus B_{R}} u.
\end{align*}
\end{proposition}

\begin{proof}
To prove the first inequality, choose \(\eta \in C^\infty((0, \infty))\) such that \(\eta \ge 0\), \(\eta = 1\) on \([1, 2]\) and \(\supp \eta \subset (1/2, 4)\). Let \(\theta \in C^\infty((0, \infty))\) be the solution of the Cauchy problem
\[
 \left\{
   \begin{aligned}
     -\theta''(s) & - \frac{N-1}{s} \theta'(s) + \frac{\nu^2-(\frac{N - 2}{2})^2}{s^2}\theta (s) = - \eta (s) && \text{for \(s \in (0, \infty)\) },\\
     \theta(4) & = 0,\\
     \theta'(4) & = 0.
   \end{aligned}
 \right.
\]
By the variation of parameters formula \(\theta\) can be represented for \(s \in (0, \infty)\) by
\[
 \theta (s)
  = \frac{1}{2\nu} \int_s^4 \biggl(\Bigl( \frac{\sigma}{s}\Bigr)^{\frac{N - 2}{2}+\nu} - \Bigl( \frac{\sigma}{s}\Bigr)^{\frac{N - 2}{2}-\nu} \biggr) \sigma \eta (\sigma) \, d\sigma.
\]
In particular, for \(s \in (0, 1)\),
\[
 \theta(s) = \frac{1}{2\nu {s}^{\frac{N - 2}{2}+\nu}} \int_1^4 \eta (\sigma) \sigma^{\frac{N}{2}+\nu} \, d\sigma
- \frac{1}{2\nu {s}^{\frac{N - 2}{2}-\nu}}\int_1^4 \eta (\sigma) \sigma^{\frac{N}{2}-\nu} \, d\sigma,
\]
so we conclude that \(\lim_{s\to 0}\theta(s)=+\infty\) and for every \(\theta \ge 0\) in \((0, \infty)\).
Choose \(\psi \in C^\infty(\R^N)\) such that \(\psi = 0\) on \(B_{1/2}\) and \(\psi=1\) on \(\R^N \setminus \Bar{B}_1\) and define for \(x \in B_{4R} \setminus \Bar{B}_{r/4}\)
\[
 \varphi (x) = \theta \Bigl(\frac{\abs{x}}{R}\Bigr) \psi \Bigl(\frac{x}{r}\Bigr).
\]
Since \(\varphi \in C^\infty_c(B_{4R} \setminus \Bar{B}_{r/4})\) and \(\varphi \ge 0\),
we have
\[
 \int_{B_{4R} \setminus B_{r/2}} u (x) \Bigl(-\Delta \varphi (x) + \frac{\nu^2-(\frac{N - 2}{2})^2}{\abs{x}^2} \varphi (x)\Bigr)\,dx \ge \int_{B_{4R} \setminus B_{r/2}} f \varphi.
\]
Noting that
\begin{align*}
 R^{\frac{N - 2}{2}+\nu} \int_{B_{R} \setminus B_{r}} \frac{f(x)}{\abs{x}^{\frac{N - 2}{2}+\nu}} \,dx & \le \int_{B_{4R} \setminus B_{r/2}} f \varphi,\\
 \int_{B_{4R} \setminus B_{r}} u (x) \Bigl(-\Delta \varphi (x) + \frac{\nu^2-(\frac{N - 2}{2})^2}{\abs{x}^2} \varphi (x)\Bigr)\,dx
 &= - \frac{1}{R^2} \int_{B_{4R} \setminus B_{r}} u (x) \eta\Bigl( \frac{x}{R}\Bigr)\,dx\\
& \le  -\frac{1}{R^2} \int_{B_{2R} \setminus B_{R}} u,\\
  \int_{B_{r} \setminus B_{r/2}} u (x) \Bigl(-\Delta \varphi (x) + \frac{\nu^2-(\frac{N - 2}{2})^2}{\abs{x}^2} \varphi (x)\Bigr)\,dx
 & \le C \frac{R^{\frac{N - 2}{2}+\nu}}{r^{\frac{N+2}{2}+\nu}} \int_{B_{r} \setminus B_{r/2}} u.
\end{align*}
we complete the proof of the first inequality.

To obtain the second inequality, choose \(\eta \in C^\infty((0, \infty))\) such that \(\eta = 1\) on \([1/2, 1]\) and \(\supp \eta \subset (1/4, 2)\), define \(\theta \in C^\infty((0, \infty))\) as the solution of the Cauchy problem
\[
  \left\{
   \begin{aligned}
     -\theta''(s) &- \frac{N-1}{s} \theta'(s) + \frac{\nu^2-(\frac{N - 2}{2})^2}{s^2}\theta (s) = \eta (s) && \text{for every \(s \in (0, \infty)\) },\\
     \theta(\tfrac{1}{4}) & = 0,\\
     \theta'(\tfrac{1}{4}) & = 0,
   \end{aligned}
 \right.
\]
choose
\(\psi \in C^\infty(\R^N)\) such that \(\psi = 1\) on \(B_{1/2}\) and \(\psi=0\) on \(\R^N \setminus \Bar{B}_1\), set for \(x \in B_{4R} \setminus \Bar{B}_{r/4}\),
\[
 \varphi (x) = \theta \Bigl(\frac{\abs{x}}{r}\Bigr) \psi \Bigl(\frac{x}{R}\Bigr).
\]
and conclude similarly to the above argument.
\end{proof}

Taking \(f=0\) and employing the weak Harnack inequality
(see \cite{Lieb-Loss}*{Theorem 9.10} or Lemma~\ref{P-XXX} below),
we immediately derive from Proposition~\ref{propositionLinearLowerBoundFree} the usual
pointwise lower bound and a Phragmen--Lindel\"of type integral upper bound on nontrivial nonnegative supersolutions
to linear equations with Hardy's potential in exterior domains
(see~\cite{LLM}*{Lemma 4.7} for relevant results).

\begin{lemma}\label{lemmaGreenDecay-Hardy}
Let \(N\ge 2\), \(\nu>0\), \(r > 0\) and \(u \in L^1_\mathrm{loc}(\R^N\setminus B_{r})\).
If \(u \ge 0\) and
\[
 -\Delta u (x) + \frac{\nu^2-(\frac{N - 2}{2})^2}{\abs{x}^2} u (x) \ge 0\quad\text{in}\quad \R^N\setminus \Bar{B}_{r}
\]
in the sense of distributions, then
\begin{equation*}
\limsup_{R \to \infty} R^{-\frac{N+2}{2}-\nu} \int_{B_{2R} \setminus \Bar{B}_R} u < \infty
\end{equation*}
and either \(u = 0\) in \(\R^N\setminus \Bar{B}_{\rho}\), or
\begin{equation*}
\liminf_{\abs{x} \to \infty} u (x) \abs{x}^{\frac{N - 2}{2}+\nu} > 0.
\end{equation*}
\end{lemma}

\subsection{Sketch of the proof of the nonexistence.}

\begin{proof}[Sketch of the proof of Theorem~\ref{Thm-Hardy}]
The proof follows closely the proof of Theorem~\ref{Thm-free},
with Lemma~\ref{lemmaGreenDecay-Hardy} being used instead of Lemma~\ref{lemmaGreenDecay}.
Note that when \(V\) is the Hardy potential,
and \(\varphi_R\) is defined as in \eqref{phiR},
then there exists \(C > 0\) such that for every \(R \ge \rho\),
\[
 \int_{{\R^N \setminus B_{\rho}}} V\abs{\varphi_R}^2 \le  C R^{N - 2}
\]
growth at the same rate as \eqref{scaling-grad}.
Hence, the estimates of Lemma~\ref{lemmaLiminfExistence}
are not affected by perturbation of \eqref{eqChoquardV} by the Hardy potential.
Then the proofs of nonexistence in the cases
\eqref{A-1-Hardy}, \eqref{A-2-Hardy}, \eqref{A-3-Hardy}, \eqref{A-3plus-Hardy}, \eqref{A-3+Hardy}
are carried over following the same arguments as in the proof of Theorem~\ref{Thm-free}.

The nonexistence region \eqref{A-4-Hardy}, which is nonempty only for \(0<\nu<\frac{N - 2}{2}\),
is a new phenomenon compared to the free Laplacian. This is the case when the upper bound of Lemma~\ref{lemmaGreenDecay-Hardy}
becomes relevant. The complete proof of the nonexistence of nontrivial nonnegative supersolutions below the line \eqref{A-4-Hardy}
is given in Proposition~\ref{non-A-4-Hardy} below.
\end{proof}

\begin{proposition}\label{non-A-4-Hardy}
Let \(N\ge 3\), \(N - 2<\alpha<N\), \(p>0\), \(\nu>0\) and \(\rho > 0\).
If \(q  < 1\),
\[
  \Big(\frac{N - 2}{2}-\nu\Big)(1-q) \ge N - \alpha - 2.
\]
and \(u \ge 0\) is a supersolution of \eqref{eqChoquardHardy} in \(\R^N \setminus \Bar{B}_{\rho}\), then \(u = 0\) in \(\R^N\setminus \Bar{B}_{\rho}\).
\end{proposition}

\begin{proof}
In the \emph{subcritical case}
\((\frac{N - 2}{2}-\nu)(1-q) > N - \alpha - 2\),
by Lemma~\ref{propositionLinearLowerBound},
\begin{equation}\label{u-upper}
  \int_{B_{2R} \setminus B_R} u \le C R^{\frac{N+2}{2} + \nu},
\end{equation}
by the counterpart of Lemma~\ref{lemmaLiminfExistence},
\begin{equation}\label{lower-q - 1}
  \int_{B_{2R} \setminus B_R} u^{q - 1} \le \frac{C}{\int_{B_{2 \rho \setminus \rho} u^p}} R^{2 N - \alpha - 2}.
\end{equation}
Now by H\"older's inequality, since \(q < 1\),
\[
  \Bigl(\int_{B_{2 R} \setminus B_R} 1 \Bigr)^{2 - q} \le \Bigl(\int_{B_{2R} \setminus B_R} u \Bigr)^{1 - q} \int_{B_{2R} \setminus B_R} u^{q - 1},
\]
this brings a contradiction in the subcritical case \(\big(\frac{N - 2}{2}-\nu\big)(1-q) > N - \alpha - 2\).

In the critical case \((\frac{N - 2}{2}-\nu)(1-q) = N - \alpha - 2\) we first note that
\[
 \int_{\R^N \setminus B_{\rho}} \frac{(I_\alpha \ast u^p)(x) u(x)^q}{\abs{x}^{\frac{N - 2}{2} + \nu}} \,dx \ge c \int_{\R^N \setminus B_{\rho}} \frac{u (x)^q}{\abs{x}^{N - \alpha + \frac{N - 2}{2} + \nu}} \, dx
\]
If we can prove that the integral on the right hand side diverges, then
using Proposition~\ref{propositionLinearLowerBound} we can improve the upper bound \eqref{u-upper}
and reach  the contradiction as before.
This is clearly the case when \(q = 0\).

If \(0< q < 1\) then by H\"older's inequality
\[
 \int_{B_{2R} \setminus B_R} u^q \ge \frac{\displaystyle \Bigl(\int_{B_{2 R} \setminus B_R} 1\Bigr)^{\frac{1}{1 - q} }}{\displaystyle \Bigl(\int_{B_{2 R} \setminus B_R} u^{q - 1} \Bigr)^{\frac{q}{1 - q} }}.
\]
This implies by \eqref{u-upper} that
\[
 \int_{B_{2R} \setminus B_R} u^q \ge c R^{N - \frac{(N - \alpha - 2)q}{1 - q}}.
\]
Since \((\frac{N - 2}{2}-\nu)(1-q) = N - \alpha - 2\) we conclude that
\[
 \int_{\R^N \setminus B_{\rho}} \frac{u (x)^q}{\abs{x}^{N - \alpha + \frac{N - 2}{2} + \nu}} \, dx = \infty.
\]

If \(q < 0\) then by H\"older's inequality
\[
 \int_{B_{2R} \setminus B_R} u^q \ge \frac{\displaystyle \Bigl(\int_{B_{2 R} \setminus B_R} 1\Bigr)^{1 - q} }{\displaystyle \Bigl(\int_{B_{2 R} \setminus B_R} u \Bigr)^{-q}},
\]
and thus by \eqref{lower-q - 1}
\[
 \int_{B_{2R} \setminus B_R} u^q
\ge c R^{N - (\frac{N - 2}{2} - \nu) q},
\]
so we conclude as previously.
\end{proof}

An alternative proof of Proposition~\ref{non-A-4-Hardy} is obtained by
noting that \(u\) solves
\begin{equation*}
-\Delta u(x)+V(x)u(x) \ge\frac{A_\alpha \int_{B_{2\rho} \setminus B_\rho} u^p  }{2^{N - \alpha} \abs{x}^{N - \alpha}}u (x)^q,
\end{equation*}
which does not have positive solutions in exterior domains
if \(q\le 1 - \frac{N - \alpha - 2}{\frac{N - 2}{2}-\nu}\), \(\alpha > N - 2\) and \(0<\nu<\frac{N - 2}{2}\)
(see \cite{LLM}*{Theorem 2.2}).

\subsection{Pointwise decay bounds and optimal decay.}
In the Green decay region \(q \ge 1-\frac{N - \alpha - 2}{\frac{N - 2}{2}+\nu}\),
the polynomial lower bound \eqref{B-1-Hardy} follows directly from Lemma~\ref{lemmaGreenDecay-Hardy}.
In the sublinear decay region \(q<1-\frac{N - \alpha - 2}{\frac{N - 2}{2}+\nu}\), the polynomial lower bound \eqref{B-2-Hardy}
is derived similarly to the proof of Proposition~\ref{propLimsupExistence-local-free},
using the integral estimate \eqref{lower-q - 1} and the weak Harnack inequality of Lemma~\ref{P-XXX}.

The constructions of explicit nontrivial nonnegative supersolutions showing the optimality \eqref{B-1-Hardy}, \eqref{B-1plus-Hardy} and \eqref{B-1plusplus-Hardy} follow very closely the proofs of Proposition~\ref{sharp-B1}, \ref{sharp-B1plus} and \ref{sharp-B1plusplus}.
The same applies to the existence of explicit nontrivial nonnegative supersolutions with the decay rate of \eqref{B-2-Hardy}
in the sublinear decay region in the case when \(\nu\ge\frac{N - 2}{2}\),
where the arguments repeat Proposition~\ref{sharp-B2}.
Below we construct a radial supersolution with the decay rate of \eqref{B-2-Hardy}
in the sublinear decay region under the assumption \(0<\nu<\frac{N - 2}{2}\),
when the upper bound of Lemma~\ref{lemmaGreenDecay-Hardy} becomes relevant.

\begin{proposition}
Let \(N\ge 3\) and \(\rho > 0\).
If \(\alpha<N - 2\), \(0<\nu<\frac{N - 2}{2}\),
\[
   1-\frac{N - \alpha - 2}{N}p < q,
\]
and
\[
  1-\frac{N - \alpha - 2}{\frac{N - 2}{2}-\nu}<  q < 1-\frac{N - \alpha - 2}{\frac{N - 2}{2}+\nu}.
\]
then \eqref{eqChoquardHardy}
admits a radial nontrivial nonnegative supersolution
\(u \in C^\infty (\R^N \setminus \Bar{B}_\rho)\),
which satisfies
\[
 \lim_{\abs{x} \to \infty} u(x) \abs{x}^{\frac{N - \alpha - 2}{1-q}} < \infty.
\]
\end{proposition}
\begin{proof}
For \(\mu >0\) and \(x \in \R^N \setminus \Bar{B}_\rho\), set
\[
 u_\mu (x)=\frac{\mu}{\abs{x}^\frac{N - \alpha - 2}{1-q}}.
\]
We compute
\[
  -\Delta u_\mu(x)+\frac{\nu^2-(\frac{N - 2}{2})^2}{\abs{x}^2} u_\mu(x)=
  \frac{\mu\Big(\nu-\tfrac{N - 2}{2}+\tfrac{N - \alpha - 2}{1-q}\Big)\Big(\nu+\tfrac{N - 2}{2}-\tfrac{N - \alpha - 2}{1-q}\Big)}{\abs{x}^\frac{N - 2q-\alpha}{1-q}}
\]
and we observe that if \(1-\frac{N - \alpha - 2}{\frac{N - 2}{2}-\nu}< q < 1-\frac{N - \alpha - 2}{\frac{N - 2}{2}+\nu}<1\), \(\big(\nu-\tfrac{N - 2}{2}+\tfrac{N - \alpha - 2}{1-q}\big)\big(\nu+\tfrac{N - 2}{2}-\tfrac{N - \alpha - 2}{1-q}\big)> 0\).
Since \(p \frac{N - \alpha - 2}{1-q} > N\), Lemma~\ref{lemmaUpperasympt} gives \(C > 0\) such that for every \(x \in \R^N \setminus \Bar{B}_\rho\),
\[
 (I_\alpha \ast u_\mu^p)(x)\, u_\mu(x)^q \le \frac{C\mu^{p + q}}{\abs{x}^\frac{N - 2q-\alpha}{1-q}}.
\]
Noting that \(p + q > 1+\frac{\alpha+2}{N}p > 1\), we conclude that \(u_\mu\) is the required supersolution for all sufficiently small \(\mu>0\).
\end{proof}

\section{Equation with slow decay potentials: case $q\ge 1$.}\label{sectSlowLargeq}

In this and subsequent sections we consider perturbed Choquard equation \eqref{eqChoquardV}
with \emph{slow decay} potential
\[
V(x)=\frac{\lambda^2}{\abs{x}^\gamma}.
\]
for some fixed \(\lambda>0\) and \(-\infty<\gamma<2\).
It is well known that nontrivial nonnegative supersolutions
to the linear Schr\"odinger operator \(-\Delta+V\) decay exponentially at infinity.
The decay rates for the nonlocal Choquard equation \eqref{eqChoquardV} are more complex.
We will distinguish between the exponential decay region \(q\ge 1\) and polynomial decay region \(q<1\).
Within the exponential decay region we consider separately the case \(q>1\) and the borderline \emph{locally linear} case \(q=1\).
Before doing this, we shall consider a related class of linear equation.

\subsection{Estimates for linear equations with slow decay potentials.}\label{fundamental-sect}

Here we establish sharp decay estimates for the minimal positive solutions at infinity (see Definition~\ref{App-minimal} in the Appendix~\ref{App-B}) of the linear Schr\"odinger equations with slow decay potentials.
The following result extends a fine decay estimate by S.\thinspace Agmon \cite{Agmon-2}*{Theorem 3.3}.

\begin{proposition}\label{prop-exp-lambda}
Let \(N\ge 1\), \(\gamma < 2\), \(\rho \ge 0\)  and \(W \in C^1([\rho, \infty))\) be a nonnegative function.
If
\[
  \lim_{s \to \infty} W(s)>0,
\]
and for some \(\beta > 0\),
\[
 \lim_{s \to \infty} W'(s)s^{1 + \beta} = 0,
\]
then there exists a nonnegative radial function \(H : \R^N \setminus \Bar{B}_\rho \to \R\)
such that
\begin{equation}\label{lin-exp-W}
 -\Delta H(x) + \frac{W(\abs{x})^2}{\abs{x}^\gamma} H(x) = 0\quad\text{for every \(x \in \R^N \setminus \Bar{B}_\rho\)},
\end{equation}
and
\begin{equation}\label{asymp-W}
\lim_{\abs{x}\to \infty} H(x)\abs{x}^{\frac{N - 1}{2} - \frac{\gamma}{4}}
\exp \int_{\rho}^{\abs{x}}\frac{W(s)}{s^\frac{\gamma}{2}}\,ds =1.
\end{equation}
\end{proposition}

When \(W(s)=\lambda>0\) the result is standard (see for example \cite{AmbrosettiMalchiodiRuiz}*{p.332}).
When \(\gamma = 1\) and \(\beta>\frac{1}{2}\) the result was proved by Agmon \cite{Agmon-2}*{Theorem 3.3}.
Relevant estimates also could be found in \cite{Kato}*{Theorem 2, 3}.

It follows immediately that \(H\) is a minimal positive solution at infinity of \eqref{lin-exp-W}. Indeed, the function \(U(x)=1\)
is a supersolution to \eqref{lin-exp} and the pair \(U(x)=1\) and \(H(x)\) satisfy condition \eqref{minimal-def}.

\begin{proof}[Proof of Proposition~\ref{prop-exp-lambda}.]
Since \(\gamma < 2\), we can assume without loss of generality that \(\beta < 1 - \frac{\gamma}{2}\).
Similarly to the construction in the proof of S.\thinspace Agmon \cite{Agmon-2}*{Theorem 3.3}, for \(s > \rho\) and \(\tau \in \R\) we define
\[
 \phi_\tau (s) = - \frac{N - 1 - \frac{\gamma}{2}}{2 s} - \frac{W(s)}{s^\frac{\gamma}{2}} + \frac{\tau}{s^{1 + \beta}},
\]
and for \(x \in \R^N \setminus \Bar{B}_{\rho}\)
\[
  \Phi_\tau (x) = \exp\Big(\int_{\rho}^{\abs{x}} {\phi_\tau}(s)\,ds\Big).
\]
By the chain rule,
\[
  \Delta \Phi_\tau (x) = \Big(\phi_\tau'(\abs{x})+\frac{N-1}{\abs{x}}\phi_\tau (\abs{x})+\phi_\tau (\abs{x})^2\Big)\Phi_\tau (x),
\]
hence
\[
\begin{split}
 -\Delta \Phi_\tau (x) &= \biggl(\frac{(N - 2) (N - 1 - \frac{\gamma}{2})}{2\abs{x}^2} + \frac{W'(\abs{x})}{\abs{x}^{\frac{\gamma}{2}}} + \frac{(N - 1 - \frac{\gamma}{2})W(\abs{x})}{\abs{x}^{\frac{\gamma}{2} + 1}} \\
&\qquad - \frac{(N - 2 - \beta)\tau}{\abs{x}^{2 + \beta}} - \Bigl( \frac{N - 1 - \frac{\gamma}{2}}{2 \abs{x}} + \frac{W(\abs{x})}{\abs{x}^\frac{\gamma}{2}} - \frac{\tau}{\abs{x}^{1 + \beta}} \Bigr)^2\biggr) \Phi_\tau (x)\\
&= \Bigl(- \frac{W(\abs{x})^2}{\abs{x}^\gamma} + \frac{2\tau W(\abs{x}) +\omega_\tau (x)}{\abs{x}^{1 + \frac{\gamma}{2} + \beta}}\Bigr)\Phi_\tau (x),
\end{split}
\]
where \(\omega_\tau : \R^N \setminus \Bar{B}_\rho \to \R\) is given by
\[
  \omega_\tau (x)=W'(\abs{x})\abs{x}^{1+\beta}+\frac{(N-1-\frac{\gamma}{2})(N-3+\frac{\gamma}{2})}{4\abs{x}^{1-\frac{\gamma}{2}-\beta}}
+\frac{\tau (1-\frac{\gamma}{2}+\beta)}{\abs{x}^{1-\frac{\gamma}{2}}}-\frac{\tau^2}{\abs{x}^{1-\frac{\gamma}{2}+\beta}}.
\]
Observe that \(\lim_{\abs{x} \to \infty} \omega_\tau (x) = 0\).

Choose \(\overline{\tau} > 0\) and \(\underline{\tau} < 0\).
A direct computation verifies that \(\Phi_{\underline{\tau}}\) is a subsolution
and \(\Phi_{\overline{\tau}}\) is a supersolution to equation \eqref{lin-exp-W} in the exterior of a ball \(B_{R}\),
for a sufficiently large \(R>\rho\).
Applying the classical sub and supersolutions principle,
we conclude that \eqref{lin-exp-W} admits a radial solution \(H\) in \(\R^N \setminus \Bar{B}_R\),
such that \(\Phi_{\underline{\tau}} <H< \Phi_{\overline{\tau}} \). Since \(\limsup_{\abs{x} \to \infty} \frac{\Phi_{\overline{\tau}}(x)}{\Phi_{\underline{\tau}}(x)} < \infty\),
we deduce that up to multiplication by a constant \(H\) has the required asymptotic.

Since \(H\) is radial, it can be extended to a positive solution on \(\R^N \setminus \Bar{B}_\rho\).
Indeed, otherwise \(H\) would vanish on a sphere \(\partial B_r\) with \(r>\rho\).
Since \(\lim_{\abs{x} \to \infty} H(x) = 0\) and \(V \ge 0\),
this would imply by the maximum principle that \(H = 0\) on \(\R^N \setminus \Bar{B}_r\).
\end{proof}

\begin{remark}\label{r-exp-lambda}
We apply Proposition~\ref{prop-exp-lambda} in order to understand the rate of decay of positive solutions
of the linear equation
\begin{align}\label{lin-exp}
-\Delta u (x) +\frac{\lambda^2}{\abs{x}^\gamma}u(x) & =\frac{m}{\abs{x}^{\sigma}}u(x) &
&\text{for every \(x \in \R^N\setminus \Bar{B}_\rho\)}.
\end{align}
where \(-\infty<\gamma<2\), \(\lambda>0\), \(\sigma>\gamma\), \(m \le \lambda^2 \rho^{\gamma - \sigma}\)
(see~\cite{Agmon-2}*{Theorem 3.3} for the case \(\gamma=0\)).
Equation \eqref{lin-exp} appears as a localization of Choquard's equation \eqref{eqChoquardSlow} in the half-linear case \(q=1\).
Let \(H(x)\) be a minimal positive solution at infinity of \eqref{lin-exp}, as constructed in Proposition~\ref{prop-exp-lambda}.
The asymptotics of \(H\) are related to the asymptotics as \(r \to \infty\) infinity of the function
\[
\psi(r):=\int_{\rho_0}^r \sqrt{\frac{\lambda^2}{s^{\gamma}}-\frac{m}{s^{\sigma}}\;} \,ds
= \frac{\lambda}{1 - \gamma/2} \int_{\rho_0^{1 - \gamma/2}}^{r^{1 - \gamma/2}} \sqrt{1-\frac{m}{\lambda^2 t^{\frac{\sigma-\gamma}{1-\gamma/2}}}\;} \,ds,
\]
where \(\rho_0>0\) is chosen so that
\(\frac{\lambda^2}{\rho_0^{\gamma}}-\frac{m}{\rho_0^{\sigma}} \ge 0\).
By the Taylor expansion of the square root, we have for every \(k \in \N\),
\[
 \sqrt{1-\frac{m}{\lambda^2 t^{\frac{\sigma-\gamma}{1-\gamma/2}}}\;}
= 1 - \sum_{j=1}^k \frac{1}{2 j - 1} \binom{2 j}{j} \biggl( \frac{m}{(2\lambda)^2 t^{\frac{\sigma-\gamma}{1-\gamma/2}}}\biggr)^j + O \biggl( \frac{1}{t^{(k + 1) \frac{\sigma-\gamma}{1-\gamma/2}}}\biggr).
\]
If \(k <\frac{1-\gamma/2}{\sigma-\gamma} < k + 1\), then
\[
  \psi (r) = \lambda \biggl(\frac{r^{1 - \gamma/2}}{1 - \gamma/2} - \sum_{j = 1}^k \frac{r^{1 - \gamma/2 - j (\sigma - \gamma)}}{(1 -  \gamma/2 - j(\sigma - \gamma))(2 j - 1)} \binom{2 j}{j} \frac{m^j}{(2\lambda)^{2 j}} \biggr) + O (1).
\]
whereas if \(k = \frac{1-\gamma/2}{\sigma-\gamma}\),
\begin{multline*}
  \psi (r) = \lambda \biggl(\frac{r^{1 - \gamma/2}}{1 - \gamma/2} - \sum_{j = 1}^k \frac{r^{1 - \gamma/2 - j (\sigma - \gamma)}} {(1 -  \gamma/2 - j(\sigma - \gamma))(2 j - 1)} \binom{2 j}{j} \frac{m^j}{(2\lambda)^{2 j}} \\
- \frac{1}{2 k - 1} \binom{2 k}{k} \frac{m^k}{(2\lambda)^{2 k}} \log r \biggr)+ O (1).
\end{multline*}
In particular, if \(\sigma > 1 + \frac{\gamma}{2}\), then
\[
 \lim_{\abs{x}\to \infty} H(x)\abs{x}^{\frac{N - 1}{2} - \frac{\gamma}{4}}
\exp\big(\tfrac{2\lambda}{2-\gamma}\abs{x}^{1-\frac{\gamma}{2}}\big) \in (0, \infty);
\]
if \(\sigma=1+\textstyle{\frac{\gamma}{2}}\), then
\[
\lim_{\abs{x}\to \infty}\textstyle{H(x)\abs{x}^{{\frac{N - 1}{2} - \frac{\gamma}{4}}-\frac{m}{2 \lambda}}
\exp\big(\frac{2\lambda}{2-\gamma}\abs{x}^{1-\frac{\gamma}{2}}\big)} \in (0, \infty);
\]
if \(\frac{1}{2} + \frac{3\gamma}{4} <\sigma<\textstyle{1+\frac{\gamma}{2}}\), then
\[
\lim_{\abs{x}\to \infty}\textstyle{H(x)\abs{x}^{\frac{N - 1}{2} - \frac{\gamma}{4}}
\exp\big(\frac{2\lambda}{2-\gamma}\abs{x}^{1-\frac{\gamma}{2}}\!-\!\frac{m}{\lambda(2+\gamma-2\sigma)}\abs{x}^{1+\frac{\gamma}{2}-\sigma}\big) \in (0, \infty)};
\]
and if \(\sigma=\frac{1}{2} + \frac{3\gamma}{4}\), then
\[
\lim_{\abs{x}\to \infty}\textstyle{H(x)\abs{x}^{{\frac{N - 1}{2} - \frac{\gamma}{4}}-\frac{m^2}{8\lambda^3}}
\exp\big(\frac{2\lambda}{2-\gamma}\abs{x}^{1-\frac{\gamma}{2}}\!-\!\frac{m}{\lambda(2+\gamma-2\sigma)}\abs{x}^{1+\frac{\gamma}{2}-\sigma}\big) \in (0, \infty)}.
\]
\end{remark}

\subsection{Exponential decay region $q>1$: proof of Theorem~\ref{Thm-exp}.}

First we establish the lower bound \eqref{exp-1+} of Theorem~\ref{Thm-exp}.

\begin{proposition}
\label{lemma-exp-decay}
Let \(N\ge 1\),  \(\gamma<2\), \(\lambda > 0\), \(0 < \alpha < N\), \(p > 0\), \(q>1\) and \(\rho > 0\).
If \(u\ge 0\) is a nontrivial supersolution of \eqref{eqChoquardSlow} in \(\R^N \setminus \Bar{B}_{\rho}\), then
\begin{equation*}
\liminf_{\abs{x}\to \infty}u (x)\abs{x}^{\frac{N - 1}{2} - \frac{\gamma}{4}}\exp\Big(\frac{2\lambda}{2-\gamma}\abs{x}^{1-\frac{\gamma}{2}}\Big)>0.
\end{equation*}
\end{proposition}
\begin{proof}
Simply note that \(u\) is a supersolution to the linear equation
$$-\Delta u (x)+\frac{\lambda^2}{\abs{x}^\gamma}u(x)=0\quad\text{for every \( x \in \R^N\setminus \Bar{B}_\rho\)},
$$
so the assertion follows
by the comparison principle of Proposition~\ref{p-minimal} and the standard decay estimates for the linear equation, see \cite{AmbrosettiMalchiodiRuiz}*{p.\thinspace 332}, or Proposition~\ref{prop-exp-lambda} with \(W(s)=\lambda\) above.
\end{proof}

Next we construct a supersolution to \eqref{eqChoquardSlow} which justifies optimality
of the lower bound of Proposition~\ref{lemma-exp-decay} and thus completes the proof of Theorem~\ref{Thm-exp}.

\begin{proposition}
Let \(N\ge 1\),  \(\gamma<2\), \(\lambda > 0\), \(0 < \alpha < N\), \(p > 0\), \(q>1\) and \(\rho > 0\).
Then \eqref{eqChoquardSlow} admits a radial nontrivial nonnegative supersolution
\(u \in C^\infty (\R^N \setminus\Bar{B}_\rho)\),
such that
\begin{equation*}
\limsup_{\abs{x}\to \infty} u (x)\abs{x}^{\frac{N - 1}{2} - \frac{\gamma}{4}}\exp\Big(\frac{2\lambda}{2-\gamma}\abs{x}^{1-\frac{\gamma}{2}}\Big)<\infty.
\end{equation*}
\end{proposition}

\begin{proof}
Choose \(m>0\) such that \(m < \lambda^2 \rho^{2 - \gamma}\).
Let
\(H \in C^2(\R^N \setminus \Bar{B}_\rho)\) be the positive solution of the equation
\begin{align*}
  -\Delta H(x) + \frac{\lambda^2}{\abs{x}^\gamma}H(x) &= \frac{m}{\abs{x}^2} H(x)&
&\text{for every \(x \in \R^N\setminus\Bar{B}_\rho\)}
\end{align*}
such that
\[
 \lim_{\abs{x}\to \infty} H(x) \abs{x}^{\frac{N-1}{2} -\frac{\gamma}{4}} \exp \Big(\frac{2\lambda}{2-\gamma}\abs{x}^{1-\frac{\gamma}{2}}\Big) = 1
\]
given by Proposition~\ref{prop-exp-lambda}.
Set \(u_\mu = \mu H\) for \(\mu > 0\).
By Lemma~\ref{lemmaUpperasympt} there exists \(C > 0\) such that for every \(x \in \R^N \setminus \Bar{B}_\rho\),
\[
  (I_\alpha\ast u_\mu^p)(x)\le \frac{C \mu^p }{\abs{x}^{N - \alpha}}.
\]
Since \(q > 1\),  there exists a sufficiently small \(\bar\mu>0\) such that for every \(x \in \R^N \setminus \Bar{B}_\rho\),
\[
  \frac{C\bar\mu^{p + q - 1}}{\abs{x}^{N - \alpha}}H(x)^{q - 1}
  \le\frac{m}{\abs{x}^2}.
\]
Hence, if \(\mu < \bar{\mu}\),
\begin{equation*}
\begin{split}
-\Delta u_\mu (x)+\frac{\lambda^2}{\abs{x}^\gamma} u_\mu (x)&= \frac{m}{\abs{x}^{2}} u_\mu (x)\\
&\ge \frac{C\mu^{p + q - 1}}{\abs{x}^{N - \alpha}}H(x)^{q - 1} u_\mu (x)\ge \bigl(I_\alpha\ast u _\mu^p(x)\bigr)u_\mu^q(x),
\end{split}
\end{equation*}
that is, \(u_\mu\) is the required supersolution of \eqref{eqChoquardSlow}.
\end{proof}

\subsection{Borderline region $q=1$: proof of Theorem~\ref{Thm-exp+}.}

Our next step is to explore nonlocal positivity principle of Proposition~\ref{propositionGroundState}
in order to obtain a slow decay counterpart of Lemma~\ref{lemmaLiminfExistence}.

\begin{lemma}
\label{lemmaLiminfExistence-slow}
Let \(N\ge 1\), \(\gamma < 2\), \(\lambda > 0\), \(0 < \alpha < N\), \(p>0\), \(q\in\R\) and \(\rho > 0\).
If \(u \ge 0\) is a supersolution of \eqref{eqChoquardSlow} in \(\R^N \setminus \Bar{B}_{\rho}\), then either \(u = 0\) almost everywhere in \(\R^N \setminus \Bar{B}_{\rho}\) or
\(u = 0\) almost everywhere in \(\R^N \setminus \Bar{B}_{\rho}\), \(u^{q - 1} \in L^1_\mathrm{loc} (\R^N \setminus \Bar{B}_{\rho})\) and for
every \(R>2 \rho\),
\[
   \Bigl(\int_{B_{2R}  \setminus B_{\rho}} u^{p} \Bigr) \Bigl(\int_{B_{2R} \setminus B_R} u^{q - 1}\Bigr) \le C R^{2N - \alpha - \gamma}.
\]
\end{lemma}
\begin{proof}
Apply Lemma~\ref{propositionGroundState} to the family of test functions \eqref{phiR}
and note that if \(V\) is a slow decay potential then
\begin{align*}
 \int_{{\R^N \setminus B_{\rho}}} V\abs{\varphi_R}^2&= O\big(R^{N - \gamma}\big)&
&\text{as \(R\to\infty\)}.
\end{align*}
Then proceed as in the proof of Lemma~\ref{lemmaLiminfExistence}.
\end{proof}

As an immediate consequence of the upper bound of Lemma~\ref{lemmaLiminfExistence-slow}
we prove the nonexistence statement \eqref{nonexist-exp+} of Theorem~\ref{Thm-exp+}.

\begin{proposition}
Let \(N\ge 1\),  \(\alpha\in(N-2,N)\), \(N - \alpha<\gamma<2\), \(\lambda > 0\), \(p>0\), \(q=1\) and \(\rho > 0\).
If \(u \ge 0\) is a supersolution of \eqref{eqChoquardSlow} in \(\R^N \setminus \Bar{B}_{\rho}\) then \(u = 0\) in \(\R^N\setminus \Bar{B}_{\rho}\).
\end{proposition}
\begin{proof}
Assume that \(u > 0\) almost everywhere.
Since \(q=1\), by Lemma~\ref{lemmaLiminfExistence-slow} for any \(R>\rho\) we have
\[
R^{2N - \alpha-\gamma}\ge C \Bigl(\int_{B_R \setminus B_{\rho}} u^{p}\Bigr)\Bigl(\int_{B_{2R} \setminus B_R} 1\Bigr).
\]
This brings a contradiction if \(N - \gamma<\alpha\).
\end{proof}

To understand the existence and asymptotic properties of positive supersolution
of \eqref{eqChoquardSlow} when \(q=1\) and \(\gamma\le N - \alpha\),
we consider the local equation
\begin{align}\label{fundamental-2}
-\Delta u (x) + \frac{\lambda^2}{\abs{x}^\gamma} u (x) &=\frac{m}{\abs{x}^{N - \alpha}} u (x)&
\text{for \(x \in \R^N\setminus\Bar{B}_\rho\)}.
\end{align}
Clearly, if \(\gamma= N - \alpha\) and \(m>\lambda^2\) then \eqref{fundamental-2}
has no positive supersolutions, by the local positivity principle of Proposition~\ref{propositionLocalGroundstate}.
If \(\gamma= N - \alpha\) and \(m=\lambda^2\) then \eqref{fundamental-2} simplifies to \(-\Delta u=0\) in \(\R^N\setminus\Bar{B}_\rho\).
In all other cases, that is if \(\gamma= N - \alpha\) and \(m<\lambda^2\) or if  \(\gamma<N - \alpha\) and \(m\in\R\),
equation \eqref{fundamental-2} in \(\R^N\setminus\Bar{B}_\rho\) with a sufficiently large \(\rho>0\) admits by Proposition~\ref{prop-exp-lambda} an exponentially decaying
minimal positive solution at infinity with the decay rate given by \eqref{asymp-W}.
The result below shows that the decay of positive supersolutions of Choquard's equation \eqref{eqChoquardSlow}
is controlled by the decay rate of minimal positive solution at infinity of \eqref{fundamental-2}.

\begin{proposition}
Let \(N\ge 1\), \(\gamma < 2\), \(\lambda > 0\), \(0 < \alpha < N\), \(p>0\), \(q=1\) and \(\rho > 0\). If \( \gamma \le N - \alpha\) and if \(u\ge 0\) is a supersolution of \eqref{eqChoquardSlow} in \(\R^N \setminus \Bar{B}_{\rho}\)
then either \(u = 0\) almost everywhere or there exists \(\rho_m > 0\) and
\(m > 0\) such that \(m < \lambda^2 \rho_m^{N - \alpha - \gamma}\) and
\begin{equation*}
\liminf_{\abs{x}\to \infty} u(x)\abs{x}^{\frac{N - 1}{2} - \frac{\gamma}{4}}
\exp \int_{\rho_m}^{\abs{x}}\sqrt{\frac{\lambda^2}{s^{\gamma}}-\frac{m}{s^{N - \alpha}}\;} \,ds >0.
\end{equation*}
\end{proposition}

\begin{proof}
Let \(u\ge 0\) be a nontrivial supersolution of \eqref{eqChoquardSlow} in \(\R^N\setminus \Bar{B}_\rho\).
One has then for every \(x \in \R^N \setminus \Bar{B}_{2 \rho}\),
\[
 (I_\alpha\ast u^p)(x)\ge\frac{A_\alpha}{2^{N - \alpha} \abs{x}^{N - \alpha}} \int_{B_{2 \rho} \setminus B_\rho} u^p,
\]
Hence, for every \(x \in \R^N\setminus\Bar{B}_{2 \rho}\),
\[
  (I_\alpha\ast u^p)(x)\ge\frac{m}{\abs{x}^{N - \alpha}},
\]
with \(m = \frac{A_\alpha}{2^{N - \alpha}} \int_{B_{2 \rho} \setminus B_\rho} u^p > 0\).
By the comparison principle of Proposition~\ref{p-minimal} and decay estimate \eqref{asymp-W} we conclude that
\(u\) satisfies the announced asymptotics.
\end{proof}

The presence of the correction term related to the size of the constants \(m\) in the asymptotic is essential. We prove that for any admissible \(m>0\) there is a supersolution to \eqref{eqChoquardSlow} for which the lower bound cannot be improved. This justifies optimality of the lower bound and thus completes the proof of Theorem~\ref{Thm-exp+}.

\begin{proposition}\label{prop-exp-equiv}
Let \(N\ge 1\), \(\gamma < 2\), \(\lambda > 0\), \(0 < \alpha < N\), \(p>0\), \(q=1\) and \(\rho > 0\).
If \(\gamma \le N - \alpha\), then for every \(\rho_m > 0\) and \(0 < m < \lambda^2 \rho_m^{N - \alpha - \gamma}\),
there exists a radial nontrivial nonnegative supersolution of \eqref{eqChoquardSlow}
\(u \in C^\infty (\R^N \setminus\Bar{B}_{\rho_m})\) such that
\begin{equation*}
\limsup_{\abs{x}\to \infty} u (x)\abs{x}^{\frac{N - 1}{2} - \frac{\gamma}{4}}
\exp \int_{\rho_m}^{\abs{x}}\sqrt{\frac{\lambda^2}{s^{\gamma}}-\frac{m}{s^{N - \alpha}}\;} \,ds < \infty.
\end{equation*}
\end{proposition}

\begin{proof}
Without loss of generality, we can assume that \(m < \lambda^2 \rho^{N - \alpha - \gamma}\).
Let \(H_m\) be the solution of \eqref{fundamental-2} given by Proposition~\ref{prop-exp-lambda}.
By Lemma~\ref{lemmaUpperasympt}, there exists \(C > 0\) such that for every \(x \in \R^N \setminus \Bar{B}_{\rho}\),
\[
  (I_\alpha\ast H_m^p)(x)\le\frac{C}{\abs{x}^{N - \alpha}},
\]
for some \(C > 0\).
Set \(u_\mu= \mu H_m\). Then if \(0<\mu<\left(\frac{m}{C}\right)^{1/p}\), one has for every \(x \in \R^N \setminus \Bar{B}_{\rho}\)
\[
  -\Delta u_\mu(x)+\frac{\lambda^2}{\abs{x}^{\gamma}}u_\mu(x)=\frac{m}{\abs{x}^{N - \alpha}}u_\mu(x)\ge \bigl(I_\alpha\ast u_\mu^p(x)\bigr)\,u_\mu(x),
\]
that is, \(u_\mu\) is the required solution of \eqref{eqChoquardSlow}.
\end{proof}

\section{Equation with slow decay potentials: case $q < 1$.}
\label{sectSlowSmallq}

\subsection{Nonexistence.}
We shall establish two qualitatively different nonexistence results,
first in the region where \(q<1\) and \(p + q\ge1\), and second in the sublinear region \(p + q<1\).
We will see that the values \(\gamma=N - \alpha\) and  \(\gamma=-\alpha\) represent the critical decay rate thresholds where different mechanisms are responsible for the existence and nonexistence of positive solutions of \eqref{eqChoquardSlow}.
First we prove nonexistence statements of Theorems~\ref{Thm-slow} and~\ref{Thm-slow-moderate}.

\begin{proposition}\label{nonSlowDecayq}
Let \(N\ge 1\), \(0 < \alpha < N\), \(\gamma < 2\), \(\lambda > 0\), \(q<1\), \(p > 0\) and \(\rho > 0\).
If
\begin{align*}
  p + q &> 1 &
& \text{ and }&
q &\le 1 - \frac{N - \alpha - \gamma}{N}p,\\
\intertext{or}
p + q &= 1&
& \text{ and }&
\gamma &> - \alpha,
\end{align*}
and \(u \ge 0\) is a supersolution of \eqref{eqChoquardSlow} in \(\R^N \setminus \Bar{B}_{\rho}\), then \(u = 0\) in \(\R^N\setminus \Bar{B}_{\rho}\).
\end{proposition}

The statement simplifies for some values of \(\gamma\): if \(\gamma \ge N - \alpha\), then there is no nontrivial solution for \(q < 1\) whereas if \(\gamma \le - \alpha\) there is no nontrivial solution for \(q <1\) and \(p + q > 1\).

\begin{proof}[Proof of Proposition~\ref{nonSlowDecayq}]
Let \(R>\rho\).
Since \(q < 1\), \(p > 0\), by H\"older's inequality we have
\begin{equation}
\label{eqTSR1}
\begin{split}
  \int_{B_{2R} \setminus B_R} 1 &\le \Bigl( \int_{B_{2R} \setminus B_R} u^p \Bigr)^\frac{1 - q}{p + 1 - q}\Bigl( \int_{B_{2R} \setminus B_R} u^{q - 1} \Bigr)^\frac{p}{p + 1 - q}\\
  &\le \left(\Bigl( \int_{B_{2R} \setminus B_R} u^p \Bigr)\Bigl( \int_{B_{2R} \setminus B_R} u^{q - 1} \Bigr)\right)^\frac{1 - q}{p + 1 - q}\Bigl( \int_{B_{2R} \setminus B_R} u^{q - 1} \Bigr)^\frac{p + q - 1}{p + 1 - q}.
\end{split}
\end{equation}
By Lemma~\ref{lemmaLiminfExistence-slow},
on the one hand
\begin{equation}
\label{eqTSR2}
 \Bigl(\int_{B_{2R} \setminus B_R} u^p \Bigr)\Bigl( \int_{B_{2R} \setminus B_R} u^{q - 1} \Bigr) \le C R^{2N - \alpha - \gamma}
\end{equation}
and on the other hand
\begin{equation}
\label{eqTSR3}
 \int_{B_{2R} \setminus B_R} u^{q - 1} \le C \frac{R^{2N - \alpha - \gamma}}{\displaystyle \int_{B_{2R}  \setminus B_\rho} u^{p}} \le C \frac{R^{2N - \alpha - \gamma}}{\displaystyle \int_{B_{2 \rho}  \setminus B_\rho} u^{p}}.
\end{equation}
This brings a contradiction when \(p + q\ge 1\) and \(q < 1 - \frac{N - \alpha - \gamma}{N}p\).

In the \emph{critical case} \(q = 1 - \frac{N - \alpha - \gamma}{N}p\),
by \eqref{eqTSR1}, \eqref{eqTSR2} and \eqref{eqTSR3}, there exists \(c > 0\) such that
\[
 \int_{B_{2R} \setminus B_R} u^p \ge
 \frac{\Bigl(\displaystyle{\int_{B_{2R} \setminus B_R} 1}\Bigr)^{1 + \frac{p}{1 - q}}} {\displaystyle{\Bigl(\int_{B_{2R} \setminus B_R} u^{1 - q} }\Bigr)^\frac{p}{1 - q} } \ge c R^{N-\frac{p}{1 - q} (N + \alpha + \gamma)}= c.
\]
Therefore,
\[
 \lim_{R \to \infty} \int_{B_{2R}  \setminus B_\rho} u^{p} = \infty
\]
and \eqref{eqTSR3} can be improved to
\[
 \lim_{R \to \infty} R^{\alpha + \gamma - 2 N} \int_{B_{2R} \setminus B_R} u^{q - 1} = 0,
\]
so that we can conclude as previously.
\end{proof}

In the sublinear region \(p + q<1\), described in Theorem~\ref{Thm-slow-fast}, the nonexistence r\'egime is different.

\begin{proposition}\label{nonSlowDecaypq}
Let \(N\ge 1\), \(0 < \alpha < N\), \(\gamma < 2\), \(\lambda > 0\), \(p>0\), \(q < 1\) and \(\rho > 0\).
If \(p + q<1\),
\[
  q \le 1 + \frac{\gamma}{\alpha}p,
\]
and \(u \ge 0\) is a supersolution of \eqref{eqChoquardSlow} in \(\R^N \setminus \Bar{B}_{\rho}\) then \(u = 0\) in \(\R^N\setminus \Bar{B}_{\rho}\).
\end{proposition}

If \(\gamma \ge -\alpha\), this proposition merely states that there is no nontrivial supersolution for \(p + q < 1\).

\begin{proof}[Proof of Proposition~\ref{nonSlowDecaypq}]
Let \(R>\rho\).
Since \(p + q  < 1\), by the H\"older inequality we have
\begin{equation}
\label{eqAUIE}
\begin{split}
  \int_{B_{2R} \setminus B_R} 1 &\le \Bigl( \int_{B_{2R} \setminus B_R} u^p \Bigr)^\frac{1 - q}{p + 1 - q}\Bigl( \int_{B_{2R} \setminus B_R} u^{q - 1} \Bigr)^\frac{p}{p + 1 - q}\\
  &\le \Bigl( \int_{B_{2R} \setminus B_R} u^p \Bigr)^\frac{1 - q - p}{p + 1 - q}
  \left(\Bigl( \int_{B_{2R} \setminus B_R} u^{p} \Bigr)\Bigl( \int_{B_{2R} \setminus B_R} u^{q - 1} \Bigr)\right)^\frac{p}{p + 1 - q}.
\end{split}
\end{equation}
By \eqref{equLp} and by Lebesgue's dominated convergence theorem, one has
\[
  \lim_{R \to \infty} R^{N - \alpha} \int_{B_{2R} \setminus B_R} u^p = 0,
\]
hence by \eqref{eqAUIE}
\[
 \lim_{R \to \infty} R^{-(2 N - \alpha) - \alpha \frac{1 - q}{p}} \Bigl( \int_{B_{2R} \setminus B_R} u^{p} \Bigr)\Bigl( \int_{B_{2R} \setminus B_R} u^{q - 1} \Bigr) = \infty.
\]
This brings a contradiction with Lemma~\ref{lemmaLiminfExistence-slow} when \(q \le 1 + \frac{\gamma}{\alpha}p\).
\end{proof}

\subsection{Pointwise decay bounds.}
In the sublinear decay region \(q<1\) the exponential decay estimates of Proposition~\ref{prop-exp-lambda} are no longer relevant.
In fact, if \(q<1\) then nontrivial nonnegative supersolution of \eqref{eqChoquardSlow} in \(\R^N \setminus \Bar{B}_{\rho}\)
decay at a polynomial rate. We prove this in several steps.
We first observe that the decay of \(u\) is related to the behavior of the integral of \(u^p\) on large balls.

\begin{lemma}
\label{lemmaLimsupExistence-local}
Let \(N\ge 1\), \(\gamma<2\), \(\lambda > 0\), \(0 < \alpha < N\), \(p>0\), \(q<1\) and \(\rho > 0\).
If \(u\ge 0\) is a nontrivial supersolution of \eqref{eqChoquardSlow} in \(\R^N \setminus \Bar{B}_{\rho}\), then
\begin{equation*}
 \liminf_{\abs{x} \to \infty} \frac{\abs{x}^{N - \alpha - \gamma} }{\displaystyle \int_{B_{2\abs{x}} \setminus B_\rho} u^p} u (x)^{1 - q} > 0.
\end{equation*}
\end{lemma}

\begin{proof}
Choose a function \(\varphi \in C^\infty_c (\R^N)\) such that \(\varphi = 1\) on \(B_{1/2}\) and \(\supp \varphi \subset B_1\).
For \(x \in \R^N \setminus B_r\) define
\[
  \psi_x (y) = \varphi \Bigl( \frac{y - x}{\abs{x}^\frac{\gamma}{2}} \Bigr).
\]
If \(\abs{x}\) is large enough, one has \(\supp \psi_x \cap B_\rho = \emptyset\).
Note that \(\varphi \ge 1 \) on \(B_{\abs{x}^\gamma / 2}(x)\) and that
\[
 \int_{\R^N \setminus B_r} \abs{\nabla \psi_x (y)}^2 + \frac{1}{\abs{x}^\gamma} \abs{\psi_x (y)}^2\,dy
 \le C \abs{x}^{\gamma \frac{N - 2}{2} }.
\]
By Proposition~\ref{propositionGroundState}, one has
\[
  \frac{1}{\abs{x}^{\gamma \frac{N - 2}{2} }} \int_{B_{\abs{x}^\gamma / 2}(x)} u^{q - 1} \le C \frac{\abs{x}^{N - \alpha - \gamma} }{\displaystyle \int_{B_{2\abs{x}} \setminus B_r} u^p}.
\]
Since \(q < 1\), by the weak Harnack inequality of Proposition~\ref{P-XXX},
applied in the ball \(B_{\abs{x}^\gamma / 2}(x)\), we conclude that
\[
  u (x)^{1 - q} \ge \frac{c}{\abs{x}^{N - \alpha - \gamma}} \int_{B_{2\abs{x}} \setminus B_r} u^p,
\]
where the constant \(c > 0\) does not depend on \(x \in \R^N \setminus B_r\).
\end{proof}

As first consequence, we have the following asymptotics:

\begin{proposition}
\label{propositionSlowDecayq}
Let \(N\ge 1\), \(\gamma<2\), \(\lambda > 0\), \(0 < \alpha < N\), \(p>0\), \(q<1\) and \(\rho > 0\).
If \(u\ge 0\) is a nontrivial supersolution of \eqref{eqChoquardSlow} in \(\R^N \setminus \Bar{B}_{\rho}\), then
\begin{equation*}
 \liminf_{\abs{x} \to \infty} \abs{x}^{\frac{N - \alpha - \gamma}{1-q}} u (x) > 0.
\end{equation*}
\end{proposition}
\begin{proof}
Apply Lemma~\ref{lemmaLimsupExistence-local} and note that \(\int_{B_{2\abs{x}} \setminus B_\rho} u^p \ge \int_{B_{2\rho} \setminus B_\rho} u^p > 0\)
when \(\abs{x} \ge \rho\).
\end{proof}

In the fully sublinear case \(p + q<1\) the lower bound of Proposition~\ref{propositionSlowDecayq} can be further improved.

\begin{lemma}
\label{lemmaRieszdivergence}
Let \(N\ge 1\), \(\gamma<2\), \(\lambda > 0\), \(0 < \alpha < N\), \(p>0\), \(p + q<1\) and \(\rho > 0\).
Let \(u\ge 0\) be a nontrivial supersolution of \eqref{eqChoquardSlow} in \(\R^N \setminus \Bar{B}_{\rho}\).
\begin{enumerate}[(i)]
\item If \(q < 1 - \frac{N - \alpha - \gamma}{N}p\) then for every \(R \ge 2 \rho\)
\[
 \int_{B_R \setminus B_\rho} u^p \ge cR^{N + p \frac{\alpha + \gamma}{1 - p - q}}.
\]
\item If \(q = 1 - \frac{N - \alpha - \gamma}{N}p\) then for every \(R \ge 2 \rho\)
\[
  \int_{B_R \setminus B_\rho} u^p \ge c\Bigl(\log \frac{R}{\rho}\Bigr)^{\frac{1-q}{1-p-q}}.
\]
\end{enumerate}
\end{lemma}
\begin{proof}
From Lemma~\ref{lemmaLiminfExistence-slow}, for \(R \ge 2 \rho\) it holds
\[
 \int_{B_{2R} \setminus B_R} u^p \int_{B_{2R} \setminus B_R} u^{q - 1} \le C R^{2N - \alpha - \gamma}.
\]
On the other hand, since \(q <  1\) by H\"older's inequality we have
\[
 \int_{B_{2R} \setminus B_R} u^{q - 1} \ge \frac{\displaystyle \Bigl(\int_{B_{2R} \setminus B_R} 1 \Bigr)^{1 + \frac{1-q}{p}} }{\displaystyle \Bigl(\int_{B_{2R} \setminus B_R} u^p \Bigr)^\frac{1-q}{p}}.
\]
Since \(2N - \alpha - \gamma \le N(1 + \frac{1-q}{p})\) we obtain
\[
 \int_{B_{2R} \setminus B_R} u^p \le C R^{N - \alpha - \gamma - N \frac{1-q}{p}} \Bigl(\int_{B_{2R} \setminus B_R} u^p \Bigr)^\frac{1-q}{p}.
\]
We deduce that there exists \(c > 0\) such that for every \(R \ge 2 \rho\),
\[
 \int_{B_{2R} \setminus B_R} u^p \ge c R^{N + p \frac{\alpha + \gamma}{1 - p - q}}.
\]
Then the conclusion is immediate when \(q < 1 - \frac{N - \alpha - \gamma}{N}p\),
and follows by summation over dyadic annuli when \(q = 1 - \frac{N - \alpha - \gamma}{N}p\).
\end{proof}

\begin{proposition}
Let \(N\ge 1\), \(\gamma<2\), \(\lambda > 0\), \(0 < \alpha < N\), \(p>0\), \(p + q<1\) and \(\rho > 0\).
Let \(u\ge 0\) be a nontrivial supersolution of \eqref{eqChoquardSlow} in \(\R^N \setminus \Bar{B}_{\rho}\).
\begin{enumerate}
 \item  If \(q < 1 - \frac{N - \alpha - \gamma}{N}p\) then
\begin{equation*}
 \liminf_{\abs{x} \to \infty} \abs{x}^{-\frac{\gamma + \alpha}{1 - q - p}} u (x) > 0.
\end{equation*}
 \item If \(q = 1 - \frac{N - \alpha - \gamma}{N}p\) then
\begin{equation*}
 \liminf_{\abs{x} \to \infty} \abs{x}^{-\frac{\gamma + \alpha}{1 - q - p}}(\log \abs{x})^\frac{-1}{1 - p - q} u (x) > 0.
\end{equation*}
\end{enumerate}
\end{proposition}
\begin{proof}
Writing
\[
  u(x)^{q - 1} \ge \frac{u (x)^{q - 1}}{\displaystyle \int_{B_{2\abs{x}} \setminus B_\rho} u^p}\displaystyle \int_{B_{2\abs{x}} \setminus B_\rho} u^p,
\]
the conclusion follows from Lemmas~\ref{lemmaLimsupExistence-local} and~\ref{lemmaRieszdivergence}.
\end{proof}

\subsection{Optimal decay.}

We complete the proof of Theorem~\ref{Thm-slow-moderate} by constructing explicit supersolutions to \eqref{eqChoquardSlow}.
which show the optimality of nonexistence and decay results of Propositions~\ref{nonSlowDecayq} and~\ref{propositionSlowDecayq}.

\begin{proposition}\label{slow-polynom-pq+optimal}
Let \(N\ge 1\), \(\gamma< 2\), \(\lambda > 0\), \(0 < \alpha < N\), \(p>0\) and \(\rho > 0\).
If
\[
1-\frac{N - \alpha - \gamma}{N}p<q<1,
\]
and \(p + q \ne 1\),
then \eqref{eqChoquardSlow}
admits a radial nontrivial nonnegative supersolution
\(u\in C^\infty (\R^N \setminus\Bar{B}_\rho)\),
such that
\begin{equation*}
 \limsup_{\abs{x} \to \infty} u(x) \abs{x}^{\frac{N - \alpha - \gamma}{1-q}} < \infty.
\end{equation*}
\end{proposition}

Note that the assumption implies that \(\gamma < N - \alpha\).

\begin{proof}[Proof of Proposition~\ref{slow-polynom-pq+optimal}]
Set for \(\nu > 0\) and \(x \in \R^N \setminus \Bar{B}_{\rho}\),
\[
  v_{\nu} (x)=\frac{1}{(\abs{x}^2 + \nu^2)^{\frac{N - \alpha - \gamma}{2(1-q)}}}.
\]
Compute for every \(x \in \R^N \setminus \Bar{B}_{\rho}\),
\begin{multline*}
  - \Delta v_{\nu} (x) + \frac{\lambda^2}{\abs{x}^\gamma} v_{\nu} (x)\\
  = \Bigl( \frac{N - \alpha - \gamma}{(1 - q)(\abs{x}^2 + \nu^2)} \Bigl(N - \Bigl(2 + \frac{N - \alpha - \gamma}{1 - q} \Bigr)\frac{\abs{x}^2}{\abs{x}^2 + \nu^2} \Bigr) + \frac{\lambda^2}{\abs{x}^\gamma} \Bigr) v_\nu (x).
\end{multline*}
Since \(\gamma < 2\), there exists \(R > 0\) such that for every \(\nu > 0\) and \(x \in \R^N \setminus B_R\),
\[
 - \Delta v_{\nu} (x) + \frac{\lambda^2}{\abs{x}^\gamma} v_{\nu} (x) \ge \frac{\lambda^2}{2 \abs{x}^\gamma} v_\nu (x).
\]
If \(\nu > 0\) is sufficiently large, for every \(x \in B_R \setminus \Bar{B}_\rho\),
\[
 - \Delta v_{\nu} (x) + \frac{\lambda^2}{\abs{x}^\gamma} v_{\nu} (x) > 0.
\]
Then there exists \(c > 0\) such that for every \(x \in \R^N \setminus \Bar{B}_\rho\),
\[
 - \Delta v_{\nu} (x) + \frac{\lambda^2}{\abs{x}^\gamma} v_{\nu} (x) \ge \frac{c}{\abs{x}^{\frac{N - \alpha - \gamma}{1 - q} + \gamma}}.
\]

For \(\mu > 0\), set \(u_\mu = \mu v_\nu\).
Since \(p \frac{N - \alpha - \gamma}{1-q} > N\), by Lemma~\ref{lemmaUpperasympt} there exists \(C > 0\) such that for every \(x \in \R^N \setminus \Bar{B}_\rho\),
\begin{equation*}
 (I_\alpha \ast u_\mu^p)(x) \le \frac{C \mu^p}{\abs{x}^{N - \alpha}}.
\end{equation*}
One concludes that \(u\) is a supersolution to \eqref{eqChoquardSlow} in \(\R^N \setminus \Bar{B}_\rho\)
for all sufficiently small \(\mu>0\) if \(p + q>1\), or for all large \(\mu\) if \(p + q<1\).
\end{proof}

Now we complete the proof of Theorem~\ref{Thm-slow-fast} by establishing the optimality
of Propositions~\ref{nonSlowDecayq} and~\ref{propositionSlowDecayq}.

\begin{proposition}
Let \(N\ge 1\), \(\lambda > 0\), \(0 < \alpha < N\), \(p>0\), \(\gamma<-\alpha\) and \(\rho > 0\).
If
\begin{equation*}
q=1-\frac{N - \alpha - \gamma}{N}p,
\end{equation*}
then \eqref{eqChoquardSlow}
admits a radial nontrivial nonnegative supersolution
\(u\in C^\infty (\R^N \setminus\Bar{B}_\rho)\),
such that
\begin{equation*}
 \lim_{\abs{x} \to \infty} u(x) \abs{x}^{\frac{N}{p}}(\log\abs{x})^\frac{1}{1 - q - p} = c>0.
\end{equation*}
\end{proposition}

\begin{proof}
Set for \(\nu > 1\) and \(x \in \R^N \setminus \Bar{B}_\rho\),
\[
 v_\nu (x)= \frac{\big(\log (\abs{x}^2 + \abs{\nu}^2)\big)^\frac{1}{1 - q - p}}{(\abs{x}^2 + \nu^2)^{\frac{N}{2 p}}}.
\]
One has
\begin{multline*}
  - \Delta v_\nu (x)
= \biggl( \frac{N^2}{p} - \frac{2 N}{(1 - p - q)\log (\abs{x}^2 + \nu^2)}\\
+
\Bigl( - \frac{N (N + 2 p)}{p^2} + \frac{4 (N + p)}{p (1 - p - q) \log (\abs{x}^2 + \nu^2)}\\
- \frac{4 (p + q)}{(1 - p - q)^2 \bigl(\log (\abs{x}^2 + \nu^2)\bigr)^2} \Bigr) \frac{\abs{x}^2}{\abs{x}^2 + \nu^2} \biggr)v_\nu (x).
\end{multline*}
One observes that as \(\gamma < 2\), if \(\nu\) is large enough, then there exists \(c > 0\) such that for every \(x \in \R^N \setminus \Bar{B}_\rho\),
\[
 - \Delta v_\nu (x) \ge \frac{c}{\abs{x}^{-\frac{\alpha + \gamma}{1 - p - q} + \gamma}}.
\]
Set now \(u_\mu = \mu v_\nu\).
By Lemma~\ref{lemmaUpperasymptCritical}, there exists \(C > 0\) such that for every \(x \in \R^N \setminus \Bar{B}_\rho\),
\begin{equation*}
 (I_\alpha \ast u^p)(x) \le \frac{C\mu^p\big(1+\log\frac{\abs{x}}{\rho}\big)^{\frac{p}{1 - q - p}+1}}{\abs{x}^{N - \alpha}}.
\end{equation*}
Since \(q=1-\frac{N - \alpha - \gamma}{N}p < 1 - p\), \(u_\mu\) is a supersolution to \eqref{eqChoquardSlow} in \(\R^N \setminus \Bar{B}_\rho\)
for all sufficiently large \(\mu>0\).
\end{proof}

\begin{proposition}
Let \(N\ge 1\), \(\lambda > 0\), \(0 < \alpha < N\), \(p>0\), \(\gamma<-\alpha\) and \(\rho > 0\). If
\begin{equation*}
1+\frac{\gamma}{\alpha}p<q<1-\frac{N - \alpha - \gamma}{N}p,
\end{equation*}
then \eqref{eqChoquardSlow} admits a nontrivial nonnegative supersolution \(u\) in \(\R^N \setminus \Bar{B}_\rho\) such that
\begin{equation*}
 \lim_{\abs{x} \to \infty} u(x) \abs{x}^{-\frac{\alpha + \gamma}{1 - q - p}} = c>0.
\end{equation*}
\end{proposition}

\begin{proof}
Set
\[
u_\mu (x)=\frac{\mu}{(\abs{x}^2 + \nu^2)^{-\frac{\alpha + \gamma}{2 (1 - q - p)}}},
\]
where \(\nu\) is chosen as in the proof of Proposition~\ref{slow-polynom-pq+optimal} so that there exists \(c > 0\) such that for every \(x \in \R^N \setminus \Bar{B}_\rho\),
\[
 - \Delta u_{\mu} (x) + \frac{\lambda^2}{\abs{x}^\gamma} u_{\mu} (x) \ge \frac{c \mu}{\abs{x}^{-\frac{\alpha + \gamma}{1 - q - p} + \gamma}}.
\]

By our assumptions, we have \(\alpha<-\frac{p(\alpha + \gamma)}{1 - q - p}<N\).
Therefore, by Lemma~\ref{lemmaUpperasympt} we obtain for every \(x \in \R^N \setminus \Bar{B}_\rho\),
\begin{equation*}
 (I_\alpha \ast u_\mu^p)(x) \le \frac{C\mu^p}{\abs{x}^{-\frac{p(\alpha + \gamma)}{1 - q - p}-\alpha}}.
\end{equation*}
Since \(p + q < 1\), we conclude that \(u\) a supersolution to \eqref{eqChoquardSlow} in \(\R^N \setminus \Bar{B}_\rho\)
for all sufficiently large \(\mu>0\).
\end{proof}

\subsection{Homogeneous regime $p+q=1$: proof of Theorem~\ref{t-pq1}.}\label{sect-pq1}

We now consider the homogeneous case of equation \eqref{eqChoquardSlow},
that is the equation
\begin{equation}\label{eqChoquardSlow-hom}
-\Delta u+\frac{\lambda^2}{\abs{x}^\gamma} u=(I_\alpha\ast u^p)u^{1-p}.
\end{equation}
In order to study \eqref{eqChoquardSlow-hom}
we will need a modified version of the nonlocal positivity principle of Proposition~\ref{propositionGroundState} which allows a more accurate control of constants in the integral inequality.

\begin{lemma}
\label{lemmaGroundState-symmetric}
Let \(N\ge 1\), \(0<\alpha<N\), \(\lambda>0\), $\gamma<2$, \(p>0\) and \(\rho > 0\).
If \(u \in L^1_\mathrm{loc}(\R^N\setminus\Bar{B}_\rho)\) be a nonnegative supersolution of \eqref{eqChoquardSlow-hom}
in \(\R^N\setminus\Bar{B}_\rho\),
then either \(u = 0\) almost everywhere in \(\R^N\setminus \Bar{B}_{\rho}\), or \(u > 0\) almost everywhere in \(\R^N\setminus\Bar{B}_\rho\)
and for every \(\varphi \in C^\infty_c(\R^N\setminus\Bar{B}_\rho)\), one has
\[
  \int_{\R^N\setminus\Bar{B}_\rho} \abs{\nabla \varphi}^2 +
  \int_{\R^N\setminus\Bar{B}_\rho} \frac{\lambda^2}{\abs{x}^\gamma}\abs{\varphi}^2 \ge
  \int_{\R^N\setminus\Bar{B}_\rho} \int_{\R^N\setminus\Bar{B}_\rho}\varphi(x)\,I_\alpha(x-y)\varphi(y)\,dx\,dy.
\]
\end{lemma}
\begin{proof}
By Proposition~\ref{propositionLocalGroundstate} with \(W = I_\alpha \ast u^p\in L^1_\mathrm{loc} (\Omega)\),
either \(u = 0\) in \(\R^N\setminus \Bar{B}_{\rho}\) or \(u > 0\) almost everywhere in \(\R^N\setminus\Bar{B}_\rho\), \(u^{-p} \in L^1_\mathrm{loc} (\Omega)\)
and for every \(\varphi \in C^\infty_c(\Omega)\)
\begin{equation*}
  \int_{\Omega} \abs{\nabla \varphi}^2 +\frac{\lambda^2}{\abs{x}^\gamma} \varphi^2
  \ge \int_{\Omega}(I_\alpha \ast u^p) u^{-p}\varphi^2.
\end{equation*}
Now by the Cauchy--Schwarz inequality we derive
\begin{multline*}
 \int_{\Omega} \int_{\Omega}\varphi(x)I_\alpha(x-y)\varphi(y)\,dx\,dy \\
\le \Bigl( \int_{\Omega} \int_{\Omega} u(x)^{-p} \varphi(x)^2 I_\alpha(x-y) u(y)^p \,dx\,dy\Bigr)^\frac{1}{2}\\
\shoveright{\Bigl( \int_{\Omega} \int_{\Omega} u(x)^{p} I_\alpha(x-y) u(y)^{-p} \varphi(y)^2\,dx\,dy\Bigr)^\frac{1}{2}}\\
= \int_{\Omega} \int_{\Omega} u(x)^{-p}\varphi(x)^2 I_\alpha(x-y) u(y)^p \,dx\,dy,
\end{multline*}
and the conclusion follows.
\end{proof}

\subsubsection{Case \(\gamma\neq -\alpha\) and \(p + q=1\).}
If $\alpha>-\gamma$ the nonexistence follows from Proposition \ref{nonSlowDecayq},
while for $\alpha<-\gamma$ we can construct a solution outside
a sufficiently large ball.

\begin{proposition}\label{slow-polynom-pq+optimalhomog}
Let \(N\ge 1\), \(\lambda > 0\), \(0 < \alpha < N\), \(p>0\) and \(q = 1 - p\).
If \(\gamma < - \alpha\),
then there exists \(\rho_0 > 0\) such that \eqref{eqChoquardSlow}
admits a radial nontrivial nonnegative supersolution
\(u\in C^\infty (\R^N \setminus\Bar{B}_{\rho_0})\),
such that
\begin{equation*}
 \limsup_{\abs{x} \to \infty} u(x) \abs{x}^{\frac{N - \alpha - \gamma}{1-q}} < \infty.
\end{equation*}
\end{proposition}

\begin{proof}
Define \(u_\rho : \R^N \setminus \Bar{B}_\rho \to \R\) for \(x \in  \R^N \setminus \Bar{B}_\rho\) by
\[
  u_\rho (x) = \frac{1}{\abs{x}^{\frac{N - \alpha - \gamma}{1 - q}}}.
\]
One has for every \(x \in  \R^N \setminus \Bar{B}_\rho\)
\[
 - \Delta u_\rho (x) + \frac{\lambda^2}{\abs{x}^\gamma} u_\rho (x)
  = \frac{N - \alpha - \gamma}{1 - q}\Bigl(N - 2 - \frac{N - \alpha - \gamma}{1 - q}\Bigr)
\frac{1}{\abs{x}^{\frac{N - \alpha - \gamma}{1 - q} + 2}} + \frac{\lambda^2}{\abs{x}^{\frac{N - \alpha - \gamma}{1 - q} + \gamma}}.
\]
Since \(\gamma < 2\), there exists \(\rho_0\) such that if \(\rho \ge \rho_0\),
\[
 - \Delta u_\rho (x) + \frac{\lambda^2}{\abs{x}^\gamma} u_\rho (x) \ge \frac{\lambda^2}{2 \abs{x}^{\frac{N - \alpha - \gamma}{1 - q} + \gamma}}.
\]
By a change of variable, by the assumption \(\gamma + \alpha < 0\)  and by Lemma~\ref{lemmaUpperasympt}, there exists \(C > 0\) such that for every \(x \in \R^N \setminus B_\rho\),
\[
\begin{split}
 (I_\alpha \ast u_\rho^p) (x)
&= \int_{\R^N \setminus \Bar{B}_\rho} I_\alpha (x - y) \frac{1}{\abs{y}^{N - \alpha - \gamma}}\, dy\\
&= \rho^{\gamma + \alpha - (N - \alpha)}\int_{\R^N \setminus \Bar{B}_1} I_\alpha \Bigl(\frac{x}{\rho} - y\Bigr) \frac{1}{\abs{y}^{N - \alpha - \gamma}}\, dy \le \frac{C \rho^{\gamma + \alpha}}{\abs{x}^{N - \alpha}};
\end{split}
\]
hence,
\[
 (I_\alpha \ast u_\rho^p) (x) u (x)^{q} \le \frac{C \rho^{\gamma + \alpha}}{\abs{x}^{\frac{N - \alpha - \gamma}{1 - q} + \gamma}}.
\]
Since \(\gamma + \alpha < 0\), we conclude that \(u_\rho\) is a supersolution
to \eqref{eqChoquardSlow} in \(\R^N \setminus \Bar{B}_\rho\) for all sufficiently large \(\rho \ge \rho_0\).
\end{proof}

The restriction on the radius $\rho_0>0$ is essential.

\begin{proposition}
Let \(N\ge 1\), \(\lambda > 0\), \(0 < \alpha < N\), \(p>0\) and \(q = 1 - p\).
If \(\gamma < - \alpha\),
then there exists \(\rho^* > 0\) such that \eqref{eqChoquardSlow} has no nontrivial nonnegative supersolution in \(\R^N \setminus\Bar{B}_{\rho^*}\),
\end{proposition}
\begin{proof}
Choose \(\varphi \in C^\infty_c (\R^N \setminus B_1)\) and let for \(R > 0\) and \(x \in \R^N\),
\(\varphi_R (x) = \varphi (x/R)\).
One has
\begin{align*}
 \int_{\R^N} (I_\alpha \ast \varphi_R) \varphi_R& =R^{N + \alpha} \int_{\R^N} (I_\alpha \ast \varphi) \varphi,\\
\int_{\R^N} \abs{\nabla \varphi_R}^2 & =R^{N - 2} \int_{\R^N} \abs{\nabla \varphi}^2,\\
\int_{\R^N}  \frac{\abs{\varphi_R (x)}^2}{\abs{x}^\gamma}\,dx
& = R^{N - \gamma} \int_{\R^N}  \frac{\abs{\varphi (x)}^2}{\abs{x}^\gamma}\,dx.
\end{align*}
Since \(- \gamma < \alpha < 2\), this brings a contradiction with the positivity principle of Lemma~\ref{lemmaGroundState-symmetric} if \(R\) is small enough.
\end{proof}

\subsubsection{Case \(\gamma=-\alpha\) and \(p + q=1\).}
We now consider the most delicate case of equation \eqref{eqChoquardSlow} when \(\gamma=-\alpha\), that is the equation
\begin{equation}\label{eqChoquardSlow-1}
-\Delta u+\lambda^2\abs{x}^\alpha u=(I_\alpha\ast u^p)u^{1-p}.
\end{equation}

In order to study our problem, we review relevant inequalities.
The weighted version
of the Hardy--Littlewood--Sobolev inequality of Stein and Weiss \cite{Stein-Weiss},
\begin{equation}\label{HLSw}
\int_{\R^N} \abs{I_{\alpha/2}\ast\varphi}^2 \le \lambda^2 \int_{\R^N} \abs{x}^\alpha \abs{\varphi (x)}^2 \,dx,
\end{equation}
holds if and only if
\[
  \lambda^2 \ge \sigma_\alpha^\ast = 2^{-\alpha}\Bigl(\frac{\Gamma(\frac{N - \alpha}{4})}{\Gamma(\frac{N + \alpha}{4})}\Bigr)^2
\]
see \cite{Herbst}*{Theorem 2.5}.

The constant \(\sigma_\alpha\) is also related to convolution of Riesz kernels.
Indeed, according to the semigroup property of the Riesz kernels \cite{Riesz}*{p.20}, for \(0 < \alpha <\beta< N\) it holds
\[
  \int_{\R^N} I_\alpha (x - y) \frac{1}{\abs{y}^\beta}\,dy = \frac{\sigma_\alpha(\beta)}{\abs{x}^{\beta-\alpha}},
\]
where
\begin{equation*}
\sigma_\alpha(\beta)=\frac{A_{N-\beta+\alpha}}{A_{N-\beta}}=
2^{-\alpha}\frac{\Gamma\Big(\frac{\beta-\alpha}{2}\Big)\Gamma\Big(\frac{N-\beta}{2}\Big)}
{\Gamma\Big(\frac{N-\beta+\alpha}{2}\Big)\Gamma\Big(\frac{\beta}{2}\Big)}.
\end{equation*}
One can verify (see \cite{Frank}*{Lemma 2.1}) that \(\sigma_\alpha:(\alpha,N)\to\R\) is an even function with respect to \(\beta=(N + \alpha)/2\). Moreover,
\[
 \lim_{\beta\to\alpha}\sigma_\alpha(\beta)=\lim_{\beta\to N}\sigma_\alpha(\beta)=+\infty,
\]
\(\sigma_\alpha\) is strictly decreasing on \((\alpha,N + \alpha)/2)\), strictly increasing on \(((N + \alpha)/2,N)\),
and attains its minimum at \(\beta=(N + \alpha)/2\), with
\begin{equation*}
\sigma_\alpha^\ast=\min_{\alpha<\beta<N}\sigma_\alpha(\beta)=\sigma_\alpha\Big(\frac{N + \alpha}{2}\Big).
\end{equation*}

\begin{lemma}
\label{lemmaOptimalSW}
Let \(N \ge 1\), \(0 < \alpha < N\), \(\lambda > 0\) and \(\rho > 0\).
One has for every \(\varphi \in C^\infty_c (\R^N \setminus \Bar{B}_\rho)\)
\[
  \int_{\R^N} (I_\alpha \ast \varphi) \varphi  \le \int_{\R^N} \abs{\nabla \varphi}^2 + \lambda^2 \int_{\R^N} \abs{x}^\alpha \abs{\varphi (x)}^2\,dx
\]
if and only if \(\lambda^2 \ge \sigma^*_\alpha\).
\end{lemma}
\begin{proof}
It is clear by \eqref{HLSw} and the semigroup property of the Riesz potential that
\(\lambda^2 \ge \sigma^*_\alpha\) implies the required inequality.

Now assume that the inequality holds. Choose \(\varphi \in C^\infty_c (\R^N \setminus \{0\})\).
Let \(\varphi_R (x) = \varphi (x/R)\). For \(R\) large enough, \(\supp \varphi_R \subset \R^N \setminus \Bar{B}_\rho\), and by assumption
\[
\begin{split}
R^{N + \alpha} \int_{\R^N} (I_\alpha \ast \varphi) \varphi
&= \int_{\R^N} (I_\alpha \ast \varphi_R) \varphi_R \\
& \le \int_{\R^N} \abs{\nabla \varphi_R}^2 + \lambda^2 \int_{\R^N} \abs{x}^\alpha \abs{\varphi_R (x)}^2\,dx\\
&= R^{N - 2} \int_{\R^N} \abs{\nabla \varphi}^2 + \lambda^2 R^{N + \alpha} \int_{\R^N} \abs{x}^\alpha \abs{\varphi (x)}^2\,dx
\end{split}
\]
Letting now \(R \to \infty\), we deduce that
\[
 \int_{\R^N} (I_\alpha \ast \varphi) \varphi \le \lambda^2\int_{\R^N} \abs{x}^\alpha \abs{\varphi (x)}^2\,dx.
\]
Since \(\varphi \in C^\infty_c (\R^N \setminus \{0\})\) is arbitrary and by the optimality condition in \eqref{HLSw}, we conclude that \(\lambda^2 \ge \sigma_\alpha^*\).
\end{proof}

Using the above modified nonlocal positivity principle we prove the following.

\begin{proposition}
\label{propauie}
Let \(N\ge 2\), \(0<\alpha<N\), \(p>0\) and \(\rho > 0\). If
\[
 0<\lambda^2 < \sigma_\alpha^\ast,
\]
and \(u \ge 0\) is a supersolution of \eqref{eqChoquardSlow-1} in \(\R^N \setminus \Bar{B}_{\rho}\), then \(u = 0\) in \(\R^N\setminus \Bar{B}_{\rho}\).
\end{proposition}

\begin{proof}
This follows from Lemma~\ref{lemmaGroundState-symmetric} and Lemma~\ref{lemmaOptimalSW}.
\end{proof}

Using semigroup property of the Riesz kernels it easy to construct explicit supersolutions to \eqref{eqChoquardSlow-1}.

\begin{proposition}
\label{proptsrn}
Let \(N\ge 1\), \(0<\alpha<N\) and \(p>0\). If
\[
\lambda^2 > \sigma_\alpha^\ast,
\]
then there exists \(\rho > 0\) such that \eqref{eqChoquardSlow-1} admits a radial nontrivial nonnegative supersolution
\(u\in C^\infty (\R^N \setminus\Bar{B}_\rho)\) satisfying
\begin{equation*}
 \limsup_{\abs{x} \to \infty} u(x) \abs{x}^{\frac{N}{p}} < +\infty.
\end{equation*}
\end{proposition}

\begin{proof}
Set for \(x \in \R^N \setminus \Bar{B}_\rho\),
\[
  u_\rho(x)=\frac{1}{\abs{x}^\frac{N + \alpha}{2 p}}.
\]
Then we compute
\[
\begin{split}
-\Delta u_\rho (x)+\frac{\lambda^2}{\abs{x}^{\alpha}}u_\rho(x)
&=
\frac{N + \alpha}{2 p}\Bigl(N - 2-\frac{N + \alpha}{2 p}\Bigr)\abs{x}^{-\frac{N + \alpha}{2 p}-2}+\lambda^2\abs{x}^{\alpha-\frac{N + \alpha}{2 p}}\\
&=\lambda^2 \abs{x}^{\alpha-\frac{N + \alpha}{2 p}}(1 + o (1)),
\end{split}
\]
as \(\abs{x} \to \infty\).
On the other hand,
\[
-(I_\alpha\ast u_\rho^p)(x)u_\rho^{1 - p}(x) \le \sigma_\alpha\Bigl(\frac{N + \alpha}{2}\Bigr) \abs{x}^{\alpha-\frac{N + \alpha}{2 p}}
= \sigma_\alpha^* \abs{x}^{\alpha-\frac{N + \alpha}{2 p}}
\]
We conclude that \(u_\rho\) is a supersolution of \eqref{eqChoquardSlow-1} if \(\rho > 0\) is large enough.
\end{proof}

We do not make any claim about the existence or nonexistence of nontrivial nonnegative supersolutions
of \eqref{eqChoquardSlow-1} at the threshold value \(\lambda=\sqrt{\sigma_\alpha^\ast}\).
The above proof shows that supersolutions exist if \(\alpha < N - 4\) and \(\frac{N + \alpha}{2 (N - 2)} < p < 1\).

\begin{proof}[Proof of Theorem~\ref{t-pq1}]
If \(\alpha\neq -\gamma\) then Theorem~\ref{t-pq1} is a consequence of
Propositions~\ref{nonSlowDecayq},~\ref{slow-polynom-pq+optimalhomog} and the decay estimate of Propositions~\ref{propositionSlowDecayq}.
When \(\alpha= -\gamma\) the conclusion of Theorem~\ref{t-pq1} follows our results for equation \eqref{eqChoquardSlow-1} (Propositions~\ref{propauie} and \ref{proptsrn})
and Proposition~\ref{propositionSlowDecayq}.
\end{proof}

\appendix
\section{Riesz potential estimates.}
Here we collect some estimates of the Riesz potentials which were extensively used in the main part of the paper.
Most of the estimates are standard,
however we sketch some of the proofs for the readers convenience.

\begin{lemma}\label{lemmaUpperasympt}
Let \(v\in L^1_{\mathrm{loc}}(\R^N)\), \(\alpha \in (0, N)\) and \(\beta > \alpha\).
If
\[
 \limsup_{\abs{x} \to \infty} v(x) \abs{x}^\beta <\infty,
\]
then
\begin{align*}
\limsup_{\abs{x} \to \infty} (I_\alpha \ast v)(x)\abs{x}^{\beta-\alpha} &<\infty& & \text{if \(\alpha<\beta<N\)},\\
\limsup_{\abs{x} \to \infty} (I_\alpha \ast v)(x) \frac{\abs{x}^{N - \alpha}}{\log \abs{x}} &<\infty&
&\text{if \(\beta=N\)},\bigskip\\
\limsup_{\abs{x} \to \infty} (I_\alpha \ast v)(x)\abs{x}^{N - \alpha} &<\infty& & \text{if \(\beta>N\)}.
\end{align*}
\end{lemma}

\begin{proof}
Without loss of generality, we can assume that \(v(x) \le \frac{1}{\abs{x}^\beta}\) if \(\abs{x} \ge 1\).
Then for  every \(x \in \R^N\setminus B_2\) it holds
\[
  (I_\alpha \ast v)(x) \le A_\alpha \int_{\R^N\setminus B_1} \frac{1}{\abs{y}^\beta} \frac{1}{\abs{x-y}^{N - \alpha}}\,dy+ \frac{2^{N - \alpha} A_\alpha}{\abs{x}^{N - \alpha}}\int_{B_1} \abs{v}.
\]
Observe that if \(\abs{x-y} \le \frac{\abs{x}}{2}\), then \(\abs{y} \ge \frac{\abs{x}}{2}\).
Hence, since \(\alpha < N\),
\[
 \int_{B_{\abs{x}/2}(x)}  \frac{1}{\abs{y}^\beta} \frac{1}{\abs{x-y}^{N - \alpha}}\,dy
\le 2^\beta \int_{B_{\abs{x}/2}(x)}  \frac{1}{\abs{x}^\beta} \frac{1}{\abs{x-y}^{N - \alpha}}\,dy=\frac{C}{\abs{x}^{\beta-\alpha}}.
\]
Next, if \(\abs{y} \ge 2\abs{x}\), then \(\abs{x-y} \ge \frac{\abs{y}}{2}\).
Therefore, since \(\beta > \alpha\),
\[
 \int_{\R^N \setminus B_{2\abs{x}}}  \frac{1}{\abs{y}^\beta} \frac{1}{\abs{x-y}^{N - \alpha}}\,dy
\le 2^{N - \alpha} \int_{\R^N \setminus B_{2\abs{x}}}   \frac{1}{\abs{y}^{\beta+N - \alpha}}\,dy=\frac{C'}{\abs{x}^{\beta-\alpha}}.
\]
Finally, we obtain
\[
\begin{split}
 \int_{B_{2\abs{x}} \setminus (B_1 \cup B_{\abs{x}/2}(x))} \frac{1}{\abs{y}^\beta} \frac{1}{\abs{x-y}^{N - \alpha}}\,dy &\le \int_{B_{2\abs{x}} \setminus (B_1 \cup B_{\abs{x}/2}(x))} \frac{1}{\abs{y}^\beta} \frac{2^{N - \alpha}}{\abs{x}^{N - \alpha}}\,dy\\
&\le \frac{2^{N - \alpha}}{\abs{x}^{N - \alpha}} \int_{B_{2\abs{x}} \setminus B_1 } \frac{1}{\abs{y}^{\beta}}\,dy,
\end{split}
\]
where
\[
\int_{B_{2\abs{x}} \setminus B_1 } \frac{1}{\abs{y}^\beta}\,dy \le
\begin{cases}
C''\abs{x}^{N-\beta}& \text{if \(\beta<N\)},\\
C''\log (2\abs{x})&\text{if \(\beta=N\)},\\
C''& \text{if \(\beta > N\)},\\
\end{cases}
\]
so the assertion follows.
\end{proof}

\begin{lemma}\label{lemmaUpperasymptCritical}
Let \(v\in L^1_{\mathrm{loc}}(\R^N)\), \(\alpha \in (0, N)\) and \(\sigma\in\R\).
If
\[
 \limsup_{\abs{x} \to \infty} v(x) \frac{\abs{x}^N}{(\log \abs{x})^\sigma} <\infty.
\]
then
\begin{align*}
\limsup_{\abs{x} \to \infty} (I_\alpha \ast v)(x) \abs{x}^{N - \alpha} &<\infty & &\text{if \(\sigma < -1\)},\\
\limsup_{\abs{x} \to \infty} (I_\alpha \ast v)(x) \frac{\abs{x}^{N - \alpha}}{(\log (\log \abs{x}))} &<\infty & &\text{if \(\sigma = -1\)},\\
\limsup_{\abs{x} \to \infty} (I_\alpha \ast v)(x) \frac{\abs{x}^{N - \alpha}}{(\log \abs{x})^{\sigma+1}} &<\infty & &\text{if \(\sigma > -1\)}.
\end{align*}
\end{lemma}

\begin{proof}
Without loss of generality, we can assume that \(v(x) \le \frac{\abs{\log \abs{x}}^\sigma}{\abs{x}^N}\)
for \(\abs{x} \ge 2\).
Then for \(x \in \R^N\setminus B_4\) it holds
\[
  (I_\alpha \ast v)(x) \le C \int_{\R^N\setminus B_2} \frac{\abs{\log \abs{y}}^\sigma}{\abs{y}^N} \frac{1}{\abs{x-y}^{N - \alpha}}\,dy+ \frac{C}{\abs{x}^{N - \alpha}}\int_{B_2} \abs{v}.
\]
Observe also that if \(\abs{x-y} \le \frac{\abs{x}}{2}\) then \(\abs{y} \ge \frac{\abs{x}}{2}\). Since \(r \mapsto (\log r)^\sigma/r^N\) is nonincreasing for \(r \ge e^\frac{\sigma}{N}\), if \(\abs{x} \ge  2 e^\frac{\sigma}{N}\)
\[
\begin{split}
 \int_{B_{\abs{x}/2}(x)}  \frac{\abs{\log \abs{y}}^\sigma}{\abs{y}^N} \frac{1}{\abs{x-y}^{N - \alpha}}\,dy
& \le 2^N \int_{B_{\abs{x}/2}(x)}  \frac{\abs{\log{\frac{\abs{x}}{2}}}^\sigma}{\abs{x}^N} \frac{1}{\abs{x-y}^{N - \alpha}}\,dy\\
&=\frac{C\abs{\log{\frac{\abs{x}}{2}}}^\sigma}{\abs{x}^{N-\alpha}}.
\end{split}
\]
Next, if \(\abs{y} \ge 2\abs{x}\) then \(\abs{x-y} \ge \frac{\abs{y}}{2}\). Hence, since \(N > \alpha\),
\[
 \int_{\R^N \setminus B_{2\abs{x}}}  \frac{\abs{\log \abs{y}}^\sigma}{\abs{y}^N} \frac{1}{\abs{x-y}^{N - \alpha}}\,dy
\le C \int_{\R^N \setminus B_{2\abs{x}}}   \frac{\abs{\log \abs{y}}^\sigma}{\abs{y}^{2 N - \alpha}}\,dy\le  C'\frac{\abs{\log \abs{x}}^\sigma}{\abs{x}^{N-\alpha}}.
\]
Finally, we obtain
\[
\begin{split}
 \int_{B_{2\abs{x}} \setminus (B_2 \cup B_{\abs{x}/2}(x))}\!\!\! \frac{\abs{\log \abs{y}}^\sigma}{\abs{y}^N} \frac{1}{\abs{x-y}^{N - \alpha}}\,dy
&\le \int_{B_{2\abs{x}} \setminus (B_2 \cup B_{\abs{x}/2}(x))} \!\!\! \frac{\abs{\log \abs{y}}^\sigma}{\abs{y}^N} \frac{2^{N - \alpha}}{\abs{x}^{N - \alpha}} \,dy\\
&\le \frac{2^{N - \alpha}}{\abs{x}^{N - \alpha}} \int_{B_{2\abs{x}} \setminus B_2 }  \frac{\abs{\log \abs{y}}^\sigma}{\abs{y}^{N}}\,dy,
\end{split}
\]
where
\[
\int_{B_{2\abs{x}} \setminus B_1 } \frac{1}{\abs{y}^N}=
\begin{cases}
C''& \text{if \(\sigma < - 1\)},\\
C''\log (\log (2\abs{x}))&\text{if \(\sigma = - 1\)},\\
C''(\log \abs{x})^{\sigma + 1} & \text{if \(\sigma > - 1\)},\\
\end{cases}
\]
which completes the proof.
\end{proof}

\section{Tools for distributional solutions.}\label{App-B}

\subsection{Truncation of supersolutions.}
The following lemma provides a powerful tool of approximation of distributional supersolutions
by weak supersolutions. It is essentially a reformulation of two truncation results
by H.\thinspace Brezis and A.\thinspace Ponce \cite{BrezisPonce2003}*{Lemma 1 and 2}.

\begin{lemma}
\label{lemmaKato}
Let \(\Omega\subseteq\R^N\) be an open connected set, \(u \in L^1_{\mathrm{loc}}(\Omega)\) and
\(V : \Omega \to \R\) be measurable
If \(V u \in L^1_{\mathrm{loc}} (\Omega)\), \(u \ge 0\) and
\[
 -\Delta u + V u \ge 0\quad\text{in}\quad\Omega
\]
in the sense of distributions, then for every \(k \in \R\) one has
\[
 T_k(u) \in H^1_{\mathrm{loc}} (\Omega)
\]
and
\[
 -\Delta T_k(u) + V T_k(u) \ge 0\quad\text{in}\quad\Omega
\]
in the weak sense, where
\(
  T_k (s) = \min (s, k).
\)
\end{lemma}
The first part of the lemma is proved by taking
\(
 \check \eta_\delta \ast (\varphi (k - T_k (\eta_\delta \ast u)))
\)
as a test function in the inequality, for a suitable family of mollifiers \(\eta_\delta\), see \cite{BrezisPonce2003}*{Lemma 1}.
The second part is a consequence of a variant on Kato's inequality \cite{BrezisPonce2003}*{Lemma 2}.

\subsection{Minimal positive solutions at infinity.}

Consider the linear Schr\"odinger equation
\begin{equation}\label{loc-V}
-\Delta u+Vu=0\quad
\text{in \(\R^N\setminus \Bar{B}_\rho\)},
\end{equation}
where \(\rho> 0\) and \(V\in L^\infty_\mathrm{loc}(\R^N\setminus \Bar{B}_\rho)\).

\begin{definition}\label{App-minimal}
We say that \(H\in C^{1}(\R^N\setminus \Bar{B}_\rho)\) is a \emph{minimal positive solution at infinity} of \eqref{loc-V}
if \(H\) is a weak positive solution of \eqref{loc-V} and there exists a weak positive supersolution \(U\in H^1(\R^N\setminus \Bar{B}_\rho)\)
of \eqref{loc-V} such that
\begin{equation}\label{minimal-def}
\liminf_{\abs{x}\to\infty}\frac{U(x)}{H(x)}=+\infty.
\end{equation}
\end{definition}

For example, the fundamental solution of \eqref{loc-V} in \(\R^N\) (if it exists)
is a minimal positive solution of \eqref{loc-V} at infinity.
A minimal positive solution of \eqref{loc-V} might however not decay at infinity
may not decay to zero at infinity. For instance, constants are minimal positive solutions
at infinity for \(-\Delta\) in \(\R^2\setminus \Bar{B}_\rho\).

\begin{proposition}\label{p-minimal}
Assume that \(V\in L^\infty_\mathrm{loc}(\R^N \setminus \Bar{B}_{\rho})\) is nonnegative.
Let \(H\) be a minimal positive solution at infinity of \eqref{loc-V}.
If \(u \in L^1_\mathrm{loc}(\R^N \setminus \Bar{B}_{\rho})\) satisfies
\[
 -\Delta u + V u \ge 0\quad\text{in}\quad \R^N\setminus \Bar{B}_\rho
\]
in the sense of distributions and
\[
 \inf_{B_{3 \rho} \setminus \Bar{B}_{\rho}} u > 0,
\]
then
\[
 u \ge  \frac{\inf_{B_{3 \rho} \setminus \Bar{B}_{\rho}} u}{\sup_{\partial B_{2\rho}} H} H\quad\text{in \(\R^N \setminus \Bar{B}_{2 \rho}\)}.\]
\end{proposition}

\begin{proof}
Define \(v = \min (u, \inf_{B_{3 \rho} \setminus \Bar{B}_{\rho}}u)\).
By Lemma~\ref{lemmaKato} we conclude that \(v \in H^1_\mathrm{loc} (\R^N \setminus \Bar{B}_\rho)\) and
\[
 -\Delta v + V v \ge 0\quad\text{in}\quad \R^N\setminus \Bar{B}_\rho.
\]
Moreover, \(v = \inf_{B_{3 \rho} \setminus \Bar{B}_{\rho}}\) on \(B_{3 \rho} \setminus \Bar{B}_{\rho}\).

Let \(U>0\) be a supersolution of \eqref{loc-V} given by \eqref{minimal-def}.
By the comparison principle for weak sub and supersolutions (cf. \cite{Agmon}*{Theorem 2.7}),
we obtain
\[
  v \ge \frac{\inf_{B_{3 \rho} \setminus \Bar{B}_{\rho}} u}{\sup_{\partial B_{2\rho}} H} (H-\varepsilon U)\quad\text{in \(\R^N \setminus \Bar{B}_{2 \rho}\),}
\]
for every sufficiently small \(\varepsilon>0\).
Since \(u \ge v\), the assertion follows.
\end{proof}

\subsection{Weak Harnack inequality.}
The weak Harnack inequality is usually formulated in the literature for classical or weak supersolutions of elliptic equations \cite{GilbargTrudinger1983}*{Theorem 8.18}.
The proposition below shows that the result remains valid for distributional supersolutions
(see \cite{Lieb-Loss}*{Theorem 9.10} for the case \(p=1\)).

\begin{proposition}\label{P-XXX}
Let \(N \ge 0\), \(\lambda \ge 0\) and \(p \in (-\infty, \tfrac{N}{N - 2})\) and \(p \ne 0\).
There exists \(C > 0\) such that for every \(u \in L^1 (B_{3\rho})\), if \(u\ge 0\),
\[
  - \Delta u + \lambda u \ge 0 \quad\text{in}\quad B_{3\rho}
\]
and \(u^p \in L^1 (B_{2\rho})\), then
\[
  \inf_{B_\rho} u \ge C\,\Bigl(\frac{1}{\rho^N} \int_{B_{2\rho}} u^p \Bigr)^\frac{1}{p}.
\]
\end{proposition}
\begin{proof}
Define \(
 u_k = T_k (u)
\),
where  \(T_k (s) = \min (s, k)\).
By Lemma~\ref{lemmaKato}, \(u_k \in H^1_\mathrm{loc} (\R^N \setminus \Bar{B}_{3\rho})\) and
\[
 -\Delta v_k + \lambda v_k \ge 0\quad\text{in}\quad B_{3\rho}.
\]
Then by the weak Harnack inequality for weak supersolutions (cf. \cite{GilbargTrudinger1983}*{Theorem 8.18}),
\[
  \inf_{B_{\rho}} u_k \ge C\Bigl(\frac{1}{\rho^N} \int_{B_{2\rho}} (u_k)^p \Bigr)^\frac{1}{p}.
\]
We conclude by Lebesgue's monotone convergence theorem if \(p>0\), or by Fatou's lemma if \(p<0\).
\end{proof}

\section*{Acknowledgements}

VM is grateful to Wolfgang Reichel and to Marcello Lucia and Prashanth Srinivasan for stimulating discussions
on the Agmon--Allegretto--Piepenbrink positivity principle.

\end{document}